\documentclass[12pt]{amsart}
\usepackage[utf8]{inputenc}
\usepackage{amsmath,amssymb,amsthm,colonequals,mathrsfs,mathtools}
\usepackage{array,enumitem,yfonts}
\usepackage{comment}
\usepackage[alphabetic]{amsrefs}
\usepackage[all,cmtip]{xy}
\usepackage[colorlinks,anchorcolor=blue,citecolor=blue,linkcolor=blue,urlcolor=blue,bookmarksdepth=2]{hyperref}
\allowdisplaybreaks

\usepackage{comment}

\usepackage[margin=1in]{geometry}

\usepackage{tikz}
\usetikzlibrary{cd,arrows}
\tikzset{>=stealth}
\tikzcdset{arrow style=tikz}
\tikzset{link/.style={column sep=1.8cm,row sep=0.16cm}}

\AtBeginDocument{%
	\def\MR#1{}
}
\usepackage{fancyhdr}
\fancypagestyle{titlepage}
{
	\fancyhf{}

\fancyfoot[l]{
\href{https://mathscinet.ams.org/mathscinet/msc/msc2020.html}
	}
    }
\newcommand{\bZ}{\mathbb{Z}}    
\newcommand{\bQ}{\mathbb{Q}}    
\newcommand{\bR}{\mathbb{R}}    
\newcommand{\bC}{\mathbb{C}}    
\newcommand{\bH}{\mathbb{H}}    
\newcommand{\bB}{\mathbb{B}}

\newcommand{\cO}{\mathcal{O}}   
\newcommand{\bP}{\mathbb{P}}    
\newcommand{\abs}[1]{\left\lvert#1\right\rvert}
\newcommand{\isoarrow}{\stackrel{\sim}{\longrightarrow}}
\newcommand{\transpose}[1]{{#1}^\top}

\newcommand{\Proj}{\operatorname{Proj}}
\newcommand{\Pic}{\operatorname{Pic}} 
\newcommand{\SU}{\operatorname{SU}}
\newcommand{\SO}{\operatorname{SO}}
\newcommand{\Sp}{\operatorname{Sp}}
\newcommand{\SL}{\operatorname{SL}}
\def\div{\mathop{\mathrm{div}}\nolimits}
\def\vol{\mathop{\mathrm{vol}}\nolimits}
\def\volHM{\mathop{\mathrm{vol}_{\mathrm{HM}}}\nolimits}

\def\O{\mathrm{O}}

\def\U{{\mathrm{U}}}

\newcommand{\defeq}{\vcentcolon=} 

\DeclareMathOperator{\id}{id}
\DeclareMathOperator{\Hom}{Hom}
\DeclareMathOperator{\fin}{f}
\DeclareMathOperator{\scn}{sc}
\DeclareMathOperator{\np}{np}
\DeclareMathOperator{\Ind}{Ind}

\DeclareMathOperator{\qs}{qs}
\DeclareMathOperator{\rel}{rel}
\DeclareMathOperator{\res}{res}

\newtheorem{thm}{Theorem}[subsection]
\newtheorem{prop}[thm]{Proposition}
\newtheorem{lemma}[thm]{Lemma}
\newtheorem{cor}[thm]{Corollary}
\numberwithin{equation}{section}

\theoremstyle{definition}
\newtheorem{conj}[thm]{Conjecture}
\newtheorem{defn}[thm]{Definition}

\newtheorem{rmk}[thm]{Remark}
\newtheorem{asm}[thm]{Assumption}

\begin{document}

\title[Finiteness of Free Algebras of Modular Forms on Unitary Groups]{Finiteness of Free Algebras of Modular Forms on Unitary Groups}
\author{Yota Maeda$^{1, 2}$ \and Kazuma Ohara$^{3}$}
\email{y.maeda.math@gmail.com, kazuma@mpim-bonn.mpg.de}

\def\l@subsection{\@tocline{2}{0pt}{2.3pc}{5pc}{}}

\maketitle
\thispagestyle{titlepage}
\vspace{-1em}
\begin{center}
  \begin{minipage}{0.9\textwidth}
    \centering
    {\small
    $^{1}$ Fachbereich Mathematik, Technische Universität Darmstadt, Germany\\
    $^{2}$ Mathematical Institute, Tohoku University, Japan\\
    $^{3}$ Max Planck Institute for Mathematics, Germany
    }
  \end{minipage}
\end{center}

\vspace{1em}

\begin{abstract}
Classical results on the classification of reflections in an arithmetic subgroup $\Gamma$ imply that if the graded algebra of modular forms $M_*(\Gamma)$ is freely generated, then $\Gamma$ must be an arithmetic subgroup of either the orthogonal group  $\operatorname{O}^+(2,n)$ or the unitary group $\operatorname{U}(1,n)$.
Vinberg and Schwarzman showed that in the orthogonal case, if $n>10$, then it is never free.
In this paper, we investigate the remaining unitary case and prove that, up to scaling, there are only finitely many isometry classes of Hermitian lattices of signature $(1, n)$ with $n > 2$ over imaginary quadratic fields with odd discriminant that admit a free algebra of modular forms.
In particular, when $n>99$ (except over $\mathbb{Q}(\sqrt{-3})$, where we require $n > 154$), the graded algebra $M_*(\Gamma)$ is never free for any arithmetic subgroup $\Gamma< \operatorname{U}(1,n)$, thereby partially confirming a conjecture by Wang and Williams.
As a byproduct, we also establish a finiteness result for reflective modular forms.
In the course of this proof, we derive a formula for the covolume of an arithmetic subgroup of a special unitary group, presented as the stabiliser of a Hermitian lattice, which generalises Prasad’s volume formula for principal arithmetic subgroups in the case of special unitary groups.
\end{abstract}

\setcounter{tocdepth}{1}
\tableofcontents

\section{Introduction}

An arithmetic variety $\Gamma\backslash\mathcal{D}$ is defined as the quotient of a Hermitian symmetric domain $\mathcal{D}$ by an arithmetic subgroup $\Gamma$.
The geometric (or specifically birational) classification of arithmetic varieties is a cornerstone problem at the intersection of algebraic geometry and number theory \cites{tai1982kodaira,freitag1983siegelsche,mumford2006kodaira,gritsenko2006hirzebruch,ma2018kodaira}.
A first step toward such a classification is to understand their \emph{minimal} algebraic compactification, that is, the Baily-Borel compactification $\overline{\Gamma\backslash\mathcal{D}}$ \cite{baily1966compactification}.
A hallmark of $\overline{\Gamma\backslash\mathcal{D}}$ is that it is always a normal, projective variety but admits singularities, which tend to worsen near the boundary and are often more severe than quotient singularities, making systematic research difficult.
Since singularities heavily influence the geometry of $\overline{\Gamma\backslash \mathcal{D}}$, it is natural to seek arithmetic groups for which every singularity remains especially mild.
 Among all quotient singularities, those coming from cyclic groups form the simplest, and crucially, the most tractable subclass.  The prototypical projective variety that has only cyclic quotient singularities is the weighted projective space $\bP(k_0,\dots,k_n)$ \cite{Dolgachev1982weighted}.
Returning to our situation, it is known that $\overline{\Gamma\backslash\mathcal{D}}$ is isomorphic to $\Proj M_* (\Gamma)$ where $M_*(\Gamma)$ denotes the graded algebra of modular forms for $\Gamma$.
Therefore, $\overline{\Gamma\backslash\mathcal{D}}\cong \bP(k_0,\cdots,k_n)$ holds most typically when $M_*(\Gamma)$ is freely generated by modular forms of weight $k_0,\cdots k_n$.

As discussed in \cite{Wang2021classification}, the space $M_*(\Gamma)$ is freely generated only if $\Gamma$ is generated by reflections  \cite{vinberg1989invariant}, inspired by Shephard-Todd-Chevalley theorem \cites{shephard1954finite,chevalley1955invariants}.
Combined with the classification of reflections \cite{Meschiari1972reflections}, this forces that $\Gamma$ is an arithmetic subgroup of the orthogonal group $\O^+(2,n)$ or the unitary group $\U(1,n)$.
Building on Igusa’s celebrated work \cites{igusa1962siegel,igusa1964siegel}, numerous studies on orthogonal groups have investigated the conditions under which a graded algebra is freely generated, along with concrete examples \cites{aoki2005simple,dern2003graded,dern2004graded,hashimoto2022ring,vinberg2010some,vinberg2013algebra,vinberg2018some,wang2020some,wang2021some}.
Among the results in this line of research, Vinberg and Schwarzman \cite{vinberg2017criterion} proved that the space $M_*(\Gamma)$ is not freely generated for any arithmetic subgroup $\Gamma< \O^+(2,n)$ with $n>10$, which reduces such a classification problem of $M_*(\Gamma)$ to the low-rank cases.
Building on this work, Wang classified quadratic forms whose associated graded algebra is free \cite{Wang2021classification}.

We now turn to the remaining case, namely the unitary case.
There are fewer studies on $M_*(\Gamma)$ than on orthogonal groups \cites{allcock2002cubic,freitag2002graded,freitag2011modular,freitag2019vector,resnikoff1978structure,shiga1988representation,tai1982structure,wang2023graded,williams2021two}.
In this line of research, and based on their work  \cites{Wang2021classification}, Wang and Williams \cite{WW2021} constructed explicit examples in which $M_*(\Gamma)$ is a free algebra.
In that work, they conjectured that $M_*(\Gamma)$ is rarely free for $\Gamma<\U(1,n)$.
This is a natural question in view of the preceding work.

In this paper, for any arithmetic subgroup $\Gamma< \U(1,n)$ associated with an imaginary quadratic field of odd discriminant $- D$, we prove that $M_*(\Gamma)$  is never free when $n>99$ for $D > 3$, and when $D = 3$, the result holds for $n > 154$.
Furthermore, we also prove that, up to scaling, there are only finitely many isometry classes of Hermitian lattices admitting a free algebra of modular forms.
These results provide a partial answer to the conjecture of Wang and Williams.
To formulate the conjecture and state our main results, we begin by introducing the relevant notion of unitary groups.

Let $E$ be an imaginary quadratic field with odd discriminant $-D$, and $(L, \langle \phantom{x},\phantom{x} \rangle)$ be a Hermitian lattice over $\cO_E$ of signature $(1,n)$ for $n>2$.
Let 
    \[\bB^n\defeq\{[v]\in \bP(L\otimes_{\cO_E}\bC)\mid \langle v,v\rangle>0 \}\]
be the $n$-dimensional complex ball acted on by the unitary group $\U(L)$.
For an arithmetic subgroup $\Gamma<\U(L)$, let 
\[X_{\Gamma}\defeq \Gamma\backslash\bB^n\] 
be the associated \emph{ball quotient}, a quasi-projective variety over $\bC$ of dimension $n$.
It can be realised as a moduli space via period maps, as in the cases of cubic surfaces \cites{allcock2000complex}, cubic threefolds \cites{allcock2011moduli,looijenga2007period}, and the Deligne–Mostow varieties \cites{deligne1986monodromy,mostow1986generalized}.
These varieties arise as period domains in the context of geometric invariant theory, and classical invariant theory shows, for example, that the moduli space of cubic surfaces is isomorphic to the weighted projective space $\bP(1,2,3,4,5)$ \cite{dolgachev2005complex}.
Moreover, recent research shows that the associated graded algebra in this case is free \cite{WW2021}*{Theorem~5.19}.
A detailed analysis of the geometric structure of $X_{\Gamma}$ is essential for understanding the modular interpretation of moduli spaces.
We denote by $\overline{X_{\Gamma}}$ the Baily-Borel compactification, which is isomorphic to $\Proj M_*(\Gamma)$.
In contrast to the above example, we consider the following conjecture.
\begin{conj}[{\cite{WW2021}*{Conjecture 8.1}}]
\label{conj:main}
The graded algebra $M_*(\Gamma)$ is \emph{not} a free algebra for $n>5$.
\end{conj}
Our results (Theorems~\ref{thm:not_wps} and \ref{thm:finiteness}) partially resolve Conjecture \ref{conj:main} and further show that the number of possible exceptions to the conjecture is finite.
In order to analyse $M_*(\Gamma)$, we estimate the related volume $\volHM(\SU(L))$, which is called the \emph{Hirzebruch-Mumford volume} and closely related to the covolume of the arithmetic subgroup $\SU(L)$ in $\SU(L \otimes_{\mathbb{Z}} \mathbb{R})$, through Hirzebruch's proportionality principle \cite{mumford1977hirzebruch}.
The covolume of arithmetic subgroups of algebraic groups arises in various contexts, including the special values of zeta functions \cites{Siegel1935,siegel1945some,kottwitz1988tamagawa}, dimension formulas for automorphic forms \cites{savin1989limit,wakatsuki2018dimensions}, and the computation of birational invariants \cites{gritsenko2006hirzebruch,ma2018kodaira}.

One of the most promising tools to compute the covolume of an arithmetic subgroup is Prasad's volume formula \cite{Pra89}. It enables us to compute the covolume of \emph{principal} arithmetic subgroups of absolutely simple, simply connected algebraic groups.
Here, an arithmetic subgroup is called principal if its localisation at any finite place is a parahoric subgroup.
There are several notable applications of Prasad's volume formula to the geometry of arithmetic varieties, including the classification of fake projective planes \cite{prasad2005fake}, fake compact Hermitian spaces  \cites{prasad2009arithmetic,prasad2012nonexistence,prasad2023arithmeticfakecompacthermitian}, and the Kodaira dimension of ball quotients \cites{Mae24}.
It has also been applied to the problem of finding arithmetic subgroups with minimal covolumes; $\Sp_{2n}$ \cite{dvzambic2024siegel}, $\SL_n$ \cite{Thilmany2019lattices}, $\SO^+(1,n)$ \cite{belolipetsky2004volumes}, and $\mathrm{PU}(n,1)$ \cite{emery2014covolumes}.
In studying the graded algebras of modular forms, it is further necessary to compute the covolumes not only of principal arithmetic subgroups but also of all arithmetic subgroups of the form $\SU(L)$.
To this end, we generalise his formula applicable to non-principal arithmetic subgroups presented as stabilisers of Hermitian lattices in the course of the proof of the main theorem.
As noted earlier, Prasad’s volume formula has numerous important applications. Therefore, we expect that our extension of this formula will be of independent interest and has potential application in algebraic geometry and number theory.

\section{Main results}
In this section, we give a brief summary of the main results of this paper.

\subsection{A formula for the covolumes of arithmetic subgroups of special unitary groups}
The first main result of this paper is an explicit formula for the covolume of an arithmetic subgroup presented as the stabiliser of a Hermitian lattice.
Although our result applies to slightly more general settings, we focus on the case of special unitary groups over $\bQ$ here for simplicity.
For a more general claim, see Theorem~\ref{thm:generalization of Prasad's volume formula}.

Let $E$ be an imaginary quadratic field with odd discriminant $-D$.
We write $\cO_{E}$ for the ring of integers of $E$.
Let $V_{\fin}$ be the set of finite places of $\bQ$.
For $v \in V_{\fin}$, we write $\cO_{E_{v}} \defeq \cO_{E} \otimes_{\bZ} \bZ_{v}$, and let $\varpi_{E_{v}}$ be a uniformiser of $\bZ_{v}$ if $v$ splits in $E$, and a uniformiser of $\cO_{E_{v}}$ otherwise.
We also write $q_{E_{v}} \defeq \abs{\cO_{E_{v}}/(\varpi_{E_{v}})}$.

Let $(L, \langle \phantom{x},\phantom{x} \rangle)$ be a Hermitian lattice over $\cO_E$ of rank $n+1$.
For each $v \in V_{\fin}$, we define an $\cO_{E_{v}}$-lattice $L_{v}$ as $L_{v} = L \otimes_{\bZ} \bZ_{v}$.
We denote by $\SU(L)$ (resp.\ $\SU(L_{v})$) the special unitary group attached to the Hermitian lattice $L$ (resp.\ $L_{v}$).

For each $v \in V_{\fin}$, we fix an orthogonal decomposition 
\[
L_{v} = \bigoplus_{i \ge 0} L_{v, i}
\]
that satisfies the condition in \cite{Gan-Yu}*{4.3 Corollary} and let $n_{v, i}$ denote the rank of the $\cO_{E_{v}}$-lattice $L_{v, i}$.
By using this, we define the $\cO_{E_{v}}$-lattice $M_{L_{v}}$ as 
\[
 M_{L_{v}} \defeq \bigoplus_{i \ge 0} \varpi_{E_{v}}^{-\lfloor i/2 \rfloor}L_{v, i}.
\]

We choose a Haar measure $\mu_{\infty}$ on $\SU(L \otimes_{\bZ} \bR)$ as in \cite{Pra89}*{\S3.6}.

\begin{thm}[{Theorem~\ref{thm:prasad'svolumeformulaspecialunitarycase}}]
\label{thm:prasad'svolumeformulaspecialunitarycaseintro}
   We have
    \[    \mu_{\infty}
\left(
\SU(L \otimes_{\bZ} \bR)/\SU(L)
\right) =
{D^{\lfloor \frac{n}{2}\rfloor(\lfloor \frac{n}{2}\rfloor + \frac{3}{2})}\prod_{i=1}^n\frac{i!}{(2\pi)^{i+1}}\prod_{i=1}^{\lfloor \frac{n+1}{2} \rfloor}\zeta(2i)\prod_{i=1}^{\lfloor \frac{n}{2} \rfloor}L_{E/\bQ}(2i+1)\prod_{v \in V_{\fin}} \lambda (L_{v})},\]
where $\zeta$ denotes the Riemann zeta-function, $L_{E/\bQ}$ denotes the Dirichlet $L$-function associated with the quadratic extension $E/\bQ$, and the factors $\lambda(L_{v})$ are defined as
\[
\lambda(L_{v}) \defeq \lambda(M_{L_{v}}) \cdot q_{E_v}^{\sum_{i <j} \lfloor \frac{j-i-1}{2} \rfloor n_{v, i} \cdot n_{v, j}}
\abs{
\mathbf{G}_{{M}_{L_{v}}}(\bZ_{v}/v \bZ_{v})
} \abs{\mathbf{G}_{L_{v}}(\bZ_{v}/v \bZ_{v})}^{-1}
\]
with $\lambda(M_{L_{v}})$ being defined explicitly in \eqref{defoflambdaMLv} and $\mathbf{G}_{{M}_{L_{v}}}$(resp.\ $\mathbf{G}_{L_{v}}$) denoting the reductive groups over $\bZ_{v}/v \bZ_{v}$ attached to $M_{L_{v}}$ (resp.\ $L_{v}$) as in Subsection~\ref{subseccomparisonofvolumes}.
\end{thm}

When $L_{v} = M_{L_{v}}$ for all $v \in V_{\fin}$, this theorem is a special case of the main result of \cite{Pra89}.
We generalise \cite{Pra89} in the case of special unitary groups (or more generally, simply connected covers of classical groups) using the result in \cite{Gan-Yu}.

A significant application of this theorem is the computation of the Hirzebruch-Mumford volume $\volHM(\SU(L))$.
\begin{cor}
Using the same notation as in Theorem~\ref{thm:prasad'svolumeformulaspecialunitarycaseintro}, we have
\[\volHM(\SU(L)) = \abs{\SU(L)\cap Z}D^{\lfloor \frac{n}{2}\rfloor(\lfloor \frac{n}{2}\rfloor + \frac{3}{2})}\prod_{i=1}^n\frac{i!}{(2\pi)^{i+1}}\prod_{i=1}^{\lfloor \frac{n+1}{2} \rfloor}\zeta(2i)\prod_{i=1}^{\lfloor \frac{n}{2} \rfloor}L_{E/\bQ}(2i+1)\prod_{v \in V_{\fin}} \lambda (L_{v}),\]
    where $Z$ denotes the centre of $\U(L)$.
\end{cor}

\begin{rmk}
    Gritsenko, Hulek, and Sankaran obtained a volume formula for $\Gamma< \O^+(2,n)$ \cite{gritsenko2005hirzebruch}*{Theorem 2.1}.
  This result proved to be of significant importance in the later work on the Kodaira dimension of arithmetic varieties of orthogonal type \cites{gritsenko2006hirzebruch,ma2018kodaira}. The result obtained here can be viewed as its unitary analogue, and is expected to have further applications, such as computing birational invariants of ball quotients.
\end{rmk}

\subsection{Finiteness of free algebras of modular forms on unitary groups}
Building on our earlier derivation of a volume formula for general arithmetic subgroups, we proceed to provide a partial answer to Conjecture \ref{conj:main}.

\begin{thm}[{Theorem \ref{thm:not_wps}}]
\label{mainthm:not_wps}
Let $E$ be an imaginary quadratic field with odd discriminant $-D$ for $D>3$ and $\Gamma<\U(1,n)$ be an arithmetic subgroup.
 Then, the algebra $M_*(\Gamma)$ is never free when $n>99$ or $D$ is sufficiently large.
\end{thm}
\begin{rmk}
\label{rem:-3 never free}
\begin{enumerate}
    \item In the case $E=\bQ(\sqrt{-3})$, the algebra $M_*(\Gamma)$ is never free if $n>154$.
    This is also a part of Theorem \ref{thm:not_wps}.
    \item Furthermore, we will see that $M_*(\Gamma)$ is never free if $L$ is unimodular and $n>2$; see Theorem \ref{thm:unimodular_evaluation}. This result also holds for even $D$ except the Gaussian case $D=4$.
\end{enumerate}
\end{rmk}

We briefly outline the strategy used to prove Theorem~\ref{mainthm:not_wps} here.
Our approach begins by reducing the problem to estimating the Hirzebruch-Mumford volume of unitary groups.
If $M_*(\Gamma)$ is a free algebra, then there exists a special reflective modular form in the sense of \cite{maeda2023fano}, building on the work of Aoki and Ibukiyama \cite{aoki2005simple} and Wang and Williams \cite{WW2021}.
The existence of such a modular form implies a volume identity \cite{bruinier2004}, which in turn yields a criterion, formulated in terms of the Hirzebruch-Mumford volumes, for determining when $M_*(\Gamma)$ is not freely generated (Theorem \ref{thm:criterion_wps}).
Combining this criterion with the explicit volume computations Theorem~\ref{thm:prasad'svolumeformulaspecialunitarycaseintro}, we show in Theorem~\ref{thm:not_wps} that when the inequality 
\begin{align}
\label{ineq:non-freeness main results}
    f(n,D) < \max\left\{1, \left(N(L)/4\right)^{\epsilon}\right\}
\end{align}
holds, the algebra $M_*(\Gamma)$ is not free for any $\Gamma<\U(L)$.
  Here, $\epsilon>0$ is a constant, independent of $L$, $n$, and $E$, and  $N(L)$ denotes a quantity defined in Subsection~\ref{subsecGlobalcomputation}, closely related to the exponent of the finite discriminant group $L^{\vee}/L$, where $L^{\vee}$ denotes the dual lattice of $L$.
  The function $f(n,D)$ is defined as 
  \[f(n,D)\defeq     (1+2\cdot 2^{2n+1}+4^{2n+1})\cdot\frac{2 \cdot  (2\pi)^{n+1}}{(n+1)!\cdot D^{n/2}}\]
  when $D \neq 3$ (for the case of $D = 3$, see Theorem~\ref{thm:not_wps}).
Thus, Conjecture~\ref{conj:main} holds in the range where (\ref{ineq:non-freeness main results}) is satisfied. Since $f(n,D)\to 0$ when $n,D\to\infty$, we extract this range as the statement of Theorem \ref{mainthm:not_wps}. Furthermore, the existence of an invariant $N(L)$, which depends only on $L$, implies that, even outside the scope of the theorem, there can only be finitely many examples where $M_*(\Gamma)$ is free.

  \begin{thm}[{Theorem~\ref{thm:finiteness}}]
  \label{mainthm:finiteness}
      Up to scaling, there are only finitely many isometry classes of Hermitian lattices $L$ of signature $(1,n)$ over $\cO_E$, where $n>2$ and $E$ is an imaginary quadratic field with odd discriminant, such that $M_*(\Gamma)$ is a free algebra for some arithmetic subgroup $\Gamma < \U(L)$.
  \end{thm}

\begin{rmk}
We now summarise previous studies related to our main result.
\begin{enumerate}
\item Theorem \ref{mainthm:not_wps} can be regarded as a unitary analogue of the result in \cite{vinberg2017criterion},  which concerns the case of orthogonal groups.
That work analyses the Satake topology on $\overline{\Gamma\backslash\mathcal{D}}$ by using a tube domain realisation.
In contrast, the complex ball $\bB^n$ is not a tube domain, and therefore similar methods do not apply directly in our setting.
    \item     For specific $L$, Conjecture \ref{conj:main} was established in \cite{stuken2022nonfreeness} through a case-by-case volume computation based on Lie group-theoretic techniques.
Our theorem generalises these results.

\end{enumerate}
\end{rmk}

\subsection{Reflective modular forms}
\label{subsec:reflective modular forms}
Let $f$ be a modular form of weight $\kappa$ with respect to $\Gamma < \U(L)$.
We call $f$ a \emph{reflective modular form} if the support of $\div(f)$ is contained set-theoretically in the union of ramification divisors of the uniformisation map $\bB^n\to X_{\Gamma}$.
Followed by \cite{behrens2012singularities}, any ramification divisor is caused by a reflection with respect to a vector $l\in L$.
We denote by $\mathcal{R}_{\Gamma}$ the set of $Z\Gamma$-equivalence classes of such vectors and $H_l$ the Heegner divisor associated with $l$.
Putting $\div(f) = \sum_{[l]\in\mathcal{R}^{\Gamma}} a_lH_l$, the \emph{slope} of $f$ is defined to be $\rho(f)\defeq\max\{a_i/\kappa\}$.
In the orthogonal case, such modular forms have been extensively studied due to their connections with Kac–Moody algebras \cites{scheithauer2006classification,gritsenko2018reflective,gritsenko1998automorphic,ma2017finiteness,wang2023classificationreflective,wang2024reflective}.
In the case of unitary groups, reflective modular forms are also constructed using the Borcherds lift \cite{borcherds1998automorphic} on certain moduli spaces \cites{allcock2002cubic,freitag2011modular,kondo2013segre,kondo2016igusa}.
The existence of reflective modular forms imposes strong constraints on the canonical bundle of ball quotients and has notable applications to the classification of their compactifications \cites{casalaina2009degenerations,casalaina2012genusfour,casalaina2024nonisomorphic,hulek2024compactifications,hulek2025revisiting}.
Returning to our situation, in the proofs of Theorems \ref{mainthm:not_wps}, \ref{mainthm:finiteness},
 we refer to the fact that a reflective modular form exists when $M_*(\Gamma)$ is free.
As an application of the techniques used in this paper, we also prove a finiteness theorem for such reflective modular forms.
Let $g(n,D)$ be the inverse of $4(n+1) \cdot f(n,D)$, which diverges $\infty$ when $n,D \to \infty$.
For $E=\bQ(\sqrt{-3})$, we slightly change the definition of $g(n,D)$; see Subsection \ref{subsec:Finiteness of reflective modular forms}.
\begin{thm}[{Theorem \ref{thm:finiteness of reflective modular forms}}]
\label{mainthm:finiteness of reflective modular forms}
Let $E$ be an imaginary quadratic field with odd discriminant $-D$.
\begin{enumerate}
    \item There exist no reflective modular forms $f$ such that $\rho(f) \le g(n,D)$.
\item Let $r>0$ be a fixed rational number.
Up to scaling, there are only finitely many isometry classes of Hermitian lattices $L$ of signature $(1,n)$ over $\cO_E$, where $n>2$ and $E$ is an imaginary quadratic field with odd discriminant, such that there exists a reflective modular form $f$ for some arithmetic subgroup $\Gamma < \U(L)$ with $\rho(f) \le r$.
\end{enumerate}
\end{thm}
Computer-based computations indicates that when $n>100$, there are no reflective modular forms with $\rho(f) < 1/(n+1)$ for $E\neq \bQ(\sqrt{-3})$.
This provides a solution to a unitary analogue of the conjecture of Gritsenko and Nikulin \cite{gritsenko1998automorphic}*{Conjecture 2.5.5}, whose original version for orthogonal groups was resolved by Ma \cite{ma2018kodaira}*{Corollaries 1.9, 1.10}; see Subsection \ref{subsec:Finiteness of reflective modular forms} in detail.
Our theorem shows that such modular forms are exceedingly rare.

\subsection{Organization of the paper}

In Section \ref{sec:A formula for covolumes of simply connected classical groups}, we prove Theorem~\ref{thm:prasad'svolumeformulaspecialunitarycaseintro}.
In Subsection~\ref{subsec:reviewofprasad}, we give a brief review of Prasad's volume formula for the covolumes of principal arithmetic subgroups of absolutely simple, simply connected algebraic groups.
In Subsection~\ref{subseccomparisonofvolumes}, we work on classical groups over non-archimedean local fields.
For a compact, open subgroup $G_{L}$ given as the stabiliser of a lattice $L$, we construct another lattice $M_{L}$ whose stabiliser $G_{M_{L}}$ is a parahoric subgroup containing $G_{L}$ and compute the index $\abs{G_{M_{L}}/G_{L}}$ explicitly.
Combining this local computation with Prasad's volume formula, we prove the main result (Theorem~\ref{thm:generalization of Prasad's volume formula}) of this section, which gives an explicit formula for the covolume of an arithmetic subgroup presented as the stabiliser of a lattice that is not necessarily principal.
In Subsection~\ref{subsec:thecaseofspecialunitary}, we restrict ourselves to the case of special unitary groups over $\bQ$ and rewrite Theorem~\ref{thm:generalization of Prasad's volume formula} in a more explicit form.

Building on the volume formula, Section \ref{section:Estimation of volumes} is devoted to the proof of Theorem \ref{mainthm:not_wps}.
As explained above, a key idea is a criterion Theorem \ref{thm:criterion_wps} showing that if $M_*(\Gamma)$ is free, then the associated arithmetic subgroup must satisfy a special volume constraint.
We prove Theorem \ref{mainthm:not_wps} by evaluating this volume explicitly using Theorem~\ref{thm:generalization of Prasad's volume formula}.

Section \ref{sec:concrete examples} discusses two applications. First, we prove that reflective modular forms on ball quotients are rare.
Second, we apply our volume computations to prove that $M_*(\Gamma)$ for the moduli space of cubic threefolds is not free (Proposition \ref{prop:cubic threefolds wps}), which gives another proof that does not rely on the computation of cohomology.

\subsection*{Acknowledgements}
The authors are grateful to Jan Bruinier, Nils Scheithauer, and Klaus Hulek for valuable discussions on modular forms and the geometry of ball quotients.
We also thank Haowu Wang and Beandon Williams for their helpful comments on reflections and for informing us about their research on the moduli space of cubic surfaces, and Mikhail Belolipetsky and Sai-Kee Yeung for their insightful comments regarding the applications of minimal covolumes of arithmetic subgroups.
Our thanks further go to the Max Planck Institute for Mathematics for their hospitality.
Y.M. is partially supported by the Alexander von Humboldt Foundation through a Humboldt Research Fellowship.

\section*{Notation and conventions}
For a global filed $k$, we write $\cO_{k}$ for the ring of integers of $k$.
We denote by $V_{k}$, $V_{\fin}$, and $V_{\infty}$ the set of places of $k$, the set of finite places of $k$, and the set of infinite places of $k$, respectively.
For $v \in V_{k}$, let $k_{v}$ denote the completion of $k$ at $v$.
For $v \in V_{\fin}$, we denote by $\cO_{k_{v}}$ the ring of integers of $k_{v}$ and $\mathfrak{p}_{k_{v}}$ the maximal ideal of $\cO_{k_{v}}$.
We also write $\mathfrak{f}_{v} \defeq \cO_{k_{v}}/\mathfrak{p}_{k_{v}}$ and $q_{v} \defeq \abs{\mathfrak{f}_{v}}$.
For each $v \in V_{k}$, we fix the normalized absolute value $\abs{\phantom{x}}_{v}$ on $k_{v}$ as in \cite{Pra89}*{\S0.1}.

Suppose that $E/k$ is a quadratic extension of global fields.
For $v \in V_{k}$, we write $E_{v} \defeq E \otimes_{k} k_{v}$.
For $v \in V_{\fin}$, let $\cO_{E_{v}}$ denote the maximal $\cO_{k_{v}}$-order in $E_{v}$.
If $v$ is inert or splits over $E$, we write $\mathfrak{p}_{E_{v}} \defeq \mathfrak{p}_{k_{v}} \cdot \cO_{E_{v}}$.
If $v$ ramifies over $E$, let $\mathfrak{p}_{E_{v}}$ be the maximal ideal of $\cO_{E_{v}}$.
We write $q_{E_{v}} \defeq \abs{\cO_{E_{v}} / \mathfrak{p}_{E_{v}}}$.
In Subsection~\ref{subseccomparisonofvolumes}, we work over local fields. In doing so, we adopt a simplified notation by omitting the explicit reference to the place $v$; see Remark \ref{rem:notation local global}

\section{Covolumes of arithmetic subgroups of simply connected classical groups}
\label{sec:A formula for covolumes of simply connected classical groups}

Let $k$ be a global field.
Let $G$ be a classical group over $k$ and $G_{\scn}$ be the simply connected cover of the derived group of $G$, that is,
$G_{\scn}$ is one of the following groups: spin groups; symplectic groups; and special unitary groups.
In this section, we will prove an explicit formula for the covolumes of arithmetic subgroups of $G_{\scn}$ presented as stabilisers of Hermitian lattices (see Theorem~\ref{thm:generalization of Prasad's volume formula}).
Our result is obtained by combining an explicit computation of the index of compact, open subgroups of $p$-adic groups with prior work \cite{Pra89} by Prasad, where he obtained a volume formula for the \emph{principal} arithmetic subgroups of absolutely simple, simply connected algebraic groups. (For the definition of principal arithmetic subgroups, see Subsection~\ref{subsec:reviewofprasad}.)
We will apply the main result of this section to the case of special unitary groups in Section~\ref{section:Estimation of volumes}.

\subsection{A review of Gopal Prasad's volume formula}
\label{subsec:reviewofprasad}

In this subsection, we will give a brief review of Prasad's volume formula following \cite{Pra89} and \cite{KalethaPrasad}*{Section~18}.
We fix a non-empty finite subset $S$ of $V_{k}$ containing all infinite places.
In Section~\ref{section:Estimation of volumes}, we will assume that $k = \mathbb{Q}$ and take $S = V_{\infty}$.
Let $\mathbb{A}$ denote the $k$-algebra of ad{\`e}les of $k$ and $\mathbb{A}_{S}$ denote the $k$-algebra of $S$-ad{\`e}les, which is the restricted direct product of $k_{v}$ for $v \in V_{k} \smallsetminus S$.

Let $H$ be an absolutely almost simple, simply connected group over $k$.
We write $H_{S} = \prod_{v \in S} H(k_{v})$.
We assume that $H_{S}$ is non-compact so that it satisfies the strong approximation property (see the proof of Lemma~\ref{lemma:covolume=productofnon-archimedean} for instance).
Let $\iota \colon H(k) \hookrightarrow H(\mathbb{A})$ and $\iota_{S} \colon H(k) \hookrightarrow H_{S}$ be the diagonal embeddings.

Let $K$ be a compact, open subgroup of $H(\mathbb{A}_{S})$ of the form $K = \prod_{v \in V_{k} \smallsetminus S} K_{v}$, where $K_{v}$ is a compact, open subgroup of $H(k_{v})$ for each $v \in V_{k} \smallsetminus S$.
We note that for all but finitely many $v$, the group $K_{v}$ is a hyperspecial parahoric subgroup of $H(k_{v})$ (see \cite{KalethaPrasad}*{\S18.1.9}).
We define a subgroup $\Gamma$ of $H(k)$ by
\[
\Gamma \defeq H(k) \cap \iota^{-1}\left(
H_{S} \times K
\right).
\]
Then the group $\iota_{S}(\Gamma)$ is a lattice of $H_{S}$.
In this case, we say that $\Gamma$ is the $S$-arithmetic subgroup associated to the compact, open subgroup $K$ of $H(\mathbb{A}_{S})$ (see \cite{Pra89}*{\S3.4} and \cite{KalethaPrasad}*{Definition~18.1.10}).

Let $\mu_{\mathbb{A}}$ be a Haar measure on $H(\mathbb{A})$.
We fix a Haar measure $\mu_{v}$ on $H(k_{v})$ for each $v \in V_{k}$ such that the restriction of $\mu_{\mathbb{A}}$ on $H_{S} \times K$ agrees with the product measure 
\[
\prod_{v \in S} \mu_{v} \times \prod_{v \in V_{k} \smallsetminus S} \mu_{v} 
\restriction_{K_{v}}.
\]
We define a Haar measure $\mu_{S}$ on $H_{S}$ by $\mu_{S} \defeq \prod_{v \in V_{k} \smallsetminus S} \mu_{v}$.

\begin{lemma}[\cite{Pra89}]
\label{lemma:covolume=productofnon-archimedean}
We have
\[
\mu_{S}\left(
H_{S}/\iota_{S}(\Gamma)
\right) = \mu_{\mathbb{A}}\left(
H(\mathbb{A})/\iota(H(k))
\right) \times 
\left(
\prod_{v \in V_{k} \smallsetminus S} \mu_{v}(K_{v})
\right)^{-1}.
\]
\end{lemma}

The lemma is a straightforward implication of the strong approximation property, and it is explained in \cite{Pra89}*{\S3.4}.
Although Prasad assumes that $K_{v}$ is a parahoric subgroup for each $v \in V_{k} \smallsetminus S$, we do not need the assumption for this claim.
We will give a brief proof of the lemma following \cite{Pra89}*{\S3.4} for completeness.

\begin{proof}
By the strong approximation property, we have
\[
H(\mathbb{A}) = \left(
H_{S} \times K
\right) \cdot \iota(H(k)).
\]
Hence, the natural inclusion $H_{S} \times K \subseteq H(\mathbb{A})$ induces an isomorphism
\[
(H_{S} \times K)/\iota(\Gamma) \simeq H(\mathbb{A})/\iota(H(k)).
\]
Since the projection $H_{S} \times K \rightarrow H_{S}$ defines a principal fibration $(H_{S} \times K)/\iota(\Gamma) \twoheadrightarrow H_{S}/\iota_{S}(\Gamma)$ with fibre $K = \prod_{v \in V_{k} \smallsetminus S} K_{v}$, we obtain the claim.
\end{proof}

In \cite{Pra89}, Prasad gave an explicit formula for the covolume $\mu_{S}\left(
H_{S}/\iota_{S}(\Gamma)
\right)$ with respect to appropriately fixed measure $\mu_{S}$ assuming that $\Gamma$ is a principal $S$-arithmetic subgroup of $H(k)$.
Here, an $S$-arithmetic subgroup $\Gamma$ associated to a compact, open subgroup $K = \prod_{v \in V_{k} \smallsetminus S} K_{v}$ is called \emph{principal} if the groups $K_{v}$ are parahoric subgroups of $H(k_{v})$ for all $v \in V_{k} \smallsetminus S$ (see \cite{Pra89}*{\S3.4}).
Before describing Prasad's volume formula, we record an immediate corollary of Lemma~\ref{lemma:covolume=productofnon-archimedean} that will be used to calculate the covolume $\mu_{S}\left(
H_{S}/\iota_{S}(\Gamma)
\right)$ for ``non-principal cases'' below.

\begin{cor}
\label{cor:comparisonofcovoleme}
    Let $K = \prod_{v \in V_{k} \smallsetminus S} K_{v}$ and $K' = \prod_{v \in V_{k} \smallsetminus S} K'_{v}$ be compact, open subgroups of $H(\mathbb{A}_{S})$, and let $\Gamma$ and $\Gamma'$ be $S$-arithmetic subgroups of $H(k)$ associated to $K$ and $K'$, respectively.
    Then we have
    \[
    \frac
    {
\mu_{S}
\left(
H_{S}/\iota_{S}(\Gamma)
\right) }
{
\mu_{S}
\left(
H_{S}/\iota_{S}(\Gamma')
\right)
} = \prod_{v \in V_{k} \smallsetminus S} \frac{
\mu_{v}(K'_{v})
}{
\mu_{v}(K_{v})
}.
    \]
\end{cor}

In the rest of this subsection, we suppose that the group $K_{v}$ is a parahoric subgroup of $H(k_{v})$ for each $v \in V_{k} \smallsetminus S$ and the Haar measure $\mu_{v}$ is chosen as in \cite{Pra89}*{\S3.6} for all $v \in S$.
In this case, Prasad gave an explicit formula for the covolume $\mu_{S}\left(H_{S}/\iota_{S}(\Gamma)\right)$.
\begin{thm}[{\cite{Pra89}*{3.7.~Theorem}, \cite{KalethaPrasad}*{Theorem~18.5.6}}]
\label{Prasad'svolumeformula}
We have
\[
\mu_{S}
\left(
H_{S}/\iota_{S}(\Gamma)
\right)
= D_{k}^{\frac{1}{2} \dim H} \left(
\frac{D_{\ell}}
{D_{k}^{[\ell:k]}}
\right)^{\frac{1}{2} \mathfrak{s}(H_{\qs})}
\left(
\prod_{v \in V_{\infty}}
\abs{
\prod_{i=1}^{r} \frac{m_{i}!}{(2 \pi)^{m_{i} + 1}}
}_{v}
\right)\mathscr{E}_{\Gamma},
\]
where
\begin{itemize}
\item 
$\dim H$ denotes the dimension of $H$;
\item 
$\ell$ denotes a finite field extension of $k$, and $D_{k}$ and $D_{\ell}$ denote numbers defined in \cite{Pra89}*{\S0.2};
\item 
$H_{\qs}$ denotes the quasi-split inner form of $H$;
\item 
$\mathfrak{s}(H_{\qs})$ denotes the integer defined in \cite{Pra89}*{\S0.4};
\item 
$r$ denotes the absolute rank of $H_{\qs}$, and $m_{1}, m_{2}, \ldots , m_{r}$ denote the exponents of the simple, simply connected, compact real-analytic Lie group of the same type as $H_{\qs}$ (see \cite{Pra89}*{\S1.5});
\item 
$\mathscr{E}_{\Gamma} = \mathscr{E}$ denotes the number described explicitly in \cite{Pra89}*{3.7.~Theorem} (see also \cite{KalethaPrasad}*{Proposition~18.5.10} for another description).
\end{itemize}
\end{thm}

Since we are only concerned with the case of special unitary groups in this paper, we do not recall the precise definitions of the factors appearing on the right-hand side here.
Instead, in Theorem~\ref{thm:prasad'svolumeformulaspecialunitarycase} below, we provide a more explicit description of the right-hand side of Theorem~\ref{Prasad'svolumeformula} in the case where $H$ is a special unitary group.

\begin{rmk}
In the statement of \cite{Pra89}*{3.7.~Theorem}, the number $\tau_{k}(H)$ called the \emph{Tamagwa number} appears as a factor of the right-hand side.
However, thanks to the recent works by many mathematicians, it was proved that $\tau_{k}(H) = 1$ for any simply connected semi-simple group $H$ over a global field $k$ (for details, see the discussion in \cite{KalethaPrasad}*{\S18.5.2}).
\end{rmk}

\subsection{Comparison of volumes at finite places}
\label{subseccomparisonofvolumes}

In the previous subsection, we explained the explicit calculation of the covolume $\mu_{S}
\left(
H_{S}/\iota_{S}(\Gamma)
\right)$ by Prasad for a principal $S$-arithmetic subgroup $\Gamma$.
On the other hand, for a later application in Section~\ref{section:Estimation of volumes}, we would like to take the group $\Gamma$ to be the special unitary group $\SU(L)$, where $L$ is a Hermitian lattice over $\cO_{E}$, with $E$ being a quadratic extension of $k$.
In this case, $\Gamma$ is the $S$-arithmetic subgroup associated to the compact, open subgroup $K = \prod_{v \in V_{k} \smallsetminus S} \SU(L_{v})$, where $L_{v} = L \otimes_{\cO_{k}} \cO_{k_{v}}$.
Since the group $\SU(L_{v})$ is not necessarily a parahoric subgroup, the group $\Gamma$ is not necessarily a principal $S$-arithmetic subgroup.
Motivated by this, in the rest of this section, we restrict our interest to the case of special unitary groups, or more generally, simply connected covers of classical groups, and will extend Theorem~\ref{Prasad'svolumeformula} to more general $S$-arithmetic subgroups including $\SU(L)$.

In this subsection, we fix $v \in V_{k} \smallsetminus S$ and will compare the volume $\mu_{v}(\SU(L_v))$ with the volume $\mu_{v}(\SU(L'_{v}))$ for an appropriate parahoric subgroup $\SU(L'_{v})$ by calculating the index $\abs{\SU(L'_{v})/\SU(L_{v})}$ explicitly.
In \S\ref{subsec:generalizationofprasad}, we will combine the result of this subsection with Corollary~\ref{cor:comparisonofcovoleme} to calculate the covolume of $\SU(L)$.

We will work in a bit more general setting as in \cite{Gan-Yu}, which we recall here briefly.
Let $F$ be a non-archimedean local field of residue characteristic $p$.
We write $\mathcal{O}_{F}$ for the ring of integers of $F$.
We fix a uniformiser $\varpi_{F}$ of $\mathcal{O}_{F}$.
Let $\epsilon \in \{\pm 1\}$ and let $(E, \sigma)$ be one of the following $F$-algebras with involution:
\begin{enumerate}
\item $E = F$ and $\sigma = \id_{F}$; 
\item $E$ is a quadratic extension of $F$ and $\sigma$ is the unique non-trivial automorphism of $E/F$;
\item $E = F \oplus F$ and $\sigma(x, y) = (y, x)$;
\item 
$E = D$, the quaternion algebra over $F$ and $\sigma$ is the standard involution.
\end{enumerate}
\begin{rmk}

\label{rem:notation local global}
In the following subsections, where we work in a global setting, we will consider the local filed $k_{v}$ and the $k_{v}$-algebra $E_{v}$ for a global field $k$, a finite place $v \in V_{\fin}$, and a quadratic extension $E/k$.
However, in this subsection only, we adopt the simplified notation $F$ and $E$ for $k_{v}$ and $E_{v}$ for the sake of notational convenience.
Similarly, in this subsection, we use the notation $L$ for a Hermitian lattice over a local field, which will be denote by $L_{v}$ in the following subsections.
\end{rmk}
Throughout this subsection, we impose the following assumption, as is done in most parts of \cite{Gan-Yu} (see \cite{Gan-Yu}*{Section~9}).
\begin{asm}
\label{asump=2}
    We assume $p \neq 2$ if $E$ is a ramified quadratic extension of $F$ or $E = F$ and $\epsilon = 1$.
\end{asm}

Let $\mathcal{O}_{E}$ be the maximal $\cO_{F}$-order in $E$.
If $E$ is a ramified extension of $F$, or $E = D$, we let $\varpi_{E}$ be a uniformiser of $\mathcal{O}_{E}$ and put $e = 2$.
Otherwise, we write $\varpi_{E} = \varpi_{F}$ and put $e = 1$.
We also write $\mathfrak{p}_{F} = \varpi_{F} \cO_{F}$ and $\mathfrak{p}_{E} = \varpi_{E} \cO_{E}$.

Let $L$ be an $\mathcal{O}_{E}$-lattice of finite rank with a $(\sigma, \epsilon)$-Hermitian form
\[
\langle \phantom{x},\phantom{x} \rangle \colon L \times L \rightarrow \mathcal{O}_{E}.
\]
We assume that $V := L \otimes_{\mathcal{O}_{F}} F$ is non-degenerate with respect to $\langle \phantom{x},\phantom{x} \rangle$.
We define the dual lattice $L^{\vee}$ of $L$ by
\[
L^{\vee} = \{
x \in V \mid \langle x, L \rangle \subseteq \mathcal{O}_{E}
\}.
\] 
We will use similar notation for other $\mathcal{O}_{E}$-lattices below.
According to \cite{Gan-Yu}*{4.3~Corollary}, we have an orthogonal decomposition
\[
L = \bigoplus_{i = 0}^{N} L_{i},
\]
where $L_{i}$ is a sublattice of $L$ such that
\[
L_{i}^{\vee} = \mathfrak{p}_{E}^{-i} L_{i}.
\]
Such an orthogonal decomposition is called a \emph{Jordan splitting} of $L$.
For $0 \le i \le N$, let $n_{i}$ denote the rank of the $\cO_{E}$-lattice $L_{i}$.
Although a Jordan splitting of $L$ is not unique, the rank $n_{i}$
of each summand is uniquely determined by $L$.
We record this fact as a lemma for later use.
\begin{lemma}[{\cite{Omeara2000introduction}*{91.9.Theorem}}]
\label{lem:rankofLJDwelldef}
Let $L = \bigoplus_{i=0}^{N} L_{i}$ and $L = \bigoplus_{i=0}^{N'} K_{i}$ be Jordan splittings of $L$ such that $L_{N} \neq \{0\}$ and $K_{N'} \neq \{0\}$.
Then we have $N = N'$ and the ranks of the sublattices $L_{i}$ and $K_{i}$ coincide for all $0 \le i \le N$.
In particular, the condition $L_{i} \neq \{0\}$ does not depend on the choice of a Jordan splitting.
\end{lemma}
\begin{proof}
The lemma is a part of \cite{Omeara2000introduction}*{91.9.Theorem}; see also \cite{jacobowitz}*{Section 4}.
\end{proof}

From now on, we suppose that $N \ge 2$.
We define an $\mathcal{O}_{E}$-lattice $L'$ in $V$ by
\[
L' = \bigoplus_{i = 0}^{N-1} L'_{i},
\]
where
\[
L'_{i} =
\begin{cases}
L_{i} & (i \neq N-2), \\
L_{N-2} \oplus \mathfrak{p}_{E}^{-1} L_{N} & (i = N-2).
\end{cases}
\]
We note that $(L'_{i})^{\vee} = \mathfrak{p}_{E}^{-i} L'_{i}$ for all $0 \le i \le N-1$.
For $0 \le i \le N-1$, let $n'_{i}$ denote the rank of $L'_{i}$.

For $x \in V$, we define $x^{L}_{i} \in L_{i} \otimes_{\mathcal{O}_{F}} F$ for each $0 \le i \le N$ by $x = \sum_{i=0}^{N} x^{L}_{i}$.
Similarly, for $x \in V$, we define $x^{L'}_{i} \in L'_{i}\otimes_{\mathcal{O}_{F}} F$ for each $0 \le i \le N-1$ by $x = \sum_{i=0}^{N-1} x^{L'}_{i}$.
We note that for $x \in V$, we have $x_{i}^{L} = x_{i}^{L'}$ for all $0 \le i \le N-1$ with $i \neq N-2$, and $x_{N-2}^{L'} = x_{N-2}^{L} + x_{N}^{L}$.

Let $G_{L}$ and $G_{L'}$ be the isometry groups of $(L, \langle \phantom{x},\phantom{x} \rangle)$ and $(L', \langle \phantom{x},\phantom{x} \rangle)$, respectively.
We also denote by $G_{V}$ the isometry group of $(V, \langle \phantom{x},\phantom{x} \rangle)$ and identify $G_{L}$ (resp.\ $G_{L'}$) with the stabiliser of $L$ (resp.\ $L'$) in $G_{V}$.
\begin{lemma}
\label{lemma:modplowertriangular}
Let $M = L$ or $L'$, and $0 \le i \le j \le N$ if $M = L$ and $0 \le i \le j \le N-1$ if $M = L'$.
Then, for all $g \in G_{M}$ and $x \in M_{j}$, we have $g(x)^{M}_{i} \in \mathfrak{p}_{E}^{j-i} M_{i}$.
\end{lemma}
\begin{proof}
    We write $y = \varpi_{E}^{-j} x \in \mathfrak{p}_{E}^{-j} M_{j} = M_{j}^{\vee}$.
    Since $g$ fixes $M$ and the Hermitian form $\langle \phantom{x},\phantom{x} \rangle$, it also fixes the dual lattice $M^{\vee}$ of $M$.
    Thus, we have
    \[
    g(x) = g(\varpi_{E}^{j} y) \in \mathfrak{p}_{E}^{j} M^{\vee} = \bigoplus_{k \ge 0} \mathfrak{p}_{E}^{j-k} M_{k}.
    \]
    Thus, we have $g(x)^{M}_{i} \in \mathfrak{p}_{K}^{j-i} M_{i}$, as desired.
\end{proof}
\begin{prop}
\label{prop:GLasasubgroupofGL'}
    As subgroups of $G_{V}$, we have $G_{L} \subseteq G_{L'}$.
    More precisely, we have
    \[
G_{L} = \left\{
g \in G_{L'} \mid g(x)^{L}_{N} \in L_{N} \quad (x \in L) 
\right\}.
    \]
\end{prop}
\begin{proof}
Let $g \in G_{L}$.
We will prove that $g(L') \subseteq L'$.
Since $g \in G_{L}$, for all $0 \le i \le N-1$ with $i \neq N-2$, we have 
\[
g(L'_{i}) = g(L_{i}) \subseteq L \subseteq L'.
\]
We also have $g(L_{N-2}) \subseteq L \subseteq L'$.
Moreover, according to Lemma~\ref{lemma:modplowertriangular}, we have
\[
g(\mathfrak{p}_{E}^{-1} L_{N}) = \mathfrak{p}_{E}^{-1} g(L_{N}) \subseteq \bigoplus_{i=0}^{N} \mathfrak{p}_{E}^{-1} \mathfrak{p}_{E}^{N-i} L_{i} \subseteq \bigoplus_{i=0}^{N-1} L_{i} \oplus \mathfrak{p}_{E}^{-1} L_{N} = L'.
\]
Thus, we obtain the first claim.
We will prove the last claim.
Let $g \in G_{L'}$. Then, for all $x \in L \subseteq L'$ and $0 \le i \le N-1$ with $i \neq N-2$, we have
\[
g(x)^{L}_{i} = g(x)^{L'}_{i} \in L'_{i} = L_{i}.
\]
We also have 
\[
g(x)^{L}_{N-2} + g(x)^{L}_{N} = g(x)^{L'}_{N-2} \in L'_{N-2} = L_{N-2} \oplus \mathfrak{p}_{E}^{-1} L_{N}.
\]
Hence, we have $g(x)^{L}_{N-2} \in L_{N-2}$.
Thus, we conclude that
\begin{align*}
G_{L} &= G_{L} \cap G_{L'} \\
&= \left\{
g \in G_{L'} \mid g(x) \in L \quad (x \in L)
\right\} \\
&= \left\{
g \in G_{L'} \mid g(x)^{L}_{N} \in L_{N} \quad (x \in L) 
\right\}.
\qedhere
\end{align*}
\end{proof}

We will compute the index $\abs{G_{L'}/G_{L}}$ below.
We define an open, normal subgroup $G_{L'}^{+}$ of $G_{L'}$ as
\[
G_{L'}^{+} \defeq \left\{
g \in G_{L'} \mid g(x^{L'}_{i})^{L'}_{i} - x^{L'}_{i} \in \mathfrak{p}_{E} L'_{i} \quad (0 \le i \le N-1, x \in L')
\right\}.
\]
Then we have
\[
\abs{G_{L'}/G_{L}} = 
\abs{G_{L'}/G_{L} G_{L'}^{+}} \abs{G_{L}G_{L'}^{+}/G_{L}}.
\]
To compute the first factor, we describe the quotient $G_{L'}/G_{L'}^{+}$.

For $0 \le i \le N$,
let 
$\sigma_{i} \colon \cO_{E}/\mathfrak{p}_{E} \rightarrow \cO_{E}/\mathfrak{p}_{E}$ denote the reduction modulo $\mathfrak{p}_{E}$ of the automorphism $x \mapsto \varpi_{E}^{-i} \sigma(x) \varpi_{E}^{i}$ on $\cO_{E}$.
We also define $\epsilon_{i} \in \{\pm 1\}$ by
\[
\epsilon_{i} \defeq \begin{cases}
\epsilon & (e = 1), \\
(-1)^{i} \epsilon & (e = 2).
\end{cases}
\]
For $0 \le i \le N-1$, we define a $(\sigma_{i}, \epsilon_{i})$-Hermitian form $\langle \phantom{x},\phantom{x} \rangle'_{i}$ on the $\cO_{E}/\mathfrak{p}_{E}$-vector space $\overline{L}'_{i} \defeq L'_{i}/\mathfrak{p}_{E} L'_{i}$ by $\langle x, y \rangle'_{i} \defeq \varpi_{E}^{-i} \langle x, y \rangle \mod \mathfrak{p}_{E}$.
Let $G'_{i}$ denote the isometry group of $\left(\overline{L}'_{i}, \langle \phantom{x},\phantom{x} \rangle'_{i}\right)$.

\begin{rmk}
\label{rmkaboutindepofreductivequot}
Here, we define the $\cO_{E}/\mathfrak{p}_{E}$-Hermitian lattice $\left(\overline{L}'_{i}, \langle \phantom{x},\phantom{x} \rangle'_{i}\right)$ using a Jordan splitting.
However, as in \cite{Gan-Yu}*{Section~6}, it can in fact be defined without fixing a Jordan splitting. In particular, the group $G'_{i}$ is independent of the choice of a Jordan splitting.
\end{rmk}

For $g \in G_{L'}$ and $0 \le i \le N-1$, we define an element $g_{i} \in G'_{i}$ by the reduction module $\mathfrak{p}_{E}$ of the automorphism $x \mapsto g(x)^{L'}_{i}$ for $x \in L'_{i}$.
Then the map $g \mapsto g_{i}$ defines a group homomorphism $G_{L'} \rightarrow G'_{i}$.
We define the group homomorphism $r \colon G_{L'} \rightarrow \prod_{i=0}^{N-1} G'_{i}$ by $r(g) = (g_{i})_{0 \le i \le N-1}$.
\begin{lemma}
\label{lemma:reductivequotientofGl'}
The homomorphism $r$ is surjective with kernel $G_{L'}^{+}$.
\end{lemma}
\begin{proof}
The claim that $G_{L'}^{+}$ is the kernel of $r$ follows from the definition of $G_{L'}^{+}$.
The surjectivity follows from \cite{Gan-Yu}*{6.2.1~Proposition, 6.3.1~Proposition}, which state that $\prod_{i=0}^{N-1} G'_{i}$ is the maximal reductive quotient of the special fibre of the smooth integral model $\underline{G}'$ of $G_{L'}$ constructed in \cite{Gan-Yu}*{Section~5}.
\end{proof}

According to Lemma~\ref{lemma:reductivequotientofGl'}, we have
\[
\abs{G_{L'}/G_{L} G_{L'}^{+}} =
\abs{
\left(
\prod_{i=0}^{N-1} G'_{i} 
\right)
/ r(G_{L})
}.
\]
We will determine the image $r(G_{L})$ of $G_{L}$.
We write $\overline{L}_{N} \defeq L_{N}/\mathfrak{p}_{E} L_{N}$ and let $\langle \phantom{x},\phantom{x} \rangle_{N}$ denote the $(\sigma_{N}, \epsilon_{N})$-Hermitian form of $\overline{L}_{N}$ defined by $\langle x, y \rangle_{L} \defeq \varpi_{E}^{-N} \langle x, y \rangle \mod \mathfrak{p}_{E}$.
Similarly, we define a $(\sigma_{N-2}, \epsilon_{N-2})$-Hermitian form $\langle \phantom{x},\phantom{x} \rangle_{N-2}$ on the $\cO_{E}/\mathfrak{p}_{E}$-vector space $\overline{L}_{N-2} := L_{N-2}/\mathfrak{p}_{E} L_{N-2}$.
We note that the map $L_{N-2} \oplus L_{N} \rightarrow L'_{N-2}$ defined by $(x, y) \mapsto x + \varpi_{E}^{-1} y$ induces an isomorphism of $\cO_{E}/\mathfrak{p}_{E}$-lattices
\begin{equation}
\label{isomoflattices}
\overline{L}_{N-2} \oplus \overline{L}_{N} \isoarrow \overline{L}'_{N-2}.
\end{equation}
Let $G_{N-2}$ and $G_{N}$ be the isometry groups of $\left(\overline{L}_{N-2}, \langle \phantom{x},\phantom{x} \rangle_{N-2}\right)$ and $\left(\overline{L}_{N}, \langle \phantom{x},\phantom{x} \rangle_{N}\right)$, respectively.
Then using the isomorphism in \eqref{isomoflattices}, we obtain an injection $G_{N-2} \times G_{N} \hookrightarrow G'_{N-2}$.
More precisely, we have an isomorphism
\begin{equation}
\label{isomogisometorygroups}
G_{N-2} \times G_{N} \simeq \left\{
g \in G'_{N-2} \mid g(\overline{L}_{N-2}) \subseteq \overline{L}_{N-2}, \, g(\overline{L}_{N}) \subseteq \overline{L}_{N}
\right\},
\end{equation}
where we regard $\overline{L}_{N-2}$ and $\overline{L}_{N}$ as sublattices of $\overline{L}'_{N-2}$ via \eqref{isomoflattices}.
We regard $G_{N-2} \times G_{N}$ as a subgroup of $G'_{N-2}$ via the isomorphism in \eqref{isomogisometorygroups}.
\begin{rmk}
When $e = 2$, the isomorphism in \eqref{isomoflattices} is not an isomorphism of Hermitian spaces. 
Nevertheless, we have the isomorphism of isometry groups in \eqref{isomogisometorygroups} since the Hermitian form $\langle \phantom{x},\phantom{x} \rangle_{N}$ and the restriction of $\langle \phantom{x},\phantom{x} \rangle'_{N-2}$ to $\overline{L}_{N}$ differs only by the scalar multiple of $\sigma(\varpi_{E})/\varpi_{E} \mod \mathfrak{p}_{E}$.
\end{rmk}
\begin{prop}
\label{prop:imageofGLviar}
We have
\[
r(G_{L}) = \prod_{i=0}^{N-1} G''_{i},
\]
where
\[
G''_{i} = 
\begin{cases}
G'_{i} & (i \neq N-2), \\
G_{N-2} \times G_{N} & (i = N-2).
\end{cases}
\]
\end{prop}
\begin{proof}
According to Proposition~\ref{prop:GLasasubgroupofGL'}, we have
\[
r(G_{L}) = \prod_{i=0}^{N-1} G'''_{i},
\]
where
\[
G'''_{i} = 
\begin{cases}
G'_{i} & (i \neq N-2), \\
\left\{
g \in G'_{N-2} \mid g(\overline{L}_{N-2}) \subseteq \overline{L}_{N-2}
\right\} & (i = N-2).
\end{cases}
\]
We will prove that $G'''_{N-2} = G_{N-2} \times G_{N}$.
Let $g \in G'''_{N-2}$. 
It suffices to show that $g(\overline{L}_{N}) \subseteq \overline{L}_{N}$.
Let $x \in \overline{L}_{N-2}$ and $y \in \overline{L}_{N}$.
Since $g \in G'''_{N-2} \subseteq G'_{N-2}$, we have 
\[
\langle g(x), g(y) \rangle'_{N-2}
= \langle x, y \rangle'_{N-2} = 0.
\]
Moreover, since $g \in G'''_{N-2}$, we have $g(x) \in \overline{L}_{N-2}$.
As $x$ ranges over all elements of $\overline{L}_{N-2}$, we obtain that $g(y)$ lies in the orthogonal complement of $\overline{L}_{N-2}$ in $\overline{L}'_{N-2}$, which is equal to $\overline{L}_{N}$.
Thus, we have $g(y) \in \overline{L}_{N}$, as desired.
\end{proof}
Now, we can compute the index $\abs{G_{L'}/G_{L} G_{L'}^{+}}$ as
\begin{equation}
\label{calculationoffirstfactor}
\abs{G_{L'}/G_{L} G_{L'}^{+}}
=
\abs{
\left(
\prod_{i=0}^{N-1} G_{i} 
\right)
/ r(G_{L})
} = \abs{
G'_{N-2}/ (
G_{N-2} \times G_{N}
)
}.
\end{equation}
For later use, we rewrite this result using smooth integral models of $G_{L}$ and $G_{L'}$.
Let $\mathbf{G}_{L}$ (resp.\ $\mathbf{G}_{L'}$) denote the maximal reductive quotient of the special fibre of the smooth integral model of $G_{L}$ (resp.\ $G_{L'}$) constructed in \cite{Gan-Yu}*{Section~5}.
\begin{prop}
\label{prop:calculationoffirstfactorreductivequot}
We have 
\[
\abs{G_{L'}/G_{L} G_{L'}^{+}}
=
\abs{
\mathbf{G}_{L'}(\cO_{F}/\mathfrak{p}_{F})
} \abs{\mathbf{G}_{L}(\cO_{F}/\mathfrak{p}_{F})}^{-1}.
\]
\end{prop}
\begin{proof}
According to \cite{Gan-Yu}*{6.2.1 Proposition, 6.3.1 Proposition}, we have
\[
\mathbf{G}_{L'}(\cO_{F}/\mathfrak{p}_{F}) = \prod_{i=0}^{N-1} G'_{i}
\qquad
\text{and}
\qquad
\mathbf{G}_{L'}(\cO_{F}/\mathfrak{p}_{F}) = \prod_{i=0}^{N-1} G''_{i},
\]
where $G''_{i}$ are defined in Proposition~\ref{prop:imageofGLviar}.
Thus, the claim follows from \eqref{calculationoffirstfactor}.
\end{proof}

Next, we will compute the index $\abs{G_{L}G_{L'}^{+}/G_{L}}$.
We define an open subgroup $G_{L'}^{++}$ of $G_{L'}$ as
\[
G_{L'}^{++} \defeq \left\{
g \in G_{L'} \mid g(x) - x \in \mathfrak{p}_{E} L', g(x^{L'}_{j})^{L'}_{i} \in \mathfrak{p}_{E}^{j-i+1} L'_{i} \quad (x \in L', 0 \le i < j \le N-1)
\right\}.
\]

\begin{lemma}
\label{lem:SUL'+subsetSUL}
    We have $G_{L'}^{++} \subseteq G_{L}$.
\end{lemma}
\begin{proof}
Let $g \in G_{L'}^{++}$ and $x \in L \subseteq L'$.
Then, we have
\[
g(x) - x \in \mathfrak{p}_{E} L' \subseteq L.
\]
Thus, we conclude that $g(x) \in L$.
\end{proof}

We define a finite set $\mathcal{F}_{L'}$ as the set consisting of the families $(u_{ij})_{0 \le i, j \le N-1}$, where $u_{ij} \in \Hom_{\cO_{E}/\mathfrak{p}_{E}}\left(
\overline{L}'_{j}, \overline{L}'_{i}
\right)$.
By fixing bases of the lattices $\overline{L}'_{i}$, we regard $u_{ij}$ as an element of $M_{n'_{i} \times n'_{j}}(\mathcal{O}_{E}/\mathfrak{p}_{E})$.
For $0 \le i \le N-1$, let $\delta_{i} \in M_{n'_{i} \times n'_{i}}(\mathcal{O}_{E}/\mathfrak{p}_{E})$ denote the matrix that represents the Hermitian form $\langle \phantom{x},\phantom{x} \rangle'_{i}$ with respect to the fixed basis of $\overline{L}'_{i}$.
Let $\mathcal{U}_{L'}$ denote the set of elements $(u_{ij})_{0 \le i, j \le N-1} \in \mathcal{F}_{L'}$ that satisfy the following conditions:
\begin{enumerate}
\item 
For each $0 \le i \le N-1$, the endomorphism $u_{ii}$ is the identity map on $\overline{L}'_{i}$.
\item 
For all $0 \le i < j \le N-1$, we have
\[
\Sigma_{i \le k \le j} \overline{\sigma}(\transpose{u_{ki}}) \delta_{k} u_{kj} = 0. 
\]
\end{enumerate}
For $g \in G_{L'}$, we define an element $g_{ij} \in \Hom_{\cO_{E}/\mathfrak{p}_{E}}\left(
\overline{L}'_{j}, \overline{L}'_{i}
\right)$ by the reduction module $\mathfrak{p}_{E}$ of the automorphism $x \mapsto \varpi_{E}^{-\max\{0, j-i\}} g(x)^{L'}_{i}$ for $x \in L'_{j}$ (see Lemma~\ref{lemma:modplowertriangular}).
\begin{prop}
\label{prop:descriptionoftheunipotentradical}
The map $g \mapsto (g_{ij})_{0 \le i, j \le N-1}$ gives a bijection 
\[
G_{L'}^{+} / G_{L'}^{++} \isoarrow \mathcal{U}_{L'}.
\]
\end{prop}
\begin{proof}
If $e = 1$, the proposition follows from \cite{Gan-Yu}*{6.2.1~Proposition}.
More precisely, the quotient $G_{L'}^{+} / G_{L'}^{++}$ agrees with the unipotent radical of the special fibre of the smooth integral model $\underline{G}'$ of $G_{L'}$ constructed in \cite{Gan-Yu}*{Section~5}, which is isomorphic to $\mathcal{U}_{L'}$ endowed with the appropriate group structure.
Suppose that $e = 2$.
In this case, the claim follows from \cite{Gan-Yu}*{6.3.7~Lemma}.
More precisely, the quotient $G_{L'}^{+} / G_{L'}^{++}$ agrees with the group $G^{\dagger}(\cO_{F}/\mathfrak{p}_{F})$ in the proof of \cite{Gan-Yu}*{6.3.7~Lemma}, the set $\mathcal{U}_{L'}$ agrees with the underlying set of the group $G^{\ddagger}(\mathcal{O}_{F}/\mathfrak{p}_{F})$, and the equality $G^{\dagger} = G^{\ddagger}$ is proved in the proof of \cite{Gan-Yu}*{6.3.7~Lemma}.
\end{proof}

According to the definition of $\mathcal{U}_{L'}$, we have the bijection 
\[
\mathcal{U}_{L'} \isoarrow \prod_{0 \le j < i \le N-1} \Hom_{\cO_{E}/\mathfrak{p}_{E}}\left(
\overline{L}'_{j}, \overline{L}'_{i}
\right)
\]
defined by $(u_{ij})_{0 \le i, j \le N-1} \mapsto (u_{ij})_{0 \le j < i \le N-1}$.
Combining this with Proposition~\ref{prop:descriptionoftheunipotentradical}, we obtain a bijection
\[
G_{L'}^{+} / G_{L'}^{++} \isoarrow \prod_{0 \le j < i \le N-1} \Hom_{\cO_{E}/\mathfrak{p}_{E}}\left(
\overline{L}'_{j}, \overline{L}'_{i}
\right).
\]
In particular, we have
\[
\abs{
G_{L'}^{+} / G_{L'}^{++}
} = q_{E}^{\sum_{0 \le j < i \le N-1} n'_{i} \cdot n'_{j}},
\]
where $q_{E} = \abs{\cO_{E}/\mathfrak{p}_{E}}$.
Moreover, according to Proposition~\ref{prop:GLasasubgroupofGL'}, the image of $\left(
G_{L} \cap G_{L'}^{+}\right)/G_{L'}^{++}$ in $\mathcal{U}_{L'}$ via the bijection in Proposition~\ref{prop:descriptionoftheunipotentradical} is the set
\[
\left\{
(u_{ij})_{0 \le i, j \le N-1} \in \mathcal{U}_{L'} \mid u_{N-2, j} \in \Hom_{\cO_{E}/\mathfrak{p}_{E}}\left(
\overline{L}'_{j}, \overline{L}_{N-2}
\right) \quad (0 \le j < N-2)
\right\},
\]
where we regard $\overline{L}_{N-2}$ as a sublattice of $\overline{L}'_{N-2}$ via \eqref{isomoflattices}.
By the same argument as above, we have
\[
\abs{
\left(
G_{L} \cap G_{L'}^{+}\right)/G_{L'}^{++}
} = q_{E}^{s},
\]
where
\[
s \defeq \sum_{0 \le j < i \le N-1, i \neq N-2} n'_{i} \cdot n'_{j} + \sum_{0 \le j < N-2} n_{N-2} \cdot n_{j}.
\]
Thus, we conclude that
\begin{align*}
\abs{G_{L}G_{L'}^{+}/G_{L}}
&=
\abs{G_{L'}^{+}/\left(
G_{L} \cap G_{L'}^{+}
\right)} \\
&=
\abs{
\left(
G_{L'}^{+} / G_{L'}^{++}
\right)/
\left(
\left(
G_{L} \cap G_{L'}^{+}\right)/G_{L'}^{++}
\right)
} \\
&= q_{E}^{\sum_{0 \le j < N-2} n_{j} \cdot (n'_{N-2}-n_{N-2})} \\
&= q_{E}^{\sum_{0 \le j < N-2} n_{j} \cdot n_{N}}.
\end{align*}

Combining this computation with \eqref{isomogisometorygroups}, we obtain the following result.
\begin{thm}
\label{thm:comparisonofvolume}
    We have
\[
\abs{G_{L'}/G_{L}} = 
q_{E}^{\sum_{0 \le j < N-2} n_{j} \cdot n_{N}} \abs{
\mathbf{G}_{L'}(\cO_{F}/\mathfrak{p}_{F})
} \abs{\mathbf{G}_{L}(\cO_{F}/\mathfrak{p}_{F})}^{-1}.
\]
\end{thm}

Now, we repeat the procedure above.
Repeating the procedure to define $L'$ from $L$, we define an $\mathcal{O}_{E}$-lattice $M_{L}$ in $V$ by
\begin{equation}
\label{defofmidpoint}
    M_{L} \defeq \bigoplus_{i = 0}^{N} \mathfrak{p}_{E}^{-\lfloor i/2 \rfloor}L_{i}.
\end{equation}
We write $G_{M_{L}}$ for the isometry groups of $(M_{L}, \langle \phantom{x},\phantom{x} \rangle)$, which we regard as a subgroup of $G_{V}$.
We will see later in Proposition~\ref{prop:MLparahoric} that the pullback of $G_{M_{L}}$ in the simply connected cover of $G_{V}$ is a parahoric subgroup.
We also write $\mathbf{G}_{M_{L}}$ for the maximal reductive quotient of the special fibre of the smooth integral model of $G_{M_{L}}$ constructed in \cite{Gan-Yu}*{Section~5}.
Using Theorem~\ref{thm:comparisonofvolume} inductively, we obtain the following claim.
\begin{thm}
\label{thm:calculationofindexnonparahoricvsparahoric}
We have
\[\abs{G_{M_{L}}/G_{L}} = 
q_{E}^{\sum_{0 \le i <j \le N} \lfloor \frac{j-i-1}{2} \rfloor n_{i} \cdot n_{j}}
\abs{
\mathbf{G}_{{M}_{L}}(\cO_{F}/\mathfrak{p}_{F})
} \abs{\mathbf{G}_{L}(\cO_{F}/\mathfrak{p}_{F})}^{-1}.
\]
\end{thm}

We will also give a variant of Theorem~\ref{thm:calculationofindexnonparahoricvsparahoric} that involves the simply connected cover of $G_{V}$.
Let $\mathbb{G}_{V}$ be the reductive group over $F$ such that $\mathbb{G}_{V}(F) = G_{V}$.
Let $\mathbb{G}_{V, \scn}$ be the simply connected cover of the derived group of $\mathbb{G}_{V}$ and we write $G_{V, \scn} = \mathbb{G}_{V, \scn}(F)$.
We also denote by $G_{L, \scn}$ (resp.\ $G_{M_{L}, \scn}$) the inverse image of $G_{L}$ (resp.\ $G_{M_{L}}$) in $G_{V, \scn}$.
\begin{prop}
\label{prop:MLparahoric}
The group $G_{M_{L}, \scn}$ is a parahoric subgroup of $G_{V, \scn}$.
\end{prop}
\begin{proof}
The claim follows from a similar argument as \cite{Mae24}*{Lemma~5.2}.
We include the sketch of the proof for completeness.
According to the definition of $M_{L}$, we have
\[
M_{L} \subseteq (M_{L})^{\vee} \subseteq \mathfrak{p}_{E}^{-1} M_{L}.
\]
Thus, the group $G_{M_{L}, \scn}$ agrees with the stabiliser in $G_{V, \scn}$ of the self-dual lattice chain
\[
\mathcal{M} = \left(
\cdots \subseteq  M_{L} \subseteq (M_{L})^{\vee} \subseteq \mathfrak{p}_{E}^{-1} M_{L} \subseteq \mathfrak{p}_{E}^{-1} (M_{L})^{\vee} \subseteq \cdots
\right),
\]
which defines a vertex of the Bruhat--Tits building of $G_{V, \scn}$ over $F$.
Since the group $\mathbb{G}_{V, \scn}$ is simply connected, this implies that $G_{M, \scn}$ is a parahoric subgroup of $G_{V, \scn}$ (see also \cite{Mae24}*{Remark~5.1}).
\end{proof}

To rewrite Theorem~\ref{thm:calculationofindexnonparahoricvsparahoric} for $G_{M_{L}, \scn}$ and $G_{L, \scn}$, we define a quantity $Q(L)$ by
\[
Q(L) =
\abs{
G_{M_{L}}/G_{L}
}^{-1}
\abs{
G_{M_{L}, \scn}/G_{L, \scn}
}.
\]
Although $Q(L)$ depends on the group $G_{L}$ and the lattice $L$ in general, in the cases of unitary groups, where we will treat in Section~\ref{section:Estimation of volumes}, we can prove that the quantity $Q(L)$ is always trivial:

\begin{prop}
\label{prop:QL=1}
    Suppose that $E$ is a quadratic extension of $F$ or $F \oplus F$.
    Then we have $Q(L) = 1$.
\end{prop}
\begin{proof}
We note that in this setting, $G_{V, \scn}$ is the group consisting of the elements of $G_{V}$ whose determinants (as linear maps on $V$) are one.
According to \cite{MR4008068}*{Theorem~3.7} and the definition of $M_{L}$, we obtain that the image of the determinant maps on $G_{M_L}$ and $G_{L}$ agree.
Hence, we have $
G_{M_{L}} = G_{M_{L}, \scn} \cdot G_{L}$.
Thus, we obtain that
\[
\abs{
G_{M_{L}}/G_{L}
} = \abs{
G_{M_{L}, \scn} \cdot G_{L}/G_{L}
}
=
\abs{
G_{M_{L}, \scn}/G_{L, \scn}
}.
\qedhere
\]
\end{proof}

Combining Theorem~\ref{thm:calculationofindexnonparahoricvsparahoric} with Proposition~\ref{prop:QL=1}, we obtain the following corollary:

\begin{cor}
\label{corsimplyconnectedverofindex}
We have
\[\abs{G_{M_{L}, \scn}/G_{L, \scn}} = 
q_{E}^{\sum_{0 \le i <j \le N} \lfloor \frac{j-i-1}{2} \rfloor n_{i} \cdot n_{j}}
\abs{
\mathbf{G}_{{M}_{L}}(\cO_{F}/\mathfrak{p}_{F})
} \abs{\mathbf{G}_{L}(\cO_{F}/\mathfrak{p}_{F})}^{-1} Q(L).
\]
If $E$ is a quadratic exnetsion of $F$ or $F \oplus F$, we have
\[\abs{G_{M_{L}, \scn}/G_{L, \scn}} = 
q_{E}^{\sum_{0 \le i <j \le N} \lfloor \frac{j-i-1}{2} \rfloor n_{i} \cdot n_{j}}
\abs{
\mathbf{G}_{{M}_{L}}(\cO_{F}/\mathfrak{p}_{F})
} \abs{\mathbf{G}_{L}(\cO_{F}/\mathfrak{p}_{F})}^{-1}.
\]
\end{cor}

\subsection{A volume formula for non-principal arithmetic subgroups}
\label{subsec:generalizationofprasad}

In this subsection, we come back to the global setting in \S\ref{subsec:reviewofprasad} and will prove the main result of this section.
Hence, $k$ is a global field.
Let $\epsilon \in \{\pm 1\}$ and let $(E, \sigma)$ be one of the following $k$-algebras with involution:
\begin{enumerate}
\item $E = k$ and $\sigma = \id_{k}$; 
\item $E$ is a quadratic extension of $k$ and $\sigma$ is the unique non-trivial automorphism of $E/k$;
\item 
$E = D$, the quaternion division algebra over $k$ and $\sigma$ is the standard involution.
\end{enumerate}
Let $L$ be a finitely generated $\cO_{E}$-lattice equipped with a $(\sigma, \epsilon)$-Hermitian form
\[
\langle \phantom{x},\phantom{x} \rangle \colon L \times L \rightarrow \mathcal{O}_{E}.
\]
We assume that $V := L \otimes_{\mathcal{O}_{k}} k$ is non-degenerate with respect to $\langle \phantom{x},\phantom{x} \rangle$.
We define an $\cO_{E_{v}}$-lattice $L_{v}$ by $L_{v} \defeq L \otimes_{\cO_{k}} \cO_{k_{v}}$ and let $\langle \phantom{x}, \phantom{x} \rangle_{v}$ denote the Hermitian form on $L_{v}$ induced from $\langle \phantom{x},\phantom{x} \rangle$.
We fix a Jordan splitting
\[
L_{v} = \bigoplus_{i \ge 0} L_{v, i}
\]
and let $n_{v, i}$ denote the rank of $L_{v, i}$.
By using this, we define the $\cO_{E_{v}}$-lattice $M_{L_{v}}$ as \eqref{defofmidpoint}.
We also define a subset $V_{\np}$ of $V_{k} \smallsetminus S$ by
\[
V_{\np} = \left\{
v \in V_{k} \smallsetminus S \mid L_{v} \neq M_{L_{v}}
\right\}.
\]

Let $G$ be the reductive group over $k$ such that $G(k)$ is the isometry group of $(V, \langle \phantom{x},\phantom{x} \rangle)$ and
let $G_{\scn}$ be the simply connected cover of the derived group of $G$.
We suppose that the group $(G_{\scn})_{S} = \prod_{v \in S} G_{\scn}(k_v)$ is non-compact as in \S\ref{subsec:reviewofprasad}.
For each $v \in V_{k} \smallsetminus S$, let $G_{L_{v}}$ (resp.\ $G_{M_{L_v}}$) denote the isometry group of $(L_{v}, \langle \phantom{x},\phantom{x} \rangle_{v})$ (resp.\ $(M_{L_{v}}, \langle \phantom{x},\phantom{x} \rangle_{v})$), and we regard them as subgroups of $G(k_{v})$.
We define a compact, open subgroup $K_{v}$ (resp.\ $K'_{v}$) of $G_{\scn}(k_{v})$ as the inverse image of $G_{L_{v}}$ (resp.\ $G_{M_{L_v}}$) in $G_{\scn}(k_{v})$.
By using them, we define compact, open subgroups $K$ and $K'$ of $G_{\scn}(\mathbb{A}_{S})$ by $K = \prod_{v \in V_{k} \smallsetminus S} K_{v}$ and $K' = \prod_{v \in V_{k} \smallsetminus S} K'_{v}$.
Let $\Gamma$ and $\Gamma'$ be the $S$-arithmetic subgroups of $G_{\scn}(k)$ associated to $K$ and $K'$, respectively.
According to Proposition~\ref{prop:MLparahoric}, the $S$-arithmetic subgroup $\Gamma'$ is principal.

For $v \in V_{k} \smallsetminus S$, we define
\[
\Ind(L_{v}) \defeq \abs{K'_{v}/K_{v}}.
\]
By combining Prasad's volume formula Theorem~\ref{Prasad'svolumeformula} with the explicit computation of $\Ind (L_v)$ in Subsection~\ref{subseccomparisonofvolumes}, we obtain the following main result of this section:

\begin{thm}
\label{thm:generalization of Prasad's volume formula}
We have
\[
\mu_{S}
\left(
(G_{\scn})_{S}/\iota_{S}(\Gamma)
\right)
= D_{k}^{\frac{1}{2} \dim G_{\scn}} \left(
\frac{D_{\ell}}
{D_{k}^{[\ell:k]}}
\right)^{\frac{1}{2} \mathfrak{s}(G_{\scn, \qs})}
\left(
\prod_{v \in V_{\infty}}
\abs{
\prod_{i=1}^{r} \frac{m_{i}!}{(2 \pi)^{m_{i} + 1}}
}_{v} 
\right)\mathscr{E}_{\Gamma'} \left(
\prod_{v \in V_{\np}} \Ind(L_{v})
\right),
\]
where the notation in the right hand side is explained in Theorem~\ref{Prasad'svolumeformula}.
If $\epsilon$ and $E_{v}/k_{v}$ satisfy Assumption~\ref{asump=2}, then we have 
\[
\Ind(L_{v}) = q_{E_v}^{\sum_{i <j} \lfloor \frac{j-i-1}{2} \rfloor n_{v, i} \cdot n_{v, j}}
\abs{
\mathbf{G}_{{M}_{L_{v}}}(\mathfrak{f}_{v})
} \abs{\mathbf{G}_{L_{v}}(\mathfrak{f}_{v})}^{-1} Q(L_{v}),
\]
where we use the notation in \S\ref{subseccomparisonofvolumes}.
Moreover, if $E$ is further a quadratic extension of $k$, we have
\[
\Ind(L_{v}) = q_{E_v}^{\sum_{i < j} \lfloor \frac{j-i-1}{2} \rfloor n_{v, i} \cdot n_{v, j}}
\abs{
\mathbf{G}_{{M}_{L_{v}}}(\mathfrak{f}_{v})
} \abs{\mathbf{G}_{L_{v}}(\mathfrak{f}_{v})}^{-1}.
\]
\end{thm}
\begin{proof}
According to Corollary~\ref{cor:comparisonofcovoleme}, we have
\[
    \mu_{S}
\left(
(G_{\scn})_{S}/\iota_{S}(\Gamma)
\right)
= \mu_{S}
\left(
(G_{\scn})_{S}/\iota_{S}(\Gamma')
\right)
        \prod_{v \in V_{k} \smallsetminus S} \frac{
\mu_{v}(K'_{v})
}{
\mu_{v}(K_{v})
}
= \mu_{S}
\left(
(G_{\scn})_{S}/\iota_{S}(\Gamma')
\right)
        \prod_{v \in V_{\np}} 
        \Ind(L_{v})
        .
\]
Hence, the first claim follows from Theorem~\ref{Prasad'svolumeformula}.
The second and last claims follow from Corollary~\ref{corsimplyconnectedverofindex}.
\end{proof}

\subsection{The case of special unitary groups}
\label{subsec:thecaseofspecialunitary}
In this subsection, we will apply the main result of the previous subsection to the case where $G_{\scn}$ is a special unitary group over $\bQ$.
In this subsection, we suppose that $k= \bQ$ and $E$ is an imaginary quadratic extension of $\bQ$.
We take the finite subset $S \subset V_{k}$ as $S = V_{\infty}$. 
We suppose that the Hermitian lattice $L$ over $\cO_E$ has rank $n+1$.
The reason why we use $n+1$ instead of $n$ is that we will assume that $L$ has signature $(1,n)$ in Section~\ref{section:Estimation of volumes}.
For $v \in V_{\fin}$, we introduce the associated quantity $\lambda(L_{v})$ of $L_{v}$ as follows.
We first define $\lambda(M_{L_{v}})$ of $M_{L_{v}}$ as
\begin{equation}
\label{defoflambdaMLv}
\lambda(M_{L_{v}})
\defeq
\begin{cases}
\displaystyle
    q_v^{(\dim \mathbf{G}_{{M}_{L_{v}}}-n)/2}\cdot\abs{\mathbf{G}_{M_{L_{v}}}(\mathfrak{f}_{v})}^{-1}\cdot\prod_{i=2}^{n+1}(q_v^i-(-1)^i) &(v:\mathrm{inert}),\\
    \displaystyle
    q_v^{(\dim \mathbf{G}_{{M}_{L_{v}}}-n)/2}\cdot\abs{\mathbf{G}_{M_{L_{v}}}(\mathfrak{f}_{v})}^{-1}\cdot\prod_{i=2}^{n+1}(q_v^i-1) &(v:\mathrm{split}), \\
    \displaystyle
    q_v^{(\dim \mathbf{G}_{{M}_{L_{v}}}-\lfloor (n+1)/2 \rfloor)/2}\cdot\abs{\mathbf{G}_{M_{L_{v}}}(\mathfrak{f}_{v})}^{-1}\cdot\prod_{i=1}^{\lfloor n+1/2 \rfloor}(q_v^{2i}-1) &(v:\mathrm{ramify}).
\end{cases}
\end{equation}
We then extend this definition  to general $L_{v}$, not necessarily coinciding with $M_{L_v}$, as
\[
    \lambda(L_{v}) \defeq \lambda(M_{L_{v}}) \cdot \Ind(L_{v}).
\]

\begin{rmk}
The definition of $\lambda(L_{v})$ \emph{a priori} depends on the choice of a Jordan splitting of $L$.
However, we can check by using \cite{Omeara2000introduction}*{91.9.Theorem} that the quantity $\lambda(L_{v})$ only depends on the lattice $L_{v}$ (see also Lemma~\ref{lem:rankofLJDwelldef} and Remark~\ref{rmkaboutindepofreductivequot}).
\end{rmk}

\begin{thm}
\label{thm:prasad'svolumeformulaspecialunitarycase}
    We assume that the discriminant $-D$ of $E$ is odd.
    For any Hermitian lattice $L$ of rank $n+1$, we have
    \[    \mu_{\infty}
\left(
\SU(L \otimes_{\bZ} \bR)/\SU(L)
\right)=
{D^{\lfloor \frac{n}{2}\rfloor(\lfloor \frac{n}{2}\rfloor + \frac{3}{2})}\prod_{i=1}^n\frac{i!}{(2\pi)^{i+1}}\prod_{i=1}^{\lfloor \frac{n+1}{2} \rfloor}\zeta(2i)\prod_{i=1}^{\lfloor \frac{n}{2} \rfloor}L_{E/\bQ}(2i+1)\prod_{v\nmid\infty}\lambda (L_{v})},\]
where $\zeta$ denotes the Riemann zeta-function, and $L_{E/\bQ}$ denotes the Dirichlet $L$-function associated with the quadratic extension $E/\bQ$.
\end{thm}
\begin{proof}
We substitute $k = \bQ$ and $\Gamma = \SU(L)$ to Theorem~\ref{thm:generalization of Prasad's volume formula}.
Then we have $\ell = E$, $D_{k} = 1$, $D_{\ell} = D$.
According to \cite{Pra89}*{\S0.4}, we have
\[
\mathfrak{s}(H_{\qs}) = \begin{cases}
n(n+3)/2 & (2 \mid n), \\
(n-1)(n+2)/2 & (2 \nmid n).
\end{cases}
\]
We also obtain from \cite{Pra89}*{\S1.5} that $r = n$ and $\{m_{1}, m_{2}, \ldots, m_{r}\} = \{1, 2, \ldots, n\}$.
Moreover, according to \cite{KalethaPrasad}*{Proposition~18.5.10, \S18.5.11} combined with the explicit calculations of maximal reductive quotients in \cite{prasad2012nonexistence}*{Subsection 2.4}, we obtain that
\[
\mathscr{E}_{\Gamma'} = 
\begin{cases}
\zeta(2)L_{E/\bQ}(3)\zeta(4)\dots L_{E/\bQ}(n+1)\prod_{v\nmid\infty}\lambda (M_{L_v}) & (2 \mid n), \\
\zeta(2)L_{E/\bQ}(3)\zeta(4)\dots \zeta(n+1)\prod_{v\nmid\infty}\lambda(M_{L_v})\cdot & (2\nmid n).
\end{cases}
\]
Thus, the claim follows from Theorem~\ref{thm:generalization of Prasad's volume formula} and the definition of $\lambda(L_{v})$.
\end{proof}

\section{Non-freeness of graded algebras of modular forms on complex balls}
\label{section:Estimation of volumes}
From now on, as in Subsection~\ref{subsec:thecaseofspecialunitary}, we assume that $k= \mathbb{Q}$, $E$ is an imaginary quadratic field with odd discriminant $-D$, and $S = V_{\infty}$.
.
In the remainder of this paper, we mainly focus on the case where $E\neq\bQ(\sqrt{-3})$ for simplicity.
Although this restriction avoids certain technical complications, analogous arguments apply in this exceptional case as well, at the cost of additional notation and effort.

Recall that $L$ is an $\cO_{E}$-lattice equipped with a Hermitian form $\langle \phantom{x},\phantom{x} \rangle$.
We assume that $L$ has signature $(1,n)$ for $n > 2$.
Let $\U(L)$ denote the unitary group attached to the Hermitian lattice $L$.
In this paper, we refer to a finite index subgroup $\Gamma$ of $\U(L)$ as an \emph{arithmetic subgroup} of $\U(L)$.
More generally, we call a subgroup $\Gamma < \U(1, n)$ an arithmetic subgroup if it is a finite index subgroup of $\U(L)$ for some Hermitian lattice $L$ of signature $(1, n)$.

In this section, we will prove the main results of this paper.
In Theorem~\ref{thm:not_wps}, we show that the graded algebra of modular forms $M_*(\Gamma)$ for any arithmetic subgroup $\Gamma$ of the unitary group $\U(L)$ is never free if $n > 99$ (except over $\mathbb{Q}(\sqrt{-3})$, where we require $n > 154$) or $D$ is sufficiently large. 
We also prove in Theorem~\ref{thm:finiteness} that up to scaling, there are only finitely many isometry classes of Hermitian lattices $L$ such that $M_*(\Gamma)$ is free for some arithmetic subgroup $\Gamma$ of $\U(L)$.

\subsection{Modular forms and ramifications}

In this subsection, we briefly recall the definition of the graded algebra of modular forms.

Let 
    \[\bB^n\defeq\{[v]\in \bP(L\otimes_{\cO_E}\bC)\mid \langle v,v\rangle>0 \}\]
be the $n$-dimensional complex ball acted on by the unitary group $\U(L)$.
For an arithmetic subgroup $\Gamma<\U(L)$, let $X_{\Gamma}\defeq \Gamma\backslash\bB^n$ be the ball quotient.
Let $\chi:\Gamma\to\bC^{\times}$ be a character of an arithmetic subgroup $\Gamma<\U(L)$.
We say that a holomorphic function $f$ on a principal $\bC^{\times}$-bundle 
\[\bB^0\defeq\{v\in L\otimes_{\cO_{E}}\bC\mid [v]\in\bB^n\}\]
is a \emph{modular form of weight $\kappa\in\bZ_{\geq 0}$ for $\Gamma$ with character $\chi$} if the following conditions  hold:
\begin{align*}
    f(tz)=t^{-\kappa}f(z),\quad f(\gamma z)=\chi(\gamma)f(z)
\end{align*}
for all $t\in\bC^{\times}$ and $\gamma\in\Gamma$.
We denote by $M_{\kappa}(\Gamma,\chi)$ the $\bC$-vector space of modular forms of weight $\kappa$ for $\Gamma$ with a character $\chi$.
We write $M_{\kappa}(\Gamma)\defeq M_{\kappa}(\Gamma,\mathrm{triv})$ for the space with trivial character and define the graded algebra of modular forms
\[M_*(\Gamma)\defeq\bigoplus_{\kappa\geq 0}M_{\kappa}(\Gamma).
\]
\begin{rmk}
\label{rmk:bb_compactification}

     Here,  we recall the relationship between the graded algebra of modular forms $M_{*}(\Gamma)$ and the geometry of the ball quotient $\Gamma\backslash\bB^n$.
     We briefly review the Baily-Borel compactification \cite{baily1966compactification}, which provides the minimal compactification.
     For simplicity, assume that $\Gamma$ is neat.
     \begin{enumerate}
         \item  
         For any rank 1 primitive isotropic sublattice $I\subset L$, let $C_I\defeq \bP(I\otimes E)$ denote the associated rational 0-dimensional cusp.
    We define the rational compactification of $\bB^n$ as 
    \[\overline{\bB^n}\defeq\bB^n\cup \left(
    \bigcup_I C_I
    \right)
    \]
    endowed with a topology defined by Siegel sets.
    By the theorem of Baily and Borel \cite{baily1966compactification}, the group $\Gamma$ acts on $\overline{\bB^n}$ discretely.
    Consequently, the Baily-Borel compactification of $\Gamma\backslash\bB^n$ is defined by  
\[\overline{\Gamma\backslash\bB^n}\defeq\Gamma\backslash\overline{\bB^n}\]
which naturally admits the structure of a projective variety over $\bC$.
\item By the above construction, there is a specific ample line bundle, called the \emph{automorphic line bundle}, whose global sections are modular forms of weight 1, on $\overline{\Gamma\backslash\bB^n}$.
This yields the following isomorphism as algebraic varieties 
\[\overline{\Gamma\backslash\bB^n}\cong\Proj M_*(\Gamma).\]
Hence, if the algebra $M_*(\Gamma)$ is freely generated by modular forms of weights $k_0,\cdots,k_{n}$, then the Baily-Borel compactification is isomorphic to the weighted projective space
\[\Proj M_*(\Gamma)\cong\bP(k_0,\cdots,k_{n}).\]     
\end{enumerate}
\end{rmk}

As explained above, the main result of this section is that the graded algebra $M_*(\Gamma)$ is not a free algebra if $n$ or $D$ is sufficiently large.
A key ingredient is the criterion Theorem~\ref{thm:criterion_wps} for the non-freeness of the graded algebras of modular forms.
We recall the notion of ramification divisors for the uniformisation map 
\[
\bB^n\to X_{\Gamma}= \Gamma\backslash\bB^n,
\]
which will be used to prove Theorem~\ref{thm:criterion_wps}.
Let $l \in L$ be a primitive vector satisfying $\langle l, l\rangle<0$. We define the \emph{reflection} $\sigma_{l}$ by 
\begin{align}
\label{def:reflection}
    \sigma_{l}: L\otimes_{\cO_{E}} E\to L\otimes_{\cO_{E}} E,\quad v\mapsto v-2\frac{\langle v,l\rangle}{\langle l,l\rangle}l.
\end{align}
By \cite{behrens2012singularities}*{Proposition~2, Corollary~3}, all ramification divisors of $\mathbb{B}\to \Gamma\backslash\bB^n$ arise as the fixed divisors by such reflections; while the case $E=\bQ(\sqrt{-3})$ involves additional considerations (triflections and hexaflections) concerning $\mathcal{O}_E^{\times}$, we omit these from the present discussion.
Let $H_{l}\subset\bB^n$ be the ramification divisor associated with $l$.
Such a divisor $H_{l}$ is called a \emph{Heegner divisor} and  admits a structure of the complex subball $\bB^{n-1}$ defined by the sublattice $L^{l}\defeq l^{\perp}\cap L$ of signature $(1,n-1)$.
In this paper, we say that a primitive vector $l\in L$ is \emph{$\Gamma$-reflective} if $z\cdot \sigma_{l}\in \Gamma$ for an element $z$ in the centre of $\U(L)$.
This condition implies that $L^{l}$ defines a ramification divisor.
We say that a reflective vector $l\in L$ is \emph{split} if the Hermitian lattice $L$ decomposes as $L=L^{l}\oplus l\cO_E$.
Otherwise, we call it \emph{non-split}.
For a $\Gamma$-reflective vector $l$, we denote by $B_{l}\subset \Gamma\backslash\bB^n$ the corresponding branch divisors.
These divisors are classified according to properties of the reflective vectors, as described in \cite{Mae24}*{Section~3}, but we omit the details here.
To handle the case of $E=\bQ(\sqrt{-3})$ collectively in what follows, we denote by $r_l$ the ramification index associated with a reflective vector $l$.
Note that $r_l \in\{2,3,6\}$ if $E = \bQ(\sqrt{-3})$ and $r_l =2$ otherwise.

We denote by $Z$ the centre of $\U(L)$, and for an arithmetic subgroup $\Gamma < \U(L)$, 
let $Z\Gamma$ denote the group generated by $\Gamma$ and $Z$.
Let $\mathcal{R}^{\Gamma}$ denote the set consisting of $Z\Gamma$-equivalence classes of $\Gamma$-reflective vectors in $L$.
Let $\mathcal{R}^{\Gamma}_{\mathrm{s}}$ be the subset of $\mathcal{R}^{\Gamma}$ consisting of split vectors.
For a $\Gamma$-reflective $l \in L$, we denote by $\Gamma_{l}$ (resp.\ $\Gamma_{l \cO_{E}}$) the stabiliser of $l$ (resp.\ $l \cO_{E}$) in $\Gamma$.
Then the elements of $\Gamma_{l \cO_{E}}$ preserve the subspace $L^{l}$, and the restriction defines a group homomorphism $\res_{l} \colon \Gamma_{l \cO_{E}} \to \U(L^l)$.
We define an arithmetic subgroup $\Gamma^{l}$ of $\U(L^{l})$ by $\Gamma^{l} \defeq \res_{l}\left(\Gamma_{l \cO_{E}}\right)$.
Since the restriction map $\res_{l}$ is injective on $\Gamma_{l}$, we also regard $\Gamma_{l}$ as an arithmetic subgroup of $\U(L^{l})$ that is contained in $\Gamma^{l}$.

\begin{rmk}
\label{rmkULi}
If $l$ is a split vector, the map $\res_{l} \colon \U(L)_{l \cO_{E}} \to \U(L^{l})$ is surjective on $\U(L)_{l}$.
Thus, we have $\U(L)_{l} = \U(L)^{l} = \U(L^{l})$ in this case.
\end{rmk}

At the end of this subsection, we record a lemma about the groups $\Gamma_{l \cO_{E}}$ and $\Gamma^{l}$, which will be used in the following subsection.
\begin{lemma}
\label{lem:descriptionofGammal}
We have $Z \cdot (Z \Gamma)_{l} = (Z \Gamma)_{l \cO_{E}} = Z \cdot \Gamma_{l \cO_{E}}$.
As subgroups of $\U(L^{l})$, we have $Z' \cdot (Z \Gamma)_{l} = Z' \cdot \Gamma^{l}$, where $Z'$ denotes the centre of $\U(L^l)$.
\end{lemma}
\begin{proof}
The first claim follows from the definitions of $\Gamma_{l}$ and $\Gamma_{l \cO_{E}}$.
The second claim follows by sending $Z \cdot (Z \Gamma)_{l}$ and $Z \cdot \Gamma_{l \cO_{E}}$ via the restriction map $\res_{l}$.
\end{proof}

\subsection{A criterion for non-freeness}
\label{subsec:a criterion}
In this subsection, we establish a criterion for the graded algebra $M_*(\Gamma)$ of modular forms to be non-free.
This criterion reduces the problem of determining the non-freeness of $M_*(\Gamma)$ to computing the covolumes of certain unitary groups.

For a neat arithmetic subgroup $\Gamma<\U(L)$,
we define the \emph{Hirzebruch-Mumford volume} to be
\[\volHM(\Gamma)\defeq \frac{\abs{\chi(X_{\Gamma})}}{n+1}\]
Here, $\chi(X_{\Gamma})$ denotes the Euler-Poincar\'{e} characteristic of $X_{\Gamma}$ and the factor $n+1$ comes from the one for the compact dual $\bP^n$ of $\bB^n$.
We extend this definition for any arithmetic subgroup $\Gamma$ by taking a neat  subgroup $\Gamma_{\mathrm{n}}$ and dividing $\volHM(\Gamma_{\mathrm{n}})$ by $[Z\Gamma:Z\Gamma_{\mathrm{n}}]$.
We record basic properties of the Hirzebruch-Mumford volume.
\begin{lemma}
\label{lem:compatibility of volHM and covolume}
Let $\Gamma<\U(L)$ be an arithmetic subgroup.
\begin{enumerate}
    \item  If $\Gamma'<\Gamma$ is another arithmetic subgroup, then we have 
\[\volHM(\Gamma') = [Z\Gamma:Z\Gamma']\volHM(\Gamma).\]
\item 
When $\Gamma<\SU(L)$, we have
\[
\mu_{\infty}(\SU(1,n)/\Gamma) = \frac{\volHM(\Gamma)}{\abs{\Gamma\cap Z}}, 
\]
where $Z$ denotes the centre of $\U(L)$ as introduced in the previous subsection.
\end{enumerate}
\begin{proof}
    (1) Taking a near subgroup $\Gamma_{_{\mathrm{n}}}<\Gamma'$, we have two relations
    \[\volHM(\Gamma) = \frac{\volHM(\Gamma_{\mathrm{n}})}{[Z\Gamma:Z\Gamma_{\mathrm{n}}]},\quad \volHM(\Gamma') = \frac{\volHM(\Gamma_{\mathrm{n}})}{[Z\Gamma':Z\Gamma_{\mathrm{n}}]}.\]
    The transitivity of the index shows the claim.

    (2)  
    First, while it is well-known, for completeness, we show that $\mu_{\infty}(\SU(1,n)/\Gamma_{\mathrm{n}}) = \volHM(\Gamma_{\mathrm{n}})$ for a neat arithmetic subgroup $\Gamma_{\mathrm{n}}$; see also \cite{KalethaPrasad}*{\S 18.6}.
There is the Euler-Poincar\'{e} measure $\mu^{\mathrm{EP}}$ on $\SU(1,n)$ by \cite{harder1971gauss} for the non-compact arithmetic varieties cases; see \cite{serre1971cohomology}*{Section~3}.
This measures the Euler-Poincar\'{e} characteristic of the arithmetic subgroup $\Gamma_{\mathrm{n}}$, that is, $\mu^{\mathrm{EP}}(\Gamma_{\mathrm{n}})= \abs{\chi(\Gamma_{\mathrm{n}})}$.
Then, by Hirzebruch's proportionality principle for non-compact cases \cite{mumford1977hirzebruch}, we have
\[\mu^{\mathrm{EP}}(\Gamma_{\mathrm{n}}) = \chi(\bP^n)\cdot \mu_{\infty}(\SU(1,n)/\Gamma_{\mathrm{n}}).\]
Now, since $\bB^n$ is contractible, it is the universal cover of $X_{\Gamma_{\mathrm{n}}}$, which forces
$\abs{\chi(\Gamma_{\mathrm{n}})} = \abs{\chi(X_{\Gamma_{\mathrm{n}}})}$, combined with the standard argument of the topology theory; see \cite{serre1971cohomology}*{Section 1}.
Here, we used the fact that $\Gamma_{\mathrm{n}}$ acts on $\bB^n$ freely.
Then, the definition of $\volHM(\Gamma_{\mathrm{n}})$ shows the claim.

Now, we work on $\Gamma$ in general.
Take a neat subgroup $\Gamma_{\mathrm{n}}<\Gamma$.
Then we obtain from the definition of the covolumes that
    \[
    \mu_{\infty}(\SU(1,n)/\Gamma) = \frac{\mu_{\infty}(\SU(1,n)/ \Gamma_{\mathrm{n}})}{[\Gamma:\Gamma_{\mathrm{n}}]}
    \]
    Combining this with the equality $\mu_{\infty}(\SU(1,n)/\Gamma_{\mathrm{n}}) = \volHM(\Gamma_{\mathrm{n}})$ proved above, we have
    \[
    \mu_{\infty}(\SU(1,n)/\Gamma) = \frac{\volHM(\Gamma_{\mathrm{n}})}{[\Gamma:\Gamma_{\mathrm{n}}]} = \frac{[Z\Gamma:Z\Gamma_{\mathrm{n}}]}{[\Gamma:\Gamma_{\mathrm{n}}]} \volHM(\Gamma).
    \]
    Since $\Gamma_{\mathrm{n}}$ is a neat arithmetic subgroup, it does not contain any non-trivial element in the centre of $\U(L)$.
    Thus, we have
    \[
    \frac{[Z\Gamma:Z\Gamma_{\mathrm{n}}]}{[\Gamma:\Gamma_{\mathrm{n}}]} = \frac{[\Gamma:(\Gamma \cap Z) \cdot \Gamma_{\mathrm{n}}]}{[\Gamma:\Gamma_{\mathrm{n}}]} 
    = [(\Gamma \cap Z) \cdot \Gamma_{\mathrm{n}} : \Gamma_{\mathrm{n}}]^{-1}
    = [\Gamma \cap Z : \Gamma_{\mathrm{n}} \cap Z]^{-1}
    = \abs{\Gamma \cap Z}^{-1},
    \]
    which concludes the proof.
\end{proof}

\end{lemma}

We also record a corollary of Lemma~\ref{lem:compatibility of volHM and covolume}, which will be used to restate the criterion in Theorem~\ref{thm:criterion_wps} in terms of covolumes in $\SU(1, n)$, rather than the Hirzebruch–Mumford volume. 

\begin{cor}
\label{corcovolofSUvsHMvolofU}
We have
\[
\mu_{\infty}(\SU(1,n)/\SU(L)) \le \volHM(\U(L)) \le 2 \mu_{\infty}(\SU(1,n)/\SU(L))
\]
when $E \neq \bQ(\sqrt{-3})$.
When $E = \bQ(\sqrt{-3})$, we have
\[
\mu_{\infty}(\SU(1,n)/\SU(L)) \le \volHM(\U(L)) \le 6 \mu_{\infty}(\SU(1,n)/\SU(L)).
\]
\end{cor}
\begin{proof}
According to Lemma~\ref{lem:compatibility of volHM and covolume}, we have
\[
\mu_{\infty}(\SU(1,n)/\SU(L)) = \frac{\volHM(\SU(L))}{\abs{\SU(L) \cap Z}}
= \frac{[\U(L) : Z \SU(L)]}{\abs{\SU(L) \cap Z}} \volHM(\U(L)).
\]
We can compute the factor $[\U(L) : Z \SU(L)]/\abs{\SU(L) \cap Z}$ as
\begin{align*}
\frac{[\U(L) : Z \SU(L)]}{\abs{\SU(L) \cap Z}} &= \frac{[\U(L): \SU(L)]}{
[Z \SU(L) : \SU(L)] \abs{\SU(L) \cap Z}
} \\
&= \frac{[\U(L): \SU(L)]}{
[Z : \SU(L) \cap Z] \abs{\SU(L) \cap Z} 
} \\
&= [\U(L): \SU(L)]/\abs{Z}.
\end{align*}
Since the image of the determinant map on $\U(L)$ is contained in the set
\[
\{
x \in \cO_{E}^{\times} \mid x \sigma(x) = 1
\} = \begin{cases}
\{1, -1\} & (E \neq \bQ(\sqrt{-3})), \\
\{\zeta \in \bC^{\times} \mid \zeta^6 = 1\} & (E=\bQ(\sqrt{-3})),
\end{cases}
\]
we have
\[
1 \le [\U(L): \SU(L)] \le
\begin{cases}
2 & (E \neq \bQ(\sqrt{-3})), \\
6 & (E = \bQ(\sqrt{-3})).
\end{cases}
\]
We also have
\[
\abs{Z} = \begin{cases}
2 & (E \neq \bQ(\sqrt{-3})), \\
6 & (E = \bQ(\sqrt{-3})).
\end{cases}
\]
Combining them, we obtain the claim.
\end{proof}

We now prove the criterion for the graded algebra $M_*(\Gamma)$ of modular forms to be non-free, in terms of the Hirzebruch-Mumford volume.

\begin{thm}
\label{thm:criterion_wps} 
For an arithmetic subgroup $\Gamma<\U(L)$, if the inequality 
    \[
    \sum_{[l]\in \mathcal{R}^{\Gamma}}\frac{r_l-1}{r_l}\frac{\volHM(\Gamma^{l})}{\volHM(\Gamma)}
    < 2(n+1)
    \]
    holds, then $M_*(\Gamma)$ is not free.
\end{thm}
\begin{proof}
Suppose that $M_*(\Gamma)$ is freely generated by modular forms of weights $k_1,\cdots,k_{n+1}$. 
Then, by \cite{WW2021}*{Theorem~3.3}, there exists a cusp form $F$ of weight $n+1+k_1+\cdots+k_{n+1}$ for $\Gamma$, whose divisor coincides with $\sum_{[l] \in \mathcal{R}^{\Gamma}}  (r_l-1)H_{l}$.
Passing to the quotient, we have the relation
\[\div_{X_{\Gamma}}(F) = \sum_{[l]\in\mathcal{R}^{\Gamma}} \frac{r_l-1}{r_l} B_{l}\]
in $\Pic(X_{\Gamma})\otimes\bQ$.
The coefficient $(r_l-1)/r_l$ is called the standard coefficient caused by the ramification.

Now, we recall the setting \cite{bruinier2004} for the unitary case $\U(1,n)$.
Let us take the line bundle $\cO_{\bP^n}(1)$ on the compact dual $\bP^n$ of $\bB^n$.
It defines a $\bQ$-line bundle $\mathcal{M}$ on $X_{\Gamma}$.
We denote by $\Omega \defeq c_1(\mathcal{M})$ and $\widetilde{\Omega}\defeq c_1(\cO_{\bP^n}(1))$ the first Chern classes.
Putting 
\[\deg(\div_{X_{\Gamma}}(F)) \defeq \sum_{[l]\in\mathcal{R}^{\Gamma}}\frac{r_l-1}{r_l}\int_{B_l} \Omega^{n-1},\quad \vol(X_{\Gamma}) \defeq \int_{X_{\Gamma}} \Omega^n,\]
a special case of \cite{bruinier2004}*{Theorem 1}, a non-compact analogue of an application of the Poincaré-Lelong formula, asserts that
\[\deg(\div_{X_{\Gamma}}(F)) = (n + 1 + k_1 + \cdots + k_{n+1})\vol(X_{\Gamma}).\]
It is worth noting that this equation can also be deduced by considering a neat cover and applying Bruinier’s theorem there.
Viewing $H_l$ as a sub-ball of $\bB^n$, the inclusion naturally extends to the corresponding compact duals $\bP^{n-1}\hookrightarrow \bP^n$, and the restriction of $\cO_{\bP^n}(1)$ coincides with $\cO_{\bP^{n-1}}(1)$.
Hence, using the same notation $\widetilde{\Omega}$ as $\cO_{\bP^{n-1}}$ for simplicity, Hirzebruch's proportionality principle shows that 
\[\volHM(\Gamma) = \left(\int_{X_{\Gamma}}\Omega^n\right)\left(\int_{\bP^n} \widetilde{\Omega}^{n}\right)^{-1},\quad \volHM(\Gamma^{l}) = \left(\int_{\Gamma^{l}\backslash\bB^{n-1}}\Omega^n\right)\left(\int_{\bP^{n-1}} \widetilde{\Omega}^{n-1}\right)^{-1}.\]
When considering the integral, we consider the group $\Gamma/(\Gamma\cap Z)$ in such a way that the action of the centre is ignored.
This allows us to extend Bruinier's formula to the case where the arithmetic subgroup acts non-freely.
Since we are considering the volume form defined by $\cO_{\bP^n}(1)$ and $\cO_{\bP^{n-1}}(1)$, direct computation shows 
\[\int_{\bP^n} \widetilde{\Omega}^{n} =  \int_{\bP^{n-1}} \widetilde{\Omega}^{n-1} = 1.\]
Since $\Gamma^l\backslash H_l \to B_l$ gives the normalization of $B_l$ as in \cite{ma2013finiteness}*{Subsection 3.2}, we finally obtain that
\[\sum_{[l]\in\mathcal{R}^{\Gamma}}\frac{r_l-1}{r_l}\volHM(\Gamma^{l}) = (n+1 + k_1+\cdots + k_{n+1})\volHM(\Gamma). \]
Since $k_i\geq 1$ for each $i$,, we conclude that
\[
\sum_{[l]\in\mathcal{R}^{\Gamma}} \frac{r_l-1}{r_l}\frac{\volHM(\Gamma^{l})}{\volHM(\Gamma)} = n+1+k_1+\cdots+k_{n+1} \ge 2(n+1),
\]
which contradicts the assumption.

See also \cite{stuken2017nonfreeness}*{Subsection 3.2} for the proof of the case of Hilbert modular forms.
\end{proof}
\begin{rmk}
\label{rmk:special_reflective_modular_form}
    \begin{enumerate}
    \item While Bruinier's formula is originally stated in the context of $\O^+(2,n)$, analogous computations confirm that it also applies to the unitary case $\U(1,n)$; see also \cite{stuken2022nonfreeness}*{Theorem A}.
What is required for the proof is an analysis of the growth behavior of the Peterson norm near the Baily-Borel boundary. Since the unitary group admits only zero-dimensional cusps, the argument is considerably simpler than in the orthogonal case.
Note that, under our assumption on $n$, the boundary of the Baily-Borel compactification $\overline{X_{\Gamma}} \smallsetminus X_{\Gamma}$ has codimension greater than two.
        \item     The modular form $F$ constructed in the proof of Theorem \ref{thm:criterion_wps} is a special reflective modular form in the sense of \cite{maeda2023fano}.
    It implies that the zero divisor of $F$ coincides with the ramification divisors $H_l$ with vanishing order $r_l-1$.
    The slope of $F$ encodes the birational geometry of $\overline{\Gamma\backslash\bB^n}$ as shown in \cite{maeda2023fano}*{Theorem~2.4}.
    \end{enumerate}
\end{rmk}

The following lemma reduces the problem of verifying the inequality in Theorem~\ref{thm:criterion_wps} for an arithmetic subgroup $\Gamma < \U(L)$ to the corresponding problem for $\U(L)$.

\begin{lemma}
\label{lem:inequality of the ratio between two groups}
    Let $\Gamma<\U(L)$ be an arithmetic subgroup.
    Then, we have inequalities
    \[
    \sum_{[l]\in \mathcal{R}^{\Gamma}} \frac{\volHM(\Gamma^{l})}{\volHM(\Gamma)} \le \sum_{[l]\in \mathcal{R}^{\U(L)}} \frac{\volHM(\U(L)^{l})}{\volHM(\U(L))}
    \]
    and
    \[
    \sum_{[l]\in \mathcal{R}^{\Gamma}}\frac{r_{l} - 1}{r_{l}} \frac{\volHM(\Gamma^{l})}{\volHM(\Gamma)} \le \sum_{[l]\in \mathcal{R}^{\U(L)}} \frac{r_{l} - 1}{r_{l}} \frac{\volHM(\U(L)^{l})}{\volHM(\U(L))}.
    \]
\end{lemma}
\begin{proof}
    While the argument is essentially the same as in \cite{stuken2017nonfreeness}*{Lemma 3.3}, we include a proof here for the sake of completeness.
    We denote by $Z'$ the centre of $\U(L^{l})$.
    Also, to distinguish between the $\U(L)$-equivalence class and the $Z\Gamma$-equivalence class of an element $l \in L$, we denote the latter by $[l]_{\Gamma}$. 
    Consider the projection $p:\mathcal{R}^{\Gamma} \to \mathcal{R}^{\U(L)}$ and take $l\in L$ with $[l]\in\mathcal{R}^{\U(L)}$.
    We divide the $U(L)$-orbit $[l]$ into the disjoint union of $Z \Gamma$-orbits as 
    \[
    [l] = \bigsqcup_{i=1}^{r} [l_{i}]_{\Gamma}.
    \]
    Then we have $p^{-1}([l]) \subseteq \{[l_1]_{\Gamma},\cdots,[l_{r}]_{\Gamma}\}$.
   By a standard computation, we have
    \[[\U(L):Z\Gamma] =  \sum_{i=1}^{r}[\U(L)_{l_i}:(Z\Gamma)_{l_i}].\]    
    Then according to Lemma~\ref{lem:descriptionofGammal} and Lemma~\ref{lem:compatibility of volHM and covolume}(1), we have
    \begin{align*}
    \sum_{i=1}^{r} \volHM(\Gamma^{l_i}) &= \sum_{i=1}^{r}[Z' \cdot \U(L)^{l_i} : Z' \cdot \Gamma^{l_i}]\volHM(\mathrm\U(L)^{l_i}) \\
    &= \sum_{i=1}^{r}[Z' \cdot \U(L)^{l_i} : Z' \cdot \Gamma^{l_i}]\volHM(\U(L)^{l}) \\
&= \sum_{i=1}^{r}[Z' \cdot \U(L)_{l_i} : Z' \cdot ( Z
\Gamma)_{l_i}]\volHM(\U(L)^{l}) \\
 &\le \sum_{i=1}^{r}[\U(L)_{l_i}: (Z\Gamma)_{l_i}]\volHM(\U(L)^{l}) \\
    & = [\U(L):Z\Gamma]\volHM(\U(L)^l).
    \end{align*}
    Thus, we have
    \[
\frac{\volHM(\U(L)^l)}{\volHM(\U(L))} = [\U(L):Z\Gamma]\frac{\volHM(\U(L)^l)}{\volHM(\Gamma)} \ge \sum_{i=1}^{r} \frac{\volHM(\Gamma^{l_i})}{\volHM(\Gamma)}
\ge \sum_{l_{i} \in p^{-1}([l])} \frac{\volHM(\Gamma^{l_i})}{\volHM(\Gamma)}.
    \]
    Taking the sum over $\U(L)$-reflective vectors, we obtain that
\[
\sum_{[l]\in \mathcal{R}^{\U(L)}} \frac{\volHM(\U(L)^l)}{\volHM(\U(L))}
\ge
\sum_{[l]_{\Gamma}\in \mathcal{R}^{\Gamma}}\frac{\volHM(\Gamma^{l})}{\volHM(\Gamma)}.
\]
Since $r_l \ge r_{l_i}$ for $l, l_{i} \in L$ such that $[l_{i}]_{\Gamma} \in p^{-1}([l])$, we also obtain the second claim.
\end{proof}
From Lemma \ref{lem:inequality of the ratio between two groups}, we focus on the ratio for the case of $\Gamma = \U(L)$.
Accordingly, unless otherwise specified, we will write $\mathcal{R}$ and $\mathcal{R}_{\mathrm{s}}$ in place of $\mathcal{R}^{\U(L)}$ and $\mathcal{R}_{\mathrm{s}}^{\U(L)}$ below.
In the following subsections, we will explicitly compute the left-hand side of Theorem~\ref{thm:criterion_wps} and prove the appropriate bound stated in \eqref{ineq:non-freeness main results} as follows.

We suppose $E \neq \bQ(\sqrt{-3})$ for simplicity.
Instead of considering the sum of the ratios of the covolumes $\volHM(\U(L)^{l})/\volHM(\U(L))$ for all $l \in \mathcal{R}$, we first consider the sum
\[
\sum_{[l]\in\mathcal{R}_{\mathrm{s}}} \frac{\volHM(\U(L)^{l})}{\volHM(\U(L))} = \sum_{[l]\in\mathcal{R}_{\mathrm{s}}} \frac{\volHM(\U(L^l))}{\volHM(\U(L))}
\]
over split vectors (see Remark~\ref{rmkULi}).
We obtain from Corollary~\ref{corcovolofSUvsHMvolofU} that the sum is bounded as
\[
\sum_{[l]\in \mathcal{R}_{\mathrm{s}}} \frac{\volHM(\U(L^l))}{\volHM(\U(L))}
\le
2 \sum_{[l]\in \mathcal{R}_{\mathrm{s}}} \frac{\mu_{\infty}(\SU(1,n-1)/\SU(L^{l}))}{\mu_{\infty}(\SU(1,n)/\SU(L))}.
\]
According to the formula in Theorem \ref{thm:prasad'svolumeformulaspecialunitarycase}, in the case where $n$ is even (the case of odd $n$ is similar), the sum in the right-hand side takes the following form:
\[\sum_{[l]\in \mathcal{R}_{\mathrm{s}}} \frac{\mu_{\infty}(\SU(1,n-1)/\SU(L^{l}))}{\mu_{\infty}(\SU(1,n)/\SU(L))}
    =\frac{(2\pi)^{n+1}}{D^{n+1/2}\cdot n!\cdot L_{E/\bQ}(n+1)}\cdot\sum_{[l]\in\mathcal{R}_{\mathrm{s}}} \prod_{v\nmid\infty}\frac{\lambda(L_v^{l_v})}{\lambda(L_v)}.\]
Here, for a finite place $v\in V_{\mathrm{f}}$ and a split reflective vector $l$, $l_v$ denotes the image of $l$ via the embedding $L \hookrightarrow L_v$.

The following two subsections are devoted to computing the term $\sum_{[l]\in\mathcal{R}_{\mathrm{s}}} \prod_{v\nmid\infty}\lambda(L_v^{l_v})/\lambda(L_v)$, but we give a brief overview of the strategy at this point.
For each $v \in V_{\fin}$ and a non-negative integer $m$, we say that $m$ is \emph{$L_{v}$-relevant}, if $L_{v, m} \neq \{0\}$, for some (and hence any) Jordan splitting $L_{v} = \bigoplus_{i} L_{v, i}$ (see Lemma~\ref{lem:rankofLJDwelldef}).
Note that if the split vector $l$ satisfies $\langle l_{v}, l_{v} \rangle \in \mathfrak{p}_{E_v}^{m}\setminus\mathfrak{p}_{E_v}^{m+1}$ for a non-negative integer $m$, then $m$ is $L_{v}$-relevant, since we can take a Jordan splitting such that $l_{v} \in L_{v, m}$.
We now introduce the following set
    \[
I(L_{v})_{\rel} \defeq
\left\{
m \in \bZ_{\ge 0} \mid \text{$m$ is $L_{v}$-relevant}
\right\}
\]
and write $i_{v, \rel} \defeq \abs{I(L_{v})_{\rel}}$.
We denote by $M(I(L_{v})_{\rel})$ (resp.\ $m(I(L_{v})_{\rel})$) the largest (resp.\ smallest) integer in $I(L_{v})_{\rel}$.
The following descriptions of $M(I(L_{v})_{\rel})$ and $m(I(L_{v})_{\rel})$ follow from their definitions immediately.
\begin{lemma}
\label{lem:descriptionofmandM}
The quantity $q_{v}^{M(I(L_{v})_{\rel})}$ agrees with the exponent of the discriminant group $ L_{v}^{\vee}/L_{v}$ of $L_{v}$.
The quantity $m(I(L_{v})_{\rel})$ agrees with the smallest integer $i$ such that there exists $x, y \in L_{v}$ with $\langle x, y \rangle_v \not \in \mathfrak{p}_{E_{v}}^{i+1}$.
\end{lemma}
We define a non-negative integer $N(L_{v})$ as
\[
N(L_{v}) \defeq M(I(L_{v})_{\rel}) - m(I(L_{v})_{\rel}).
\]
Note that $N(L_{v}) = 0$ for all but finitely many $v \in V_{\fin}$.
For simplicity, we also denote $N_{v}\defeq N(L_{v})$ in what follows.
For $m \in I(L_{v})_{\rel}$, we put
\[
i^{m}_{v, \rel} \defeq \abs{I(L_{v})_{\rel} \smallsetminus \left\{M(I(L_{v})_{\rel}), m(I(L_{v})_{\rel}), m\right\}}.
\]
This definition immediately yields  
\[
i^{m}_{v, \rel} =
\begin{cases}
i_{v, \rel} - 1 & (N_{v} = 0), \\
i_{v, \rel} - 2 & (N_{v} > 0, m \in \left\{M(I(L_{v})_{\rel}), m(I(L_{v})_{\rel})\right\}), \\
i_{v, \rel} - 3 & (m(I(L_{v})_{\rel}) < m < M(I(L_{v})_{\rel})).
\end{cases}
\]
To evaluate the ratio $\lambda(L_v^{l_v})/\lambda(L_v)$, we introduce a function $\phi_{v}$ on $I(L_v)_{\rel}$.
For $v \in V_{\fin}$ and $m \in I(L_{v})_{\rel}$, we denote by $n_{v, m}$ the rank of $L_{v, m}$ for a Jordan splitting $L_{v} = \bigoplus_{i} L_{v, i}$.
According to Lemma~\ref{lem:rankofLJDwelldef}, $n_{v, m}$ does not depend on the choice of a Jordan splitting.
\begin{defn}
    For $v \in V_{\fin}$, we define the $\bR$-valued function $\phi_{v}$ on $I(L_v)_{\rel}$ by
    \[
    \phi_{v}(m) 
    \defeq
    \begin{cases}
    \displaystyle
q_{v}^{-(N_{v}+i^{m}_{v, \rel})} \frac{1 - (-q_{v})^{- n_{v, m}} }{1 - (-q_{v})^{-(n+1)}} & (v\ \mathrm{is\ inert}), \\[1pt] \\
\displaystyle
q_{v}^{-(N_{v}+i^{m}_{v, \rel})} \frac{1 - q_{v}^{- n_{v, m}} }{1 - q_{v}^{-(n+1)}} & (v\ \mathrm{splits}), \\[1pt]  \\
\displaystyle
 q_{v}^{- (N_{v} + i^{m}_{v, \rel})/2} \cdot q_{v}^{n/2} \left(
1 + q_{v}^{- \lfloor n_{v, m}/2 \rfloor}
\right) & (v\ \mathrm{ramifies\ and\ } n\ \mathrm{is\ even}), \\[1pt] \\
\displaystyle
q_{v}^{-(N_{v} + i^{m}_{v, \rel})/2} \cdot q_{v}^{-(n+1)/2} \, \frac{
1 + q_{v}^{-\lfloor n_{v, m}/2 \rfloor}
}{
1 - q_{v}^{-(n+1)}
} &  (v\ \mathrm{ramifies\ and\ } n\ \mathrm{is\ odd}).
    \end{cases}\]
\end{defn}
This function satisfies the following property (Proposition \ref{prop:estimation of the ratio of lambda for one split reflective vector}): if $\langle l_{v}, l_{v} \rangle \in \mathfrak{p}_{E_v}^{m}\setminus\mathfrak{p}_{E_v}^{m+1}$, we have the bound
\[\frac{\lambda(L_v^{l_v})}{\lambda(L_v)} \leq \phi_v(m).\]
According to \cite{Mae24}*{Proposition 4.5}, which is based on the cancellation theorem for Hermitian lattices \cite{wall1970classification}*{Theorem 10}, the map
\[
\mathcal{R}_{\mathrm{s}} \rightarrow \prod_{v \nmid \infty} I(L_{v})_{\rel},\quad l \mapsto (m_{v})_{v},
\]
 where $m_{v}$ is the non-negative integer such that $\langle l_v, l_v \rangle \in \mathfrak{p}_{E_v}^{m_v}\setminus\mathfrak{p}_{E_v}^{m_v+1}$, is injective.
Thus, it follows
\[\sum_{[l]\in\mathcal{R}_{\mathrm{s}}}\prod_{v\nmid\infty} \frac{\lambda(L_v^{l_v})}{\lambda(L_v)} \leq \prod_{v\nmid\infty}\sum_{m_v\in I(L_v)_{\rel}}\phi_v(m_v).\]
We will provide an upper bound for the sum $\sum_{m_v \in I(L_v)_{\rel}}\phi_v(m_v)$ in Lemma~\ref{lem:final estimation of local factors}, which in turn gives an upper bound for the total contribution from split vectors in Proposition~\ref{prop:sumofsplitvectorsglobal}.
Finally, combining this with \cite{Mae24}*{Lemmas~3.6, 7.2}, we derive an appropriate estimation of the full sum
\[
\sum_{[l]\in\mathcal{R}} \frac{\volHM(\U(L)^{l})}{\volHM(\U(L))}
\]
in Theorem~\ref{thm:not_wps}.

\subsection{Local computation}
\label{subsec:Estimation of local factors}
In this subsection, we fix a finite place $v$.
Recall that $l_v$ denotes the image of a split reflective vector $l \in L$ under the natural embedding $L\hookrightarrow L_v$.
We will give an estimation of the ratio $\lambda(L_v^{l_v})/\lambda(L_v)$.
Since $l$ is a split vector, we can take a Jordan splitting 
\[
L_{v} = \bigoplus_{i \in I(L_{v})_{\rel}} L_{v, i}
\]
such that $l_{v} \in L_{v, m_v}$, where $m_v$ is the non-negative integer satisfying $\langle l_v, l_v \rangle \in \mathfrak{p}_{E_v}^{m_v}\setminus\mathfrak{p}_{E_v}^{m_v+1}$.
We write $L_{v}^{l_{v}}$ (resp.\ $L_{v, m_{v}}^{l_{v}}$) for the orthogonal complement of $l_v$ in $L_{v}$ (resp.\ $L_{v, m_v}$).
Then the orthogonal decomposition
\[
L_{v}^{l_{v}} = \bigoplus_{i < m_{v}} L_{v,i} \oplus L_{v, m_{v}}^{l_{v}} \oplus \bigoplus_{i > m_{v}} L_{v, i}
\]
is a Jordan splitting of $L_{v}^{l_{v}}$.

We define $d_{v} \in \{1,2\}$ by $q_{E_{v}} = q_{v}^{d_{v}}$.
Hence, we have $d_{v} = 2$ if $v$ is inert or splits over $E$, and otherwise, we have $d_{v} = 1$.
We write
\[
s_{L_{v}} \defeq d_{v} \sum_{m(I(L_{v})_{\rel}) \le i <j \le M(I(L_{v})_{\rel})} \lfloor \frac{j-i-1}{2} \rfloor n_{v, i} \cdot n_{v, j}
\]

Recall that we defined $\lambda(L_{v}) = \lambda(M_{L_{v}}) \cdot \Ind(L_{v})$.
We write $\lambda'(L_{v}) \defeq q_{v}^{-s_{L_{v}}} \lambda(L_{v})$.
Replacing $L_{v}$ with $L_{v}^{l_{v}}$, we also define $s_{L_{v}^{l_{v}}}$, $\lambda(L_{v}^{l_{v}})$, and $\lambda'(L_{v}^{l_{v}})$.
The ratio of the second factors of $\lambda(L_{v})$ and $\lambda(L_{v}^{l_{v}})$ is calculated as follows.
\begin{lemma}
\label{lem:Ind ratio and difference of s}
\begin{enumerate}
\item     The ratio $\Ind(L_v^{l_v})/\Ind(L_v)$ is
    \[q_v^{- (s_{L_{v}} - s_{L_v^{l_v}})}\cdot \abs{\mathbf{G}_{L_{v}}(\mathfrak{f}_{v})}\abs{
\mathbf{G}_{{M}_{L^{l_v}_{v}}}(\mathfrak{f}_{v})
} \abs{\mathbf{G}_{L^{l_v}_{v}}(\mathfrak{f}_{v})}^{-1}\abs{
\mathbf{G}_{{M}_{L_{v}}}(\mathfrak{f}_{v})
}^{-1}\]
    \item The difference $s_{L_{v}} - s_{L_{v}^{l_{v}}}$ is \[d_{v} \left(
\sum_{m(I(L_{v})_{\rel}) \le i < m_{v}} \lfloor \frac{m_{v}-i-1}{2} \rfloor n_{v, i} 
+ \sum_{m_{v} < j \le M(I(L_{v})_{\rel})} \lfloor \frac{j - m_{v}-1}{2} \rfloor n_{v, j} 
\right).
    \]
\end{enumerate}

\end{lemma}
\begin{proof}
    The first claim follows from Theorem~\ref{thm:generalization of Prasad's volume formula}, and the second claim follows from the definitions of $s_{L_{v}}$ and $s_{L_{v}^{l_{v}}}$.
\end{proof}

We divide the problem of the estimation of $\lambda(L_v^{l_v})/\lambda(L_v)$ into the estimations of $\lambda'(L_v^{l_v})/\lambda'(L_v)$ and $s_{L_v} - s_{L_v^{e_v}}$.
The first step to estimate the former ratio is due to the following lemma.
\begin{lemma}
\label{lem:ratiooflambda'}
    The ratio $\lambda'(L_v^{l_v})/\lambda'(L_v)$ is given by $\abs{\mathbf{G}_{L_{v}}(\mathfrak{f}_{v})}\abs{\mathbf{G}_{L^{l_v}_{v}}(\mathfrak{f}_{v})}^{-1}$ times 
    \[\begin{cases}
    \displaystyle
    q_v^{(\dim \mathbf{G}_{{M}_{L_{v}^{l_{v}}}} -\dim \mathbf{G}_{{M}_{L_{v}}} + 1)/2}(q_v^{n+1}-(-1)^{n+1})^{-1} &(v:\mathrm{inert}),\\
        \displaystyle
    q_v^{(\dim \mathbf{G}_{{M}_{L_{v}^{l_{v}}}} -\dim \mathbf{G}_{{M}_{L_{v}}} + 1)/2}(q_v^{n+1}-1)^{-1} &(v:\mathrm{split}), \\
        \displaystyle
    q_v^{(\dim \mathbf{G}_{{M}_{L_{v}^{l_{v}}}} -\dim \mathbf{G}_{{M}_{L_{v}}})/2}&(v:\mathrm{ramify},\ n\mathrm{:even}),\\
        \displaystyle
    q_v^{(\dim \mathbf{G}_{{M}_{L_{v}^{l_{v}}}} -\dim \mathbf{G}_{{M}_{L_{v}}} + 1)/2}(q_v^{n+1}-1)^{-1} &(v:\mathrm{ramify},\ n\mathrm{:odd}).
\end{cases}\]
\end{lemma}
\begin{proof}
The claim follows from the definitions of $\lambda(M_{L_{v}})$ and $\lambda(M_{L_{v}^{l_{v}}})$ given by \eqref{defoflambdaMLv} and Lemma~\ref{lem:Ind ratio and difference of s}(1).
\end{proof}

We will give explicit calculations for the factors appearing in Lemma~\ref{lem:ratiooflambda'} to obtain an estimation of $\lambda'(L_v^{l_v})/\lambda'(L_v)$.
We define sublattices $(M_{L_{v}})_{0}$ and $(M_{L_{v}})_{1}$ of $M_{L_{v}}$ by
\[
(M_{L_{v}})_{0} = \bigoplus_{i \equiv 0 \bmod 2} \mathfrak{p}_{E_{v}}^{-\lfloor i/2 \rfloor}L_{i}
\qquad 
\text{and}
\qquad
(M_{L_{v}})_{1} = \bigoplus_{i \equiv 1 \bmod 2} \mathfrak{p}_{E_{v}}^{-\lfloor i/2 \rfloor}L_{i}.
\]
Then the orthogonal decomposition
\[
M_{L_{v}} = (M_{L_{v}})_{0} \oplus (M_{L_{v}})_{1}
\]
is a Jordan splitting of $M_{L_v}$.
We define $\epsilon_{m_{v}} \in \{0, 1\}$ by $m_{v} \equiv \epsilon_{m_{v}} \bmod 2$ and let $n_{\equiv m_{v}}$ denote the rank of $(M_{L_{v}})_{\epsilon_{m_{v}}}$.
Thus, we have
\[
n_{\equiv m_{v}} = \sum_{i \equiv m_{v} \bmod 2} n_{v, i}.
\]
The description of the maximal reductive quotients by Gan and Yu gives the following estimation.
\begin{prop}
\label{prop:preciseestimationofquotoflambda}
We have
\[
 \frac{\lambda'(L_{v}^{l_{v}})}{\lambda'(L_{v})} =
\begin{cases}
    \displaystyle
    q_{v}^{-(n+1 + n_{\equiv m_{v}} - 2 n_{v, m_{v}})} \frac{1 - (-q_{v})^{- n_{v, m_{v}}} }{1 - (-q_{v})^{-(n+1)}} & (v:\mathrm{inert}), \\[2ex]
    \displaystyle
    q_{v}^{-(n+1 + n_{\equiv m_{v}} - 2 n_{v, m_{v}})} \frac{1 - q_{v}^{-n_{v, m_{v}}}}{1 - q_{v}^{-(n+1)}} & (v:\mathrm{split}).
\end{cases}
\]
If $v$ ramifies over $E$ and $n$ is even, we have
\[
 \frac{\lambda'(L_{v}^{l_{v}})}{\lambda'(L_{v})} 
\le
q_{v}^{-(n + 1 + n_{\equiv m_{v}} - 2n_{v, m_{v}})/2} \cdot q_{v}^{n/2} \left(
1 + q_{v}^{- \lfloor n_{v, m_{v}}/2 \rfloor}
\right).
\]
If $v$ ramifies over $E$ and $n$ is odd, we have
\[
 \frac{\lambda'(L_{v}^{l_{v}})}{\lambda'(L_{v})} 
\le
q_{v}^{-(n + 1 + n_{\equiv m_{v}} - 2n_{v, m_{v}})/2} \cdot q_{v}^{-(n+1)/2} \, \frac{
1 + q_{v}^{-\lfloor n_{v, m_{v}}/2 \rfloor}
}{
1 - q_{v}^{-(n+1)}
}.
\]
\end{prop}
\begin{proof}
The proof is essentially the same as the arguments in \cite{Mae24}*{Subsections~6.1 and 6.2}.
We record a brief sketch of the arguments there.
We obtain from the explicit descriptions of the reductive quotients $\mathbf{G}_{L_{v}}$, $\mathbf{G}_{L_{v}^{l_{v}}}$, $\mathbf{G}_{{M}_{L_{v}}}$, and $\mathbf{G}_{{M}_{L_{v}}^{l_{v}}}$ in \cite{Gan-Yu}*{6.2.3 Proposition, 6.3.9 Proposition} that
\[
\dim \mathbf{G}_{{M}_{L_{v}}} - \dim \mathbf{G}_{{M}_{L_{v}}^{l_{v}}} = 2 n_{\equiv m_{v}} - 1, \, 
\frac{
\abs{
\mathbf{G}_{L_{v}}(\mathfrak{f}_{v})
}
}
{
\abs{\mathbf{G}_{L_{v}^{l_{v}}}(\mathfrak{f}_{v})}
} = q_{v}^{n_{v, m_{v}}-1} (q_{v}^{n_{v, m_{v}}} - (-1)^{n_{v, m_{v}}})
\]
if $v$ is inert over $E$, 
\[
\dim \mathbf{G}_{{M}_{L_{v}}} - \dim \mathbf{G}_{{M}_{L_{v}}^{l_{v}}} = 2 n_{\equiv m_{v}} - 1, \, 
\frac{
\abs{
\mathbf{G}_{L_{v}}(\mathfrak{f}_{v})
}
}
{
\abs{\mathbf{G}_{L_{v}^{l_{v}}}(\mathfrak{f}_{v})}
} = q_{v}^{n_{v, m_{v}}-1} (q_{v}^{n_{v, m_{v}}} - 1)
\]
if $v$ splits over $E$, and
\[
\dim \mathbf{G}_{{M}_{L_{v}}} - \dim \mathbf{G}_{{M}_{L_{v}}^{l_{v}}} = n_{\equiv m_{v}} - 1, \, 
\frac{
\abs{
\mathbf{G}_{L_{v}}(\mathfrak{f}_{v})
}
}
{
\abs{\mathbf{G}_{L_{v}^{l_{v}}}(\mathfrak{f}_{v})}
} \le q_{v}^{n_{v, m_{v}}-1} + q_{v}^{\lfloor (n_{v, m_{v}}-1)/2 \rfloor} 
\]
if $v$ ramifies over $E$.

Suppose that $v$ is inert over $E$.
Then obtain from Lemma~\ref{lem:ratiooflambda'} that
\begin{align*}
 \frac{\lambda'(L_{v}^{l_{v}})}{\lambda'(L_{v})}
  &=
    q_{v}^{(1 - (\dim \mathbf{G}_{{M}_{L_{v}}} - \dim \mathbf{G}_{{M}_{L_{v}}^{l_{v}}}))/2}
   \frac{
\abs{
\mathbf{G}_{L_{v}}(\mathfrak{f}_{v})
}
}
{
\abs{\mathbf{G}_{L_{v}^{l_{v}}}(\mathfrak{f}_{v})}
}
(q_{v}^{n+1} - (-1)^{n+1})^{-1} \\
&=
q_{v}^{n_{v, m_{v}} - n_{\equiv m_{v}}} \frac{q_{v}^{n_{v, m_{v}}} - (-1)^{n_{v, m_{v}}}}{q_{v}^{n+1} - (-1)^{n+1}} \\
&= q_{v}^{-(n+1 + n_{\equiv m_{v}} - 2 n_{v, m_{v}})} \frac{1 - (-q_{v})^{- n_{v, m_{v}}} }{1 - (-q_{v})^{-(n+1)}}.
\end{align*}
If $v$ splits over $E$, a similar computation implies that
\[
\frac{\lambda'(L_{v}^{l_{v}})}{\lambda'(L_{v})} = q_{v}^{-(n+1 + n_{\equiv m_{v}} - 2 n_{v, m_{v}})} \frac{1 - q_{v}^{-n_{v, m_{v}}}}{1 - q_{v}^{-(n+1)}}.
\]
Finally, we consider the case that $v$ ramifies over $E$.
Suppose that $n$ is even.
In this case, by using \cite{Gan-Yu}*{6.2.3 Proposition, 6.3.9 Proposition} and Lemma~\ref{lem:ratiooflambda'} as above, we obtain that
\begin{align*}
 \frac{\lambda'(L_{v}^{l_{v}})}{\lambda'(L_{v})}
  &=
    q_{v}^{- (\dim \mathbf{G}_{{M}_{L_{v}}} - \dim \mathbf{G}_{{M}_{L_{v}}^{l_{v}}})/2}
   \frac{
\abs{
\mathbf{G}_{L_{v}}(\mathfrak{f}_{v})
}
}
{
\abs{\mathbf{G}_{L_{v}^{l_{v}}}(\mathfrak{f}_{v})}
}\\
&\le
q_{v}^{(1-n_{\equiv m_{v}})/2} \left(
q_{v}^{n_{v, m_{v}}-1} + q_{v}^{\lfloor (n_{v, m_{v}}-1)/2 \rfloor}
\right) \\
&= q_{v}^{(2 n_{v, m_{v}} -n_{\equiv m_{v}} - 1)/2} \left(
1 + q_{v}^{- \lfloor n_{v, m_{v}}/2 \rfloor}
\right) \\
&= q_{v}^{-(n + 1 + n_{\equiv m_{v}} - 2n_{v, m_{v}})/2} \cdot q_{v}^{n/2} \left(
1 + q_{v}^{- \lfloor n_{v, m_{v}}/2 \rfloor}
\right).
\end{align*}
Similarly, if $n$ is odd, we have
\begin{align*}
 \frac{\lambda'(L_{v}^{l_{v}})}{\lambda'(L_{v})}
  &=
    q_{v}^{(1 - (\dim \mathbf{G}_{{M}_{L_{v}}} - \dim \mathbf{G}_{{M}_{L_{v}}^{l_{v}}}))/2}
   \frac{
\abs{
\mathbf{G}_{L_{v}}(\mathfrak{f}_{v})
}
}
{
\abs{\mathbf{G}_{L_{v}^{l_{v}}}(\mathfrak{f}_{v})}
} (q_{v}^{n+1} - 1)^{-1}\\
&\le
q_{v}^{(2 -n_{\equiv m_{v}})/2} \, \frac{q_{v}^{n_{v, m_{v}}-1} + q_{v}^{\lfloor (n_{v, m_{v}}-1)/2 \rfloor}}{q_{v}^{n+1} - 1} \\
&= q_{v}^{-(n + 1 + n_{\equiv m_{v}} - 2n_{v, m_{v}})/2} \cdot q_{v}^{-(n+1)/2} \, \frac{
1 + q_{v}^{-\lfloor n_{v, m_{v}}/2 \rfloor}
}{
1 - q_{v}^{-(n+1)}
}.
\qedhere
\end{align*}
\end{proof}

Next, we proceed to address the second problem, which concerns the estimation of $s_{L_{v}} - s_{L_{v}^{l_{v}}}$.

\begin{lemma}
\label{lem:estimationofs}
If $m_{v} \in \left\{m(I(L_{v})_{\rel}), M(I(L_{v})_{\rel})\right\}$, we have
\[
s_{L_{v}} - s_{L_{v}^{l_{v}}} \ge d_{v} \lfloor \frac{N_{v}-1}{2} \rfloor.
\]
If $m(I(L_{v})_{\rel}) < m_{v} < M(I(L_{v})_{\rel})$, we have
\[
s_{L_{v}} - s_{L_{v}^{l_{v}}} \ge d_{v} \lfloor \frac{N_{v}-2}{2} \rfloor
\]
unless all of $M(I(L_{v})_{\rel})$, $m(I(L_{v})_{\rel})$, and $m_{v}$ have the same parity.
If all of $M(I(L_{v})_{\rel})$, $m(I(L_{v})_{\rel})$, and $m_{v}$ have the same parity, we have
\[
s_{L_{v}} - s_{L_{v}^{l_{v}}} \ge d_{v} \cdot \frac{N_{v}-4}{2}.
\]
\end{lemma}
\begin{proof}
  The first claim follows from Lemma \ref{lem:Ind ratio and difference of s}(2) and the definition of $N_{v}$.
  Suppose that $0 < m_{v} < N_{v}$.
Then we obtain from Lemma \ref{lem:Ind ratio and difference of s}(2) that
\[
s_{L_{v}} - s_{L_{v}^{l_{v}}} \ge d_{v} \left(
\left\lfloor \frac{m_{v}- m(I(L_{v})_{\rel}) - 1}{2} \right\rfloor + \left\lfloor \frac{M(I(L_{v})_{\rel}) - m_{v}-1}{2} \right\rfloor
\right).
\]
Thus, the claim follows from the standard calculation 
\[
\left\lfloor \frac{x}{2} \right\rfloor + \left\lfloor \frac{y}{2} \right\rfloor
=
\begin{cases}
\displaystyle
\frac{x+y-2}{2} & (\text{both of $x$ and $y$ are odd}), \\
\displaystyle
\left\lfloor \frac{x+y}{2} \right\rfloor & (\text{otherwise}).
\end{cases}
\]
\end{proof}

We analyze the exponential term in $q_{v}$ appearing from the presentation in Proposition \ref{prop:preciseestimationofquotoflambda}  and the contribution of $s_{L_v} - s_{L_v^{l_v}}$.
Using the bound provided by Lemma \ref{lem:estimationofs}, we derive the following technical result.

\begin{lemma}
\label{lem:estimationofs+alpha}
We have
\[
n+1 + n_{\equiv m_{v}} - 2 n_{v, m_{v}} + \frac{2}{d_{v}} \left(
s_{L_{v}} - s_{L_{v}^{l_{v}}}
\right)\ge N_{v} + i^{m_{v}}_{v, \rel}.
\]
\end{lemma}
\begin{proof}
If $N_{v} = 0$, the claim is obvious.
Suppose that $N_{v} > 0$.
We will prove the lemma by dividing it into five cases:
\begin{description}
  \item[Case 1] 
  \( m_v \in \{ M(I(L_v)_{\mathrm{rel}}),\; m(I(L_v)_{\mathrm{rel}}) \} \)  
  and \( N_v \) is odd.

  \item[Case 2] 
  \( m_v \in \{ M(I(L_v)_{\mathrm{rel}}),\; m(I(L_v)_{\mathrm{rel}}) \} \)  
  and \( N_v \) is even.

  \item[Case 3] 
  \( m(I(L_v)_{\mathrm{rel}}) < m_v < M(I(L_v)_{\mathrm{rel}}) \)  
  and \( N_v \) is odd.

  \item[Case 4] 
  \( m(I(L_v)_{\mathrm{rel}}) < m_v < M(I(L_v)_{\mathrm{rel}}) \)  
  and  
  \( m_v \not\equiv m(I(L_v)_{\mathrm{rel}}) \equiv M(I(L_v)_{\mathrm{rel}}) \pmod{2} \).

  \item[Case 5] 
  \( 0 < m_v < N_v \)  
  and  
  \( m_v \equiv m(I(L_v)_{\mathrm{rel}}) \equiv M(I(L_v)_{\mathrm{rel}}) \pmod{2} \).
\end{description}

We will prove the lemma for each case below.

\begin{description}
\item[Case 1]
We obtain from the definition of $i^{m_{v}}_{v, \rel}$ that $n_{v, m_{v}} \le n+1 -1 - i^{m_{v}}_{v, \rel}$.
Thus, we have 
\[
n+1 + n_{\equiv m_{v}} - 2 n_{v, m_{v}} 
\ge n+1 - n_{v, m_{v}} \ge 1 + i^{m_{v}}_{v, \rel}.
\]
Combining this with Lemma~\ref{lem:estimationofs}, it yields
\[
n+1 + n_{\equiv m_{v}} - 2 n_{v, m_{v}} + \frac{2}{d_{v}} \left(
s_{L_{v}} - s_{L_{v}^{l_{v}}}
\right) \ge 1 + i^{m_{v}}_{v, \rel} + (N_{v} - 1) = N_{v} + i^{m_{v}}_{v, \rel}.
\]
\item[Case 2]
Since $M(I(L_{v})_{\rel})$ and $m(I(L_{v})_{\rel})$ have the same parity, we have $n_{v, m_{v}} \le n_{\equiv m_{v}} - 1$.
Moreover, the definition of $i^{m_{v}}_{v, \rel}$ implies that $n_{v, m_{v}} \le n+1 -1 - i^{m_{v}}_{v, \rel}$.
Hence, we have 
\[
n+1 + n_{\equiv m_{v}} - 2 n_{v, m_{v}} 
= (n+1 - n_{v, m_{v}}) + (n_{\equiv m_{v}} - n_{v, m_{v}})
\ge 1 + i^{m_{v}}_{v, \rel} + 1 = 2 + i^{m_{v}}_{v, \rel}.
\]
Combining this with Lemma~\ref{lem:estimationofs}, it yields
\[
n+1 + n_{\equiv m_{v}} - 2 n_{v, m_{v}} + \frac{2}{d_{v}} \left(
s_{L_{v}} - s_{L_{v}^{l_{v}}}
\right) \ge 2 + i^{m_{v}}_{v, \rel} + (N_{v} - 2) = N_{v} + i^{m_{v}}_{v, \rel}.
\]
\item[Case 3]
Since one of $M(I(L_{v})_{\rel})$ and $m(I(L_{v})_{\rel})$ have the same parity as $m_{v}$, we have $n_{v, m_{v}} \le n_{\equiv m_{v}} - 1$.
Moreover, the definition of $i^{m_{v}}_{v, \rel}$ implies that $n_{v, m_{v}} \le n + 1 - 2 -i^{m_{v}}_{v, \rel}$.
Hence, we have 
\[
n+1 + n_{\equiv m_{v}} - 2 n_{v, m_{v}} = (n+1-n_{v, m_{v}}) + (n_{\equiv m_{v}} - n_{v, m_{v}}) \ge 3 + i^{m_{v}}_{v, \rel}.
\]
Combining this with Lemma~\ref{lem:estimationofs}, it yields
\[
n+1 + n_{\equiv m_{v}} - 2 n_{v, m_{v}} + \frac{2}{d_{v}} \left(
s_{L_{v}} - s_{L_{v}^{l_{v}}}
\right) \ge 3 + i^{m_{v}}_{v, \rel} + (N_{v} - 3) = N_{v} + i^{m_{v}}_{v, \rel}.
\]
\item[Case 4]
We obtain from the definition of $i^{m_{v}}_{v, \rel}$ that $
n_{v, m_{v}} \le n + 1 - 2 - i^{m_{v}}_{v, \rel}
$.
Hence, we have
\[
n+1 + n_{\equiv m_{v}} - 2 n_{v, m_{v}} = (n+1-n_{v, m_{v}}) + (n_{\equiv m_{v}} - n_{v, m_{v}}) \ge 2 + i^{m_{v}}_{v, \rel}.
\]
Combining this with Lemma~\ref{lem:estimationofs}, it yields
\[
n+1 + n_{\equiv m_{v}} - 2 n_{v, m_{v}} + \frac{2}{d_{v}} \left(
s_{L_{v}} - s_{L_{v}^{l_{v}}}
\right) \ge 2 + i^{m_{v}}_{v, \rel} + (N_{v} - 2) = N_{v} + i^{m_{v}}_{v, \rel}.
\]
\item[Case 5]
By the assumptions, we have $n_{v, m_{v}} \le n_{\equiv m_{v}} -2$.
Moreover, the definition of $i^{m_{v}}_{v, \rel}$ implies that $n_{v, m_{v}} \le n + 1 -2 - i^{m_{v}}_{v, \rel}$.
Hence, we have
\[
n+1 + n_{\equiv m_{v}} - 2 n_{v, m_{v}} = (n+1-n_{v, m_{v}}) + (n_{\equiv m_{v}} - n_{v, m_{v}}) \ge 4 + i^{m_{v}}_{v, \rel}.
\]
Combining this with Lemma~\ref{lem:estimationofs}, it yields
\[
n+1 + n_{\equiv m_{v}} - 2 n_{v, m_{v}} + \frac{2}{d_{v}} \left(
s_{L_{v}} - s_{L_{v}^{l_{v}}}
\right) \ge 4 + i^{m_{v}}_{v, \rel} + (N_{v} - 4) = N_{v} + i^{m_{v}}_{v, \rel}.
\qedhere
\]
\end{description}
\end{proof}

As mentioned in Subsection \ref{subsec:a criterion}, we here evaluate the ratio $\lambda(L_v^{l_v})/\lambda(L_v)$ by the function $\phi_v$.

\begin{prop}
\label{prop:estimation of the ratio of lambda for one split reflective vector}
We have
\[\frac{
    \lambda(L_{v}^{l_{v}})}{\lambda(L_{v})} \le\phi_{v}(m_v).\]
\end{prop}
\begin{proof}
Recall that we have
\[
\frac{
    \lambda(L_{v}^{l_{v}})}{\lambda(L_{v})} =  q_{v}^{-(s_{L_{v}} - s_{L_{v}^{l_{v}}})}\frac{\lambda'(L_{v}^{l_{v}})}{\lambda'(L_{v})}
\]
by their definitions.
Thus, the claim follows by combining Proposition~\ref{prop:preciseestimationofquotoflambda} with Lemma~\ref{lem:estimationofs+alpha}.
\end{proof}

At the end of this subsection, we derive an upper bound for the sum of $\phi_{v}(m_{v})$, which will subsequently be used to bound 
\[
\sum_{[l]\in\mathcal{R}_{\mathrm{s}}}\frac{\mu_{\infty}(\SU(1,n-1)/\SU(L^{l}))}{\mu_{\infty}(\SU(1,n)/\SU(L))}
\]
in Subsection \ref{subsecGlobalcomputation}.

\begin{lemma}
\label{lem:final estimation of local factors}
    If $v$ is inert or splits over $E$, we have
    \[
    \sum_{m_v \in I(L_v)_{\rel}} \phi_v(m_v)
    \le
    \begin{cases}
    \displaystyle
     1 & (N_{v} = 0), \\ \displaystyle
     \frac{16}{5} q_{v}^{-N_{v}} & (N_{v} > 0).
    \end{cases}
    \]
    If $v$ ramifies over $E$ and $n$ is even, we have
    \[
   \sum_{m_v \in I(L_v)_{\rel}} \phi_v(m_v)
    \le 
    \begin{cases}
    2 q_{v}^{n/2} & (N_{v} = 0), \\
    5 q_{v}^{-N_{v}/2} q_{v}^{n/2} & (N_{v} > 0).
    \end{cases}
    \]
     If $v$ ramifies over $E$ and $n$ is odd, we have
    \[
\sum_{m_v \in I(L_v)_{\rel}} \phi_v(m_v)
    \le
    \begin{cases}
    2 q_{v}^{-(n+1)/2} & (N_{v} = 0), \\
    5 q_{v}^{-N_{v}/2} q_{v}^{-(n+1)/2} & (N_{v} > 0).
    \end{cases}
    \]
\end{lemma}
\begin{proof}
The claims follows from the definition of $\phi_{v}$ when $N_{v} = 0$.
Suppose that $0 < N_{v}$.
In this case, we have $i_{v, \rel} \ge 2$ since $m(I(L_{v})_{\rel}), M(I(L_{v})_{\rel}) \in I(L_{v})_{\rel}$.

First, we suppose that $v$ is inert or splits over $E$.
    Then the definition of $\phi_{v}$ implies that
    \[
    \sum_{m_v \in I(L_v)_{\rel}} \phi_v(m_v) \le \sum_{m_v \in I(L_v)_{\rel}} q_{v}^{-(N_{v}+i^{m_v}_{v, \rel})} \frac{1 + q_{v}^{- n_{v, m_v}} }{1 - q_{v}^{-(n+1)}}
    \le \frac{q_{v}^{-N_{v}}}{1 - q_{v}^{-(n+1)}} \sum_{m_v \in I(L_v)_{\rel}} q_{v}^{-i^{m_v}_{v, \rel}} \left(
    1 + \frac{1}{q_{v}}
    \right).
\]
Since
\[
i^{m}_{v, \rel} =
\begin{cases}
i_{v, \rel} - 2 & (m \in \left\{M(I(L_{v})_{\rel}), m(I(L_{v})_{\rel})\right\}), \\
i_{v, \rel} - 3 & (m(I(L_{v})_{\rel}) < m < M(I(L_{v})_{\rel})),
\end{cases}
\]
we have 
\begin{align*}
& \quad \frac{q_{v}^{-N_{v}}}{1 - q_{v}^{-(n+1)}} \sum_{m_v \in I(L_v)_{\rel}} q_{v}^{-i^{m_v}_{v, \rel}} 
\left(1 + \frac{1}{q_{v}}\right) \\
&= \frac{q_{v}^{-N_{v}}}{1 - q_{v}^{-(n+1)}} \bigg[
2 q_{v}^{-(i_{v, \rel} - 2)} \left(1 + \frac{1}{q_{v}}\right)
+ q_{v}^{-(i_{v, \rel} - 3)} (i_{v, \rel} - 2) 
\left(1 + \frac{1}{q_{v}}\right)
\bigg] \\
&= \frac{q_{v}^{-N_{v}}}{1 - q_{v}^{-(n+1)}} q_{v}^{-(i_{v, \rel} - 2)} 
\left((i_{v, \rel} - 2)q_{v} + i_{v, \rel} + \frac{2}{q_{v}}\right).
\end{align*}
One can check easily that 
\[
q_{v}^{-(i_{v, \rel} - 2)} \left(
(i_{v, \rel} - 2)q_{v} + i_{v, \rel} + \frac{2}{q_{v}}
    \right) \le 3
\]
for all $2 \le i_{v, \rel}$.
Thus, it conclude that
\[
\sum_{m_v \in I(L_v)_{\rel}} \phi_v(m_v) \le \frac{3 q_{v}^{-N_{v}}}{1 - q_{v}^{-(n+1)}} 
\le \frac{3 q_{v}^{-N_{v}}}{1 - 2^{-(3+1)}} = \frac{16}{5} q_{v}^{-N_{v}}.
\]

Next, we suppose that $v$ ramified over $E$ and $n$ is even.
Note that in this case, we have $3 \le q_{v}$.
By a similar calculation as above, we obtain that
\begin{align*}
\sum_{m_v \in I(L_v)_{\rel}} \phi_v(m_v) &\le \sum_{m_v \in I(L_v)_{\rel}} q_{v}^{- (N_{v} + i^{m_v}_{v, \rel})/2} \cdot q_{v}^{n/2} \left(
1 + q_{v}^{- \lfloor n_{v, m_v}/2 \rfloor}
\right) \\
&\le q_{v}^{-N_{v}/2} q_{v}^{n/2} \left(
2 q_{v}^{-(i_{v, \rel} - 2)/2} \left(
1 + q_{v}^{0}
\right)
+ q_{v}^{-(i_{v, \rel} - 3)/2} (i_{v, \rel} - 2) \left(
1 + q_{v}^{0}
\right)
\right) \\
&=  q_{v}^{-N_{v}/2} q_{v}^{n/2} 2 q_{v}^{-(i_{v, \rel} - 2)/2}\left(
(i_{v, \rel} - 2) q_{v}^{1/2} + 2
\right) \\
&\le 5 q_{v}^{-N_{v}/2} q_{v}^{n/2}.
\end{align*}

Finally, we suppose that $v$ ramified over $E$ and $n$ is odd.
A similar calculation shows that 
\begin{align*}
\sum_{m_v \in I(L_v)_{\rel}} \phi_v(m_v) &\le \sum_{l_{m}} q_{v}^{-(N_{v} + i^{m_v}_{v, \rel})/2} \cdot q_{v}^{-(n+1)/2} \, \frac{
1 + q_{v}^{-\lfloor n_{v, m_v}/2 \rfloor}
}{
1 - q_{v}^{-(n+1)}
} \\
&\le 
\frac{
q_{v}^{-N_{v}/2} q_{v}^{-(n+1)/2}
}{1 - q_{v}^{-(n+1)}}
\left(
2 q_{v}^{-(i_{v, \rel} - 2)/2} (1 + q_{v}^0)
+ q_{v}^{-(i_{v, \rel} - 3)/2} (i_{v, \rel} - 2) (1 + q_{v}^{0})
\right) \\
&= \frac{
q_{v}^{-N_{v}/2} q_{v}^{-(n+1)/2}
}{1 - q_{v}^{-(n+1)}} 2 q_{v}^{-(i_{v, \rel} - 2)/2} \left(
(i_{v, \rel} - 2) q_{v}^{1/2} + 2
\right) \\
&\le
\frac{
q_{v}^{-N_{v}/2} q_{v}^{-(n+1)/2}
}{1 - 3^{-(3+1)}} 2 q_{v}^{-(i_{v, \rel} - 2)/2} \left(
(i_{v, \rel} - 2) q_{v}^{1/2} + 2
\right) \\
&= \frac{81}{40} q_{v}^{-N_{v}/2} q_{v}^{-(n+1)/2} q_{v}^{-(i_{v, \rel} - 2)/2} \left(
(i_{v, \rel} - 2) q_{v}^{1/2} + 2
\right) \\
&\le 5 q_{v}^{-N_{v}/2} q_{v}^{-(n+1)/2}.
\qedhere
\end{align*}
\end{proof}

\subsection{Global computation}
\label{subsecGlobalcomputation}

Based on the local computations in the previous subsection, we will give an estimation of the sum
\[
\sum_{[l]\in\mathcal{R}_{\mathrm{s}}}\prod_{v\nmid\infty}\frac{\lambda(L_{v}^{l_{v}})}{\lambda(L_{v})}.
\]
We define
\[
N(L) \defeq \prod_{v \nmid \infty} q_{v}^{N_{v}}.
\]
Note that if $0 \in I(L_{v})_{\rel}$ for all $v \in V_{\fin}$, the quantity $N(L)$ agrees with the exponent of the discriminant group $L^{\vee}/L$ of $L$ with $L^{\vee}$ denoting the dual lattice
\[
L^{\vee} \defeq \{
x \in L \otimes_{\bZ} \bQ \mid \langle x, L \rangle \subseteq \mathcal{O}_{E}
\};
\]
see Lemma~\ref{lem:descriptionofmandM}.

\begin{lemma}
\label{lem:global summation for reflective vectors}
There exists a constant $\epsilon > 0$, independent of $L$, $n$, and $E$, such that
    \[\sum_{[l]\in\mathcal{R}_{\mathrm{s}}}\prod_{v\nmid\infty}\frac{\lambda(L_{v}^{l_{v}})}{\lambda(L_{v})} \le 2 D^sN(L)^{-\epsilon}\]
\end{lemma}
where 
\[s=\begin{cases}
    (n+1)/2 & (2\mid n), \\
    -n/2 & (2\nmid n).
\end{cases}\]
\begin{proof}
    As explained in Subsection~\ref{subsec:a criterion}, we can bound the left-hand side of the claim as 
    \[\sum_{[l]\in\mathcal{R}_{\mathrm{s}}}\prod_{v\nmid\infty} \frac{\lambda(L_v^{l_v})}{\lambda(L_v)} \leq \prod_{v\nmid\infty}\sum_{m_v\in I(L_v)_{\rel}}\phi_v(m_v)\]
    by using the cancellation theorem for Hermitian lattices \cite{wall1970classification}*{Theorem 10} and Proposition~\ref{prop:estimation of the ratio of lambda for one split reflective vector}.
   To derive an upper bound for the right-hand side, we temporarily define the following sets consisting of finite places, which will be used only in this proof.
\begin{align*}
    A_{\mathrm{ur}}&\defeq \{v\mid v\ \mathrm{is\ unramified\ at}\ E\ \mathrm{and\ } N_{v} > 0\}, \\
    A_{\mathrm{rm}}^1&\defeq \{v\mid v\ \mathrm{ramifies\ over}\ E\ \mathrm{and\ }N_{v} > 0\}, \\
    A_{\mathrm{rm}}^2&\defeq \{v\mid v\ \mathrm{ramifies\ over}\ E\ \mathrm{and\ }v\not\in A_{\mathrm{rm}}^1\}.
\end{align*}
First, we assume that $n$ is even.
We introduce a constant $0<\epsilon <<1$, which will later be defined precisely, but is useful at this stage for the sake of formal manipulations.
    It follows from the estimations in Lemma~\ref{lem:final estimation of local factors} that 
    \begin{align*}
&\prod_{v\nmid\infty}\sum_{m_v \in I(L_v)_{\rel}} \phi_v(m_v)\\
&\le \left(\prod_{v\in A_{\mathrm{ur}}}\frac{16}{5}q_v^{-N_v}\right)\left(\prod_{v\in A_{\mathrm{rm}}^1}5q_v^{-N_v/2}q_v^{n/2}\right)\left(\prod_{v\in A_{\mathrm{rm}}^2}2q_v^{n/2}\right)\\
&= N(L)^{-\epsilon}\left(\prod_{v\in A_{\mathrm{ur}}}\frac{16}{5}q_v^{-(1-\epsilon)N_v}\right)\left(\prod_{v\in A_{\mathrm{rm}}^1}5q_v^{-(3/2-\epsilon)N_v}q_v^{(n+1)/2}\right)\left(\prod_{v\in A_{\mathrm{rm}}^2}2q_v^{-1/2}q_v^{(n+1)/2}\right)\\
&\le N(L)^{-\epsilon}\left(\prod_{v\in A_{\mathrm{ur}}}\frac{16}{5}q_v^{-(1-\epsilon)}\right)\left(\prod_{v\in A_{\mathrm{rm}}^1}5q_v^{-(3/2-\epsilon)}q_v^{(n+1)/2}\right)\left(\prod_{v\in A_{\mathrm{rm}}^2}2q_v^{-1/2}q_v^{(n+1)/2}\right)\\
&= D^{(n+1)/2}N(L)^{-\epsilon}\left(\prod_{v\in A_{\mathrm{ur}}}\frac{16}{5}q_v^{-(1-\epsilon)}\right)\left(\prod_{v\in A_{\mathrm{rm}}^1}5q_v^{-(3/2-\epsilon)}\right)\left(\prod_{v\in A_{\mathrm{rm}}^2}2q_v^{-1/2}\right)
    \end{align*}

Now, noting that $2\in A_{\mathrm{ur}}$, take a constant $\epsilon_1>0$ so that we have $16/5 \cdot 2^{-(1-\epsilon_1)} < \sqrt{3}$.
We also take $\epsilon_{2}>0$ so that
\[
\max\left\{
\frac{16}{5} \cdot 3^{-(1 - \epsilon_{2})}, 5 \cdot 3^{-(3/2-\epsilon_{2})}, \frac{2}{\sqrt{3}}
\right\} = \frac{2}{\sqrt{3}}.
\]
Finally, we take $\epsilon_{3} > 0$ so that we have
\[
\max\left\{
\frac{16}{5} \cdot 5^{-(1 - \epsilon_{3})}, 5 \cdot 5^{-(3/2-\epsilon_{3})}
\right\} < 1.
\]
Then taking $\epsilon \defeq \min\{\epsilon_{1}, \epsilon_{2}, \epsilon_{3}\}$, we conclude that
\[
\prod_{v\nmid\infty}\sum_{m_v \in I(L_v)_{\rel}} \phi_v(m_v)  < \sqrt{3} \cdot \frac{2}{\sqrt{3}} D^{(n+1)/2}N(L)^{-\epsilon} = 2 D^{(n+1)/2}N(L)^{-\epsilon}.
\]

The case where $n$ is odd is similar.

\end{proof}

Now, we obtain the upper bound of the sum
\[
\sum_{[l]\in\mathcal{R}_{\mathrm{s}}}\frac{\mu_{\infty}(\SU(1,n-1)/\SU(L^{l}))}{\mu_{\infty}(\SU(1,n)/\SU(L))}.
\]

\begin{prop}
\label{prop:sumofsplitvectorsglobal}
    The ratio of the covolumes is bounded as
    \[\sum_{[l]\in\mathcal{R}_{\mathrm{s}}}\frac{\mu_{\infty}(\SU(1,n-1)/\SU(L^{l}))}{\mu_{\infty}(\SU(1,n)/\SU(L))} < \frac{2^2\cdot(2\pi)^{n+1}}{n!\cdot D^{n/2}\cdot N(L)^{\epsilon}}.\]
\end{prop}
\begin{proof}
Let us work on the case of even $n$.
According to Theorem~\ref{thm:prasad'svolumeformulaspecialunitarycase} and
Lemma \ref{lem:global summation for reflective vectors}, we have
\begin{align*}
\sum_{[l]\in\mathcal{R}_{\mathrm{s}}} \frac{\mu_{\infty}(\SU(1,n-1)/\SU(L^{l}))}{\mu_{\infty}(\SU(1,n)/\SU(L))}
    &=\frac{(2\pi)^{n+1}}{D^{n+1/2}\cdot n!\cdot L_{E/\bQ}(n+1)}\cdot\sum_{[l]\in\mathcal{R}_{\mathrm{s}}} \prod_{v\nmid\infty}\frac{\lambda(L_v^{l_v})}{\lambda(L_v)}\\
    &\leq 2D^{(n+1)/2}N(L)^{-\epsilon}\cdot \frac{(2\pi)^{n+1}}{D^{n+1/2}\cdot n!\cdot L_{E/\bQ}(n+1)}\\
        &= \frac{2\cdot(2\pi)^{n+1}}{D^{n/2}\cdot n!\cdot L_{E/\bQ}(n+1)\cdot N(L)^{\epsilon}}.
\end{align*}
Since,
\[\frac{1}{L_{E/\bQ}(n+1)}\leq \frac{1}{2-\zeta(n+1)}<\frac{1}{2-\zeta(3)}<2.\]
This concludes the proof for even $n$.
The proof for odd $n$ proceeds in the same way by substituting Lemma \ref{lem:global summation for reflective vectors} into the expression for 
\[\sum_{[l]\in\mathcal{R}_{\mathrm{s}}} \frac{\mu_{\infty}(\SU(1,n-1)/\SU(L^{l}))}{\mu_{\infty}(\SU(1,n)/\SU(L))}
    =\frac{(2\pi)^{n+1}}{n!\cdot \zeta(n+1)}\cdot\sum_{[l]\in\mathcal{R}_{\mathrm{s}}} \prod_{v\nmid\infty}\frac{\lambda(L_v^{l_v})}{\lambda(L_v)}.
    \qedhere
    \]
\end{proof}

\subsection{Proof of Main results}
In this subsection, we give a quantitative estimation when the inequality in Theorem~\ref{thm:criterion_wps} holds.
\begin{thm}
\label{thm:not_wps}
Let $E$ be an imaginary quadratic field of discriminant $-D$ with $D > 3$ odd.
Then we have
\[
\sum_{[l]\in\mathcal{R}}\frac{\volHM(\U(L)^{l})}{\volHM(\U(L))} < (1+2\cdot 2^{2n+1}+4^{2n+1})\cdot \frac{2^3 \cdot (2\pi)^{n+1}}{n!\cdot D^{n/2}\cdot \max\left\{1, \left(N(L)/4\right)^{\epsilon}\right\}}.
\]
Thus, for any arithmetic subgroup $\Gamma < \U(L)$, the graded algebra $M_*(\Gamma)$ is never free if we have
\[
(1+2\cdot 2^{2n+1}+4^{2n+1})\cdot\frac{2 \cdot  (2\pi)^{n+1}}{(n+1)!\cdot D^{n/2}\cdot \max\left\{1, \left(N(L)/4\right)^{\epsilon}\right\}}< 1.
\]
    In particular, if $n>99$ or $D$ is sufficiently large, then $M_*(\Gamma)$ is not free.
    In the case $E=\bQ(\sqrt{-3})$, the range is replaced with
    \[
    (5+ 4 \cdot 3^{2n+1}+ 3 \cdot 4^{2n+1})\cdot \frac{2 \cdot (2\pi)^{n+1}}{(n+1)!\cdot 3^{n/2}\cdot \max\left\{1, \left(N(L)/9\right)^{\epsilon}\right\}} < 1,
    \]
    which always follows when $n>154$.
\end{thm}
\begin{proof}
    First, let us assume $E\neq\bQ(\sqrt{-3})$.
    We follow the argument in \cite{Mae24}*{Section~7} to reduce the problem of estimating the sum to that of estimating the sum over split vectors. 
    Let $T_{L}$ denote the set of sublattices $L' \subseteq L$ that can be obtained as $L' = L^{l} \oplus l \cO_{E}$ for a reflective vector $l \in L$.
    When $D \equiv 7 \bmod{8}$, let $2 \cO_{E} = \mathfrak{p}_{1} \mathfrak{p}_{2}$ denote the decomposition of $2$ in $\cO_{E}$. 
    Then by a classification of sublattices in $T_{L}$ given in \cite{Mae24}*{Section~3}, which is based on \cite{ma2018kodaira}*{Section~4.1}, we have
    \[
    T_{L} = \{L\} \amalg T_{L, 2} \amalg T_{L, 4},
    \]
    where 
    \[
    T_{L, 2} = 
    \begin{cases}
    \left\{
L' \in T_{L} \mid L/L' \simeq \cO_{E} / \mathfrak{p}_{i}, (i = 1,2)
    \right\} & (D \equiv 7 \bmod{8}), \\
    \emptyset & (\text{otherwise}),
    \end{cases}
    \]
    and
    \[
    T_{L, 4} = 
    \left\{
L' \in T_{L} \mid L/L' \simeq \cO_{E} / 2 \cO_{E}
\right\}.
\]
We note that using the notation in \cite{Mae24}, we have $T_{L, 2} = T_{L}(F, 2)_{IV} \amalg T_{L}(F, 2)_{V}$ and $T_{L, 4} = T_{L}(F, 2)_{II}$.
    For $L' \in T_{L}$. let $\mathcal{R}(L')_{\mathrm{s}}$ denote the set of $\U(L')$-equivalence classes of split $\U(L')$-reflective vectors in $L'$.
    By a similar computation as in \cite{Mae24}*{Lemma~7.2}, we have
    \begin{equation}
    \label{eq:similartoMae24lem7.2}
    \sum_{[l]\in\mathcal{R}}\frac{\volHM(\U(L)^{l})}{\volHM(\U(L))} \le 
    \sum_{L' \in T_{L}} \sum_{[l'] \in \mathcal{R}(L')_{\mathrm{s}}} 
    [\U((L')^{l'}) : \U(L)^{l'}] \frac{\volHM(\U((L')^{l'}))}{\volHM(\U(L'))}.
    \end{equation}
    Regarding the first factor in the sum above, according to \cite{Mae24}*{Lemma~3.6}, we have
    \[
    [\U((L')^{l'}) : \U(L)^{l'}] \le
    [\U((L')^{l'}) : \U(L)_{l'}] \le
    \begin{cases}
    1 & (L' = L), \\
    2^{n} & (L' \in T_{L, 2}), \\
    4^{n} & (L' \in T_{L, 4}).
    \end{cases}
    \]
    We will give an estimation of the sum of second factors.
    According to Corollary~\ref{corcovolofSUvsHMvolofU}, we have
\[
\sum_{[l']\in \mathcal{R}(L')_{\mathrm{s}}} \frac{\volHM(\U((L')^{l'}))}{\volHM(\U(L'))}
\le
2 \sum_{[l']\in \mathcal{R}(L')_{\mathrm{s}}} \frac{\mu_{\infty}(\SU(1,n-1)/\SU((L')^{l'}))}{\mu_{\infty}(\SU(1,n)/\SU(L'))}.
\]
For each $L' \in T_{L}$, we obtain from Proposition~\ref{prop:sumofsplitvectorsglobal} that
    \[
    \sum_{l' \in \mathcal{R}(L')_{\mathrm{s}}} \frac{\mu_{\infty}(\SU(1,n-1)/\SU((L')^{l'}))}{\mu_{\infty}(\SU(1,n)/\SU(L'))} \le \frac{2^2 \cdot (2\pi)^{n+1}}{n!\cdot D^{n/2}\cdot N(L')^{\epsilon}}.
    \]
    We claim that $N(L') \ge N(L)/4$.
    Indeed, for a finite place $v \neq 2$, we obtain from the definitions of $T_{L, 2}$ and $T_{L, 4}$ that $L_{v} = L'_{v}$. 
    In particular, we have $N(L_v) = N(L'_{v})$.
    Now, we consider the case $v = 2$.
    According to the definitions of $T_{L, 2}$ and $T_{L, 4}$, we have $2 L \subseteq L' \subseteq L$.
    Hence, noting that $\mathfrak{p}_{E_{2}} = (2)$, we obtain from Lemma~\ref{lem:descriptionofmandM} that 
    \[
    m(I(L'_2)_{\rel}) \le m(I(2 L_2)_{\rel}) = m(I(L_2)_{\rel}) + 2
    \qquad
    \text{and}
    \qquad
    M(I(L'_2)_{\rel}) \ge M(I(L_2)_{\rel}).
    \]
    Thus, we have $N(L'_{2}) \ge N(L_{2}) - 2$.
    Combining the arguments above, we conclude that $N(L') \ge q_{2}^{-2} \cdot N(L) = N(L)/4$.

    Moreover, we obtain from \cite{Mae24}*{(7.1)} that $\abs{T_{L, 2}} < 2 \cdot 2^{n+1}$ and $\abs{T_{L, 4}} < 4^{n+1}$.
    Substituting them into \eqref{eq:similartoMae24lem7.2}, we obtain that
    \[
    \sum_{[l]\in\mathcal{R}}\frac{\volHM(\U(L)^{l})}{\volHM(\U(L))} < (1+2 \cdot 2^{2n+1}+4^{2n+1})\cdot \frac{2^3 \cdot (2\pi)^{n+1}}{n!\cdot D^{n/2}\cdot \max\left\{1, \left(N(L)/4\right)^{\epsilon}\right\}}.
    \]
Hence, we obtain the first claim.
The second claim follows by combing this with Theorem \ref{thm:criterion_wps} and Lemma~\ref{lem:inequality of the ratio between two groups}, noting that $r_{l} = 2$ for all $l \in \mathcal{R}$.

By computer-based calculations, we obtain that
\begin{align*}
& \qquad (1+2 \cdot 2^{2n+1}+4^{2n+1})\cdot \frac{2 \cdot (2\pi)^{n+1}}{(n+1)!\cdot D^{n/2}\cdot \max\left\{1, \left(N(L)/4\right)^{\epsilon}\right\}} \\
&\le (1+2 \cdot 2^{2n+1}+4^{2n+1})\cdot\frac{2 \cdot (2\pi)^{n+1}}{(n+1)!\cdot 7^{n/2}}< 1
\end{align*}
for $n > 99$.
Thus, we obtain the third claim.

            Let us treat the remaining case $E=\bQ(\sqrt{-3})$.
As in the argument above, by using \cite{Mae24}*{Lemma 3.2}, we divide 
    \[
    T_{L} = \{L\} \amalg T_{L, 3} \amalg T_{L, 4},
    \]
where if $L'\in T_{L,3}$ (resp.\ $T_{L,4}$), then $L/L'$ is isomorphic to $\cO_{E}/ \sqrt{-3} \cO_{E}$ (resp.\ $\cO_{E}/2 \cO_{E}$).
Unlike the cases $E\neq \bQ(\sqrt{-3})$, the ramification index $r_{l}$ depends on the reflective vector $l$.
More precisely, we have $r_{l} = 6$ if $l$ is a split vector and $r_{l} = 3$ (resp.\ $r_{l} = 2$) if we have $L^{l} \oplus l \cO_{E} \in T_{L, 3}$ (resp.\ $L^{l} \oplus l \cO_{E} \in T_{L, 4}$).
According to \cite{Mae24}*{Lemma 3.6}, for $l' \in \mathcal{R}(L')_{\mathrm{s}}$, we have
    \[
    [\U((L')^{l'}) : \U(L)^{l'}] \le
    \begin{cases}
    1 & (L' = L), \\
    3^{n} & (L' \in T_{L, 3}),\\
    4^{n} & (L' \in T_{L, 4}).
    \end{cases}
    \]
Combined this with Corollary~\ref{corcovolofSUvsHMvolofU} and the inequalities $\abs{T_{L,3}} < 3^{n+1}$ and $\abs{T_{L,4}} < 4^{n+1}$, which follow from \cite{Mae24}*{(7.1)}, similar computation as above shows that
\[
\sum_{[l]\in\mathcal{R}}\frac{r_{l}-1}{r_{l}}\frac{\volHM(\U(L)^{l})}{\volHM(\U(L))}
< (5+ 4 \cdot 3^{2n+1}+ 3 \cdot 4^{2n+1})\cdot \frac{2^2 \cdot (2\pi)^{n+1}}{n!\cdot 3^{n/2}\cdot \max\left\{1, \left(N(L)/9 \right)^{\epsilon}\right\}}.
\]
The right-hand side can be bounded by $2(n+1)$ when $n>154$.
It concludes the proof.
\end{proof}

If the lattice $L$ is unimodular, we obtain a stronger result.
In the following theorem only, we allow $E$ to be any imaginary quadratic field other than $\bQ(\sqrt{-1})$ or $\bQ(\sqrt{-3})$.

\begin{thm}
\label{thm:unimodular_evaluation}
Let $E$ be an imaginary quadratic field other than $\bQ(\sqrt{-1})$ or $\bQ(\sqrt{-3})$ and $L$ be a unimodular Hermitian lattice over $\cO_{E}$ of signature $(1, n)$ for $n > 2$.
Then the algebra $M_*(\Gamma)$ of modular forms for any arithmetic subgroup $\Gamma < \U(L)$ is never free.
\end{thm}
\begin{proof}
    By \cite{maeda2023fano}*{Theorem~3.25} or \cite{WW2021}*{Lemma~2.2 (3)}, there are no branch divisors if $E\neq\bQ(\sqrt{-1}), \bQ(\sqrt{-3})$ and $n>2$.
    Then, the left-hand side of the inequality in Theorem \ref{thm:criterion_wps} is zero, which proves the claim.
\end{proof}

\begin{rmk}
For any unimodular lattice $L$, it is known from \cite{james1992orbits}*{Theorem~1.1} that the set consisting of representing numbers for a fixed element in $\mathcal{O}_E$ has only one $\U(L)$-orbit.
From this point of view, also for the case of $E=\bQ(\sqrt{-1}), \bQ(\sqrt{-3})$, we can count the number of ramification divisors and prove stronger estimation than the one in Theorem \ref{thm:not_wps}.
More directly, one can also apply the classification results \cite{WW2021}*{Lemma 2.2 (1), (2)}.
\end{rmk}

We conclude this section with a finiteness statement regarding Hermitian lattices.
\begin{thm}
\label{thm:finiteness}
    Up to scaling, there are only finitely many isometry classes of  Hermitian lattices $L$ of signature
$(1, n)$ over $\cO_E$, where $n > 2$ and $E$ is an imaginary quadratic field with odd discriminant, such that $M_*(\Gamma)$ is a free algebra for some arithmetic subgroup $\Gamma < \U(L)$.
\end{thm}
\begin{proof}
    Assume that $E\neq\bQ(\sqrt{-3})$; the remaining cases follow from a similar argument.
According to Theorem~\ref{thm:not_wps}, $M_*(\Gamma)$ is not a free algebra for any $\Gamma < \U(L)$ if we have
    \[(1+2\cdot 2^{2n+1}+4^{2n+1})\cdot\frac{2 \cdot  (2\pi)^{n+1}}{(n+1)!\cdot D^{n/2}}<  \max\left\{1, \left(N(L)/4\right)^{\epsilon}\right\}.\]
    Noting that the left-hand side of the inequality tends to zero as either $n$ and $D$ tends to infinity, while the right-hand side remains bounded below by $1$, we conclude that only finitely many triples $(n, D, N(L)) \in \mathbb{Z}^3$ can violate the inequality.
    Thus, to complete the proof, it suffices to show that, up to scaling, there are only finitely many isometry classes of Hermitian lattices $L$ for fixed $n$, $E$, and $N(L)$.
    After scaling, we may suppose that $0 \in I(L_{v})_{\rel}$ for any $v \in V_{\fin}$ that is split or unramified over $E$, and
    \[
    \{0, 1\} \cap I(L_{v})_{\rel} \neq \emptyset
    \]
    for any $v \in V_{\fin}$ that ramifies over $E$.
Under this assumption, once $N(L)$ is bounded by a fixed constant, the determinant of the lattice $L$ is also bounded above by some constant.
Moreover, it is known that for a fixed matrix size $n$ over an imaginary quadratic field $E$, there are only finitely many unitary conjugacy classes of Hermitian matrices with coefficients in $\mathcal{O}_E$ whose determinant bounded by a fixed constant from the reduction theory \cite{kitaoka1993arithmetic}.
Thus, we obtain the claim.
\end{proof}

\section{Applications}
\label{sec:concrete examples}

In this section, we give two applications of our volume computation.

\subsection{Finiteness of reflective modular forms}
\label{subsec:Finiteness of reflective modular forms}
Gritsenko and Nikulin \cite{gritsenko1998automorphic}*{Conjecture 2.5.5} conjectured that up to scaling, there are only finitely many quadratic lattices admitting reflective modular forms with fixed slope and certain conditions, which was later proved by Ma \cite{ma2018kodaira}*{Corollaries 1.9, 1.10}.
Using the techniques introduced in this paper, we can prove an analogous result for unitary groups (see Subsection \ref{subsec:reflective modular forms} for the relevant definition).

Recall that for a reflective modular form $f$ of weight $\kappa$ with respect to an arithmetic subgroup $\Gamma<\U(L)$, its slope $\rho(f)$ is defined as $\rho(f) \defeq \max\{a_l/\kappa\}$ putting $\div(f) = \sum_{[l]\in\mathcal{R}^{\Gamma}} a_lH_l$.
Let $g(n,D)$ be the function on $n$ and $D$ defined by the \emph{inverse} of 
\[\frac{2^2  \cdot  (2\pi)^{n+1}}{n!\cdot D^{n/2}}\cdot
\begin{cases}
2\cdot (1+2\cdot 2^{2n+1}+4^{2n+1})&(E\neq \bQ(\sqrt{-3})),\\
6 \cdot (1+ 3^{2n+1}+  4^{2n+1})&(E= \bQ(\sqrt{-3})).
\end{cases}\]
\begin{thm}
\label{thm:finiteness of reflective modular forms}
Let $E$ be an imaginary quadratic field with odd discriminant $-D$.
\begin{enumerate}
    \item There exist no reflective modular forms $f$ such that $\rho(f) \le g(n,D)$.
\item Let $r>0$ be a fixed rational number.
Up to scaling, there are only finitely many isometry classes of Hermitian lattices $L$ of signature $(1,n)$ over $\cO_E$, where $n>2$ and $E$ is an imaginary quadratic field with odd discriminant, such that there exists a reflective modular form $f$ for some arithmetic subgroup $\Gamma < \U(L)$ with $\rho(f) \le r$.
\end{enumerate}
\end{thm}

\begin{proof}
    (1) Suppose that $f$ is a reflective modular form of weight $\kappa$.
    Denote by $\div(f) = \sum_{[l]\in\mathcal{R}^{\Gamma}}a_l H_l$ its divisor.
    In a similar way as the proof of Theorem \ref{thm:criterion_wps}, the existence of $f$ leads the volume relation
    \[\sum_{[l]\in\mathcal{R}^{\Gamma}}(a_l-1)\frac{\volHM(\Gamma^l)}{\volHM(\Gamma)} = \kappa.\]

     Hence, combined with Lemma~\ref{lem:inequality of the ratio between two groups} and Theorem \ref{thm:not_wps}, we must have
     \[
         1 = \sum_{[l]\in\mathcal{R}}\frac{a_l-1}{\kappa}\cdot\frac{\volHM(\Gamma^l)}{\volHM(\Gamma)} <   \rho(f) \sum_{[l]\in\mathcal{R}}\frac{\volHM(\Gamma^l)}{\volHM(\Gamma)} < \frac{\rho(f)}{g(n,D) \max\left\{1, \left(N(L)/4\right)^{\epsilon}\right\}} \le \frac{\rho(f)}{g(n, D)}.
     \]
     Thus, we conclude that $g(n, D) < \rho(f)$.

Item (2) can be proved similarly to Theorem \ref{thm:finiteness}.
\end{proof}
Item (1) can be rephrased as follows: for any given rational number $r>0$, if either $n$ or $D$ is sufficiently large, which depends on $r$, there exists no reflective modular form with slope $\rho(f)\le r$.
Certain reflective modular forms on $\O^+(2,n)$ whose vanishing order at each ramification divisor is equal to $1$ are historically called \emph{Lie reflective modular forms} \cite{gritsenko1998automorphic}*{Definition~2.5.4}.
Considering that each ramification index $r_l$ is $2$ for $\O^+(2,n)$, it is natural to interpret this quantity $1$ as $r_l-1$.
In general, reflective modular forms on $\O^+(2,n)$ or $\U(1,n)$, which vanish along $H_l$ with multiplicity $r_l-1$, are referred to as \textit{special reflective modular forms} \cite{maeda2023fano}*{Assumption~2.1}, and can be regarded as a certain generalisation of Lie reflective modular forms.
A fundamentally algebraic geometrical reason to consider such modular forms is that they encode birational properties of $\overline{\Gamma\backslash\mathcal{D}}$.
If there exists a special reflective modular form with $\rho(f) < 1/(n+1)$, then $\overline{X_{\Gamma}}$ is Fano \cite{maeda2023fano}*{Theorem~2.4}.
Here, the quantity $n+1$ is referred to as the \emph{canonical weight} of the unitary group $\U(1,n)$.
The pluricanonical forms on the unique toroidal compactification of $X_{\Gamma}$ must come from modular forms whose weight is a multiple of $n+1$.
The modular forms appearing in this paper (Theorem \ref{thm:criterion_wps}) are one of the typical examples; see the proof of Theorem \ref{thm:criterion_wps}.
By computing the range of $g(n,D) \geq 1$ for which such forms may exist, one sees reflective modular forms with slope $\rho(f) < 1/(n+1)$ do not exist when $n>100$ for $E\neq \bQ(\sqrt{-3})$ (resp.\ $n> 155$ for $E=\bQ(\sqrt{-3})$).

\subsection{Moduli of cubic threefolds}

We apply our method to show that the graded algebra of modular forms associated with the moduli space of cubic threefolds is not freely generated.
We first review two simpler cases: cubic curves and cubic surfaces \cite{casalaina2023cohomology}*{Appendix~C1, C2}. It is well known that the GIT moduli space of cubic curves is isomorphic to $\bP^1$, which coincides with the Baily–Borel compactification of $\SL_2(\bZ)\backslash\bH$. 
This geometric property reflects the classical result that $M_*(\SL_2(\bZ))$ is isomorphic to $\bC[E_4, E_6]$, which is generated by two Eisenstein series $E_4$ and $E_6$ of weight 4 and 6, respectively.
Next, let us focus on the moduli space of (unmarked) cubic surfaces.
It admits a 4-dimensional complex ball quotient realisation \cite{allcock2002complex}, whose Baily–Borel compactification is identified with the weighted projective space
$\bP(1,2,3,4,5)$
by classical invariant theory \cite{dolgachev2005complex}.
Moreover, the work by Wang and Williams \cite{WW2021}*{Theorem~5.19} shows that the associated graded algebra is free, and there exists a special reflective modular form of weight $95$ on this moduli space \cite{allcock2002cubic}*{Theorem~4.7}.
 In fact, one can check that the inequality~(\ref{ineq:non-freeness main results}) fails in this case.

The moduli space of cubic threefolds is also constructed as a GIT quotient, namely $H^0(\bP^4,\cO(3))/\!/\SL_5(\bC)$, and is thus unirational.
By the work of \cite{allcock2011moduli},  it is also known that there is a period map from this moduli space to a 10-dimensional ball quotient. 
It is therefore natural to ask whether phenomena similar to the cases of cubic curves and surfaces, specifically the freeness of the algebra of modular forms, also occur in such a higher-dimensional case.
However, this expectation does not hold.
Weighted projective spaces must have the same cohomology as ordinary projective spaces, which fails for the moduli space of cubic threefolds \cite{casalaina2023cohomology}*{Theorem~1.1}.
Nevertheless, since computing the cohomology of arithmetic varieties is quite difficult in general, it is reasonable to pursue an alternative view, as we shall demonstrate below.

Let us recall the setting and use the notation in \cite{casalaina2023cohomology}*{Chapter~7}.
Let $L_{\mathrm{cub}}\defeq \mathcal{E}_1\oplus \mathcal{E}_4^{\oplus 2}\oplus \mathcal{H}$ be a Hermitian lattice of signature $(1,10)$ over the ring of Eisenstein integers; see \cite{casalaina2023cohomology}*{Subsection~7.1.1} for the definition of the concrete form of the Hermitian lattices.
Note that the associated quadratic form over $\bZ$ is $A_2(-1)\oplus E_8^{\oplus 2}(-1) \oplus U^{\oplus 2}$.
According to \cite{allcock2011moduli}, the associated ball quotient $\U(L_{\mathrm{cub}})\backslash\bB^{10}$ has an interpretation as a moduli space of cubic threefolds.
From now on, we shall work on the Baily-Borel compactification $\overline{X}_{\mathrm{cub}}\defeq \overline{\U(L_{\mathrm{cub}})\backslash\bB^{10}}$.

By lattice-theoretic computation, there are two reflective vectors $l_{\mathrm{n}}$ (split) and $l_{\mathrm{h}}$ (non-split).
Corresponding to these vectors, there are two Heegner divisors $\overline{H}_{\mathrm{h}}$ and  $\overline{H}_{\mathrm{n}}$ on $\overline{X}_{\mathrm{cub}}$.
It is known that these divisors have natural geometrical meanings, parametrising specific cubic threefolds; see \cite{casalaina2023cohomology}*{Subsection~8.1}.
In our context, these two divisors are branch divisors with indices 3 and 6.
Let us define 
\[L_{\mathrm{n}}\defeq \mathcal{E}_4^{\oplus 2} \oplus \mathcal{G},\ L_{\mathrm{h}}\defeq \mathcal{E}_1\oplus \mathcal{E}_3\oplus \mathcal{E}_4\oplus \mathcal{G}\] 
the corresponding Hermitian lattices of signature $(1,9)$.
Here we put 
\[\mathcal{G}\defeq\begin{pmatrix}
    0 & 3 \\
3 & 0 \\
\end{pmatrix}.\]
Note that the corresponding quadratic forms of $L_{\mathrm{n}}$ and $L_{\mathrm{h}}$ are $A_2(-1)\oplus E_8(-1)\oplus E_6(-1) \oplus U^{\oplus 2} \cong E_8(-1)^{\oplus 2} \oplus U \oplus U(3)$ and $A_2(-1)\oplus E_8(-1)\oplus E_6(-1) \oplus U\oplus U(3)$.

Now, we shall compute the volume of these lattices.
Clearly, possible non-trivial quantities of the ratio $\lambda((L_{\mathrm{n}})_v)/\lambda((L_{\mathrm{cub}})_v)$ and $\lambda((L_{\mathrm{h}})_v)/\lambda((L_{\mathrm{cub}})_v)$ appear in the finite place $v=3$.
For such a $v$, the orthogonal decompositions are
\[
    (L_{\mathrm{cub}})_v = (L_{\mathrm{cub}})_{v,0}\oplus (L_{\mathrm{cub}})_{v,2},\quad 
    (L_{\mathrm{n}})_v = (L_{\mathrm{n}})_{v,0}\oplus (L_{\mathrm{n}})_{v,2
},\quad 
    (L_{\mathrm{h}})_v= (L_{\mathrm{h}})_{v,0}\oplus (L_{\mathrm{h}})_{v,2}\]
where $ \mathrm{rk}((L_{\mathrm{cub}})_{v,0}) = 10$, 
$(L_{\mathrm{cub}})_{v,2}=(L_{\mathrm{h}})_{v,2}$ has rank 1, 
$\mathrm{rk}((L_{\mathrm{n}})_{v,0}) = 9$,  
$\mathrm{rk}((L_{\mathrm{h}})_{v,0}) = 8$, and  
$\mathrm{rk}((L_{\mathrm{h}})_{v,2})=2$.
By definition, all three associated intermediate lattices $M_{(L_{\mathrm{cub}})_v}, M_{(L_{\mathrm{n}})_v}$ and $M_{(L_{\mathrm{h}})_v}$ are unimodular of rank 11, 10 and 10.
Then since $v$ ramifies $\bQ(\sqrt{-3})$, considering the case of $(\mu_p,n_{p,\mu_p})=(\mathrm{even},\mathrm{odd})$ in  \cite{Mae24}*{Subsection~7.2}, a straightforward computation implies that 
\[\lambda((L_{\mathrm{n}})_3)/\lambda((L_{\mathrm{cub}})_3)\leq \lambda(M_{(L_{\mathrm{n}})_3})/\lambda((M_{L_{\mathrm{cub}})_3})=1\] and \[\lambda((L_{\mathrm{h}})_3)/\lambda((L_{\mathrm{cub}})_3)\leq \lambda(M_{(L_{\mathrm{h}})_3}))/\lambda(M_{(L_{\mathrm{cub}})_3})=1.\]
Hence, the inequality 
\[\frac{(2\pi)^{11}}{3^{10+1/2}\cdot 10!\cdot L(11)}\left(\frac{5}{6}+\frac{2}{3}\right)<22\]
implies the following.
\begin{prop}
\label{prop:cubic threefolds wps}
The graded algebra of modular forms $M_*(\U(L_{\mathrm{cub}}))$ is not free.
\end{prop}

Note that, as far as the authors' knowledge, the concrete description of the structure of $\overline{X}_{\mathrm{cub}}$ is not known in terms of invariant ring theory.

\bibliographystyle{amsalpha}
\bibliography{main}

\begin{bibdiv}
\begin{biblist}

\bib{allcock2000complex}{article}{
      author={Allcock, Daniel},
      author={Carlson, James},
      author={Toledo, Domingo},
       title={The complex hyperbolic geometry of the moduli space of cubic surfaces},
        date={2000},
     journal={arXiv preprint arXiv:0007048},
}

\bib{allcock2002complex}{article}{
      author={Allcock, Daniel},
      author={Carlson, James},
      author={Toledo, Domingo},
       title={The complex hyperbolic geometry of the moduli space of cubic surfaces},
        date={2002},
     journal={J. Algebr. Geom.},
      volume={11},
      number={4},
       pages={659\ndash 724},
}

\bib{allcock2011moduli}{book}{
      author={Allcock, Daniel},
      author={Carlson, James~A.},
      author={Toledo, Domingo},
       title={The moduli space of cubic threefolds as a ball quotient},
      series={Mem. Am. Math. Soc.},
   publisher={Providence, RI: American Mathematical Society (AMS)},
        date={2011},
      volume={985},
        ISBN={978-0-8218-4751-0; 978-1-4704-0599-1},
}

\bib{allcock2002cubic}{article}{
      author={Allcock, Daniel},
      author={Freitag, Eberhard},
       title={Cubic surfaces and {B}orcherds products},
        date={2002},
     journal={Comment. Math. Helv.},
      volume={77},
      number={2},
       pages={270\ndash 296},
}

\bib{aoki2005simple}{article}{
      author={Aoki, Hiroki},
      author={Ibukiyama, Tomoyoshi},
       title={Simple graded rings of {S}iegel modular forms, differential operators and {B}orcherds products},
        date={2005},
     journal={Int. J. Math.},
      volume={16},
      number={03},
       pages={249\ndash 279},
}

\bib{baily1966compactification}{article}{
      author={Baily, Walter~L},
      author={Borel, Armand},
       title={Compactification of arithmetic quotients of bounded symmetric domains},
        date={1966},
     journal={Ann. Math.},
       pages={442\ndash 528},
}

\bib{behrens2012singularities}{article}{
      author={Behrens, Niko},
       title={Singularities of ball quotients},
        date={2012},
     journal={Geometriae Dedicata},
      volume={159},
      number={1},
       pages={389\ndash 407},
}

\bib{belolipetsky2004volumes}{article}{
      author={Belolipetsky, Mikhail},
       title={On volumes of arithmetic quotients of $\mathrm{SO} (1, n)$},
        date={2004},
     journal={Ann. Sc. Norm. Super. Pisa, Cl. Sci. (5)},
      volume={3},
      number={4},
       pages={749\ndash 770},
}

\bib{borcherds1998automorphic}{article}{
      author={Borcherds, Richard~E.},
       title={Automorphic forms with singularities on {Grassmannians}},
        date={1998},
        ISSN={0020-9910},
     journal={Invent. Math.},
      volume={132},
      number={3},
       pages={491\ndash 562},
}

\bib{bruinier2004}{article}{
      author={Bruinier, Jan~Hendrik},
       title={Two applications of the curve lemma for orthogonal groups},
        date={2004},
        ISSN={0025-584X},
     journal={Math. Nachr.},
      volume={274-275},
       pages={19\ndash 31},
}

\bib{chevalley1955invariants}{article}{
      author={Chevalley, Claude},
       title={Invariants of finite groups generated by reflections},
        date={1955},
     journal={Am. J. Math.},
      volume={77},
      number={4},
       pages={778\ndash 782},
}

\bib{casalaina2023cohomology}{book}{
      author={Casalaina-Martin, Sebastian},
      author={Grushevsky, Samuel},
      author={Hulek, Klaus},
      author={Laza, Radu},
       title={Cohomology of the moduli space of cubic threefolds and its smooth models},
      series={Mem. Am. Math. Soc.},
   publisher={Providence, RI: American Mathematical Society (AMS)},
        date={2023},
      volume={1395},
        ISBN={978-1-4704-6020-4; 978-1-4704-7351-8},
}

\bib{casalaina2024nonisomorphic}{article}{
      author={Casalaina-Martin, Sebastian},
      author={Grushevsky, Samuel},
      author={Hulek, Klaus},
      author={Laza, Radu},
       title={Non-isomorphic smooth compactifications of the moduli space of cubic surfaces},
        date={2024},
        ISSN={0027-7630},
     journal={Nagoya Math. J.},
      volume={254},
       pages={315\ndash 365},
}

\bib{casalaina2012genusfour}{incollection}{
      author={Casalaina-Martin, Sebastian},
      author={Jensen, David},
      author={Laza, Radu},
       title={The geometry of the ball quotient model of the moduli space of genus four curves},
        date={2012},
   booktitle={Compact moduli spaces and vector bundles. conference on compact moduli and vector bundles},
   publisher={Providence, RI: American Mathematical Society (AMS)},
       pages={107\ndash 136},
}

\bib{casalaina2009degenerations}{article}{
      author={Casalaina-Martin, Sebastian},
      author={Laza, Radu},
       title={The moduli space of cubic threefolds via degenerations of the intermediate {Jacobian}},
        date={2009},
        ISSN={0075-4102},
     journal={J. Reine Angew. Math.},
      volume={633},
       pages={29\ndash 65},
}

\bib{dolgachev2005complex}{article}{
      author={Dolgachev, Igor},
      author={Geemen, B~van},
      author={Kond{\=o}, Shigeyuki},
       title={A complex ball uniformization of the moduli space of cubic surfaces via periods of ${K3}$ surfaces},
        date={2005},
        ISSN={0075-4102},
     journal={J. Reine Angew. Math.},
      volume={588},
       pages={99\ndash 148},
}

\bib{dvzambic2024siegel}{article}{
      author={D{\v{z}}ambi{\'c}, Amir},
      author={Holm, Kristian},
      author={K{\"o}hl, Ralf},
       title={The {S}iegel modular group is the lattice of minimal covolume in the symplectic group},
        date={2024},
     journal={arXiv preprint arXiv:2402.07604},
}

\bib{dern2003graded}{article}{
      author={Dern, Tobias},
      author={Krieg, Aloys},
       title={Graded rings of {H}ermitian modular forms of degree $2$},
        date={2003},
     journal={Manuscr. Math.},
      volume={110},
      number={2},
       pages={251\ndash 272},
}

\bib{dern2004graded}{article}{
      author={Dern, Tobias},
      author={Krieg, Aloys},
       title={The graded ring of {H}ermitian modular forms of degree $2$ over $\mathbb{Q} (\sqrt{- 2})$},
        date={2004},
     journal={J. Number Theory},
      volume={107},
      number={2},
       pages={241\ndash 265},
}

\bib{deligne1986monodromy}{article}{
      author={Deligne, P.},
      author={Mostow, G.~D.},
       title={Monodromy of hypergeometric functions and non-lattice integral monodromy},
        date={1986},
        ISSN={0073-8301},
     journal={Publ. Math., Inst. Hautes {\'E}tud. Sci.},
      volume={63},
       pages={5\ndash 89},
}

\bib{Dolgachev1982weighted}{inproceedings}{
      author={Dolgachev, Igor},
       title={Weighted projective varieties},
        date={1982},
   booktitle={Group actions and vector fields},
      editor={Carrell, James~B.},
   publisher={Springer Berlin Heidelberg},
     address={Berlin, Heidelberg},
       pages={34\ndash 71},
}

\bib{emery2014covolumes}{article}{
      author={Emery, Vincent},
      author={Stover, Matthew},
       title={Covolumes of nonuniform lattices in $\mathrm{PU}(n, 1)$},
        date={2014},
     journal={Am. J. Math.},
      volume={136},
      number={1},
       pages={143\ndash 164},
}

\bib{freitag2011modular}{article}{
      author={Freitag, Eberhard},
      author={Manni, Riccardo},
       title={The modular variety of hyperelliptic curves of genus three},
        date={2011},
     journal={Trans. Am. Math. Soc.},
      volume={363},
      number={1},
       pages={281\ndash 312},
}

\bib{freitag2019vector}{article}{
      author={Freitag, Eberhard},
      author={Manni, Riccardo},
       title={Vector-valued modular forms on a three-dimensional ball},
        date={2019},
     journal={Trans. Am. Math. Soc.},
      volume={371},
      number={8},
       pages={5293\ndash 5308},
}

\bib{freitag2002graded}{article}{
      author={Freitag, Eberhard},
       title={A graded algebra related to cubic surfaces},
        date={2002},
     journal={Kyushu J. Math.},
      volume={56},
      number={2},
       pages={299\ndash 312},
}

\bib{freitag1983siegelsche}{book}{
      author={Freitag, E.},
       title={Siegelsche {Modulfunktionen}},
    language={German},
      series={Grundlehren Math. Wiss.},
   publisher={Springer, Cham},
        date={1983},
      volume={254},
}

\bib{gritsenko2005hirzebruch}{article}{
      author={Gritsenko, Valery},
      author={Hulek, Klaus},
      author={Sankaran, Gregory},
       title={The {H}irzebruch-{M}umford volume for the orthogonal group and applications},
        date={2005},
     journal={arXiv preprint math},
        ISSN={0512595/},
}

\bib{gritsenko2006hirzebruch}{article}{
      author={Gritsenko, Valery},
      author={Hulek, Klaus},
      author={Sankaran, Gregory},
       title={Hirzebruch-{M}umford proportionality and locally symmetric varieties of orthogonal type},
        date={2006},
     journal={arXiv preprint math},
        ISSN={0609774/},
}

\bib{gritsenko1998automorphic}{article}{
      author={Gritsenko, Valery},
      author={Nikulin, Viacheslav~V},
       title={Automorphic forms and {L}orentzian {K}ac--{M}oody algebras part {I}},
        date={1998},
     journal={Int. J. Math.},
      volume={9},
      number={02},
       pages={153\ndash 199},
}

\bib{gritsenko2018reflective}{article}{
      author={Gritsenko, Valery},
       title={Reflective modular forms and applications},
        date={2018},
        ISSN={0036-0279},
     journal={Russ. Math. Surv.},
      volume={73},
      number={5},
       pages={797\ndash 864},
}

\bib{Gan-Yu}{article}{
      author={Gan, Wee~Teck},
      author={Yu, Jiu-Kang},
       title={Group schemes and local densities},
        date={2000},
        ISSN={0012-7094,1547-7398},
     journal={Duke Math. J.},
      volume={105},
      number={3},
       pages={497\ndash 524},
         url={https://doi.org/10.1215/S0012-7094-00-10535-2},
      review={\MR{1801770}},
}

\bib{harder1971gauss}{article}{
      author={Harder, G.},
       title={A {Gauss}-{Bonnet} formula for discrete arithmetically defined groups},
        date={1971},
        ISSN={0012-9593},
     journal={Ann. Sci. {\'E}c. Norm. Sup{\'e}r. (4)},
      volume={4},
       pages={409\ndash 455},
         url={https://eudml.org/doc/81886},
}

\bib{hulek2024compactifications}{article}{
      author={Hulek, Klaus},
      author={Kondo, Shigeyuki},
      author={Maeda, Yota},
       title={Compactifications of the eisenstein ancestral deligne-mostow variety},
        date={2024},
     journal={arXiv preprint arXiv:2403.18345},
}

\bib{hulek2025revisiting}{article}{
      author={Hulek, Klaus},
      author={Maeda, Yota},
       title={Revisiting the moduli space of 8 points on {$\mathbb{P}^1$}},
        date={2025},
        ISSN={0001-8708},
     journal={Adv. Math.},
      volume={463},
       pages={41},
        note={Id/No 110126},
}

\bib{hashimoto2022ring}{article}{
      author={Hashimoto, Kenji},
      author={Ueda, Kazushi},
       title={The ring of modular forms for the even unimodular lattice of signature $(2, 10)$},
        date={2022},
     journal={Proc. Am. Math. Soc.},
      volume={150},
      number={2},
       pages={547\ndash 558},
}

\bib{igusa1962siegel}{article}{
      author={Igusa, Jun-ichi},
       title={On {S}iegel modular forms of genus two},
        date={1962},
     journal={Am. J. Math.},
      volume={84},
      number={1},
       pages={175\ndash 200},
}

\bib{igusa1964siegel}{article}{
      author={Igusa, Jun-ichi},
       title={On {S}iegel modular forms of genus two {(II)}},
        date={1964},
     journal={Am. J. Math.},
      volume={86},
      number={2},
       pages={392\ndash 412},
}

\bib{jacobowitz}{article}{
      author={Jacobowitz, Ronald},
       title={Hermitian forms over local fields},
        date={1962},
        ISSN={0002-9327},
     journal={Am. J. Math.},
      volume={84},
       pages={441\ndash 465},
}

\bib{james1992orbits}{article}{
      author={James, Donald~G},
       title={Orbits in unimodular hermitian lattices},
        date={1992},
     journal={Trans. Am. Math. Soc.},
      volume={332},
      number={2},
       pages={849\ndash 860},
}

\bib{MR4008068}{article}{
      author={Kirschmer, Markus},
       title={Determinant groups of {H}ermitian lattices over local fields},
        date={2019},
        ISSN={0003-889X,1420-8938},
     journal={Arch. Math. (Basel)},
      volume={113},
      number={4},
       pages={337\ndash 347},
         url={https://doi.org/10.1007/s00013-019-01348-z},
      review={\MR{4008068}},
}

\bib{kitaoka1993arithmetic}{book}{
      author={Kitaoka, Yoshiyuki},
       title={Arithmetic of quadratic forms},
      series={Camb. Tracts Math.},
   publisher={Cambridge: Cambridge University Press},
        date={1993},
      volume={106},
        ISBN={0-521-40475-4},
}

\bib{kondo2013segre}{incollection}{
      author={Kond{\=o}, Shigeyuki},
       title={The {Segre} cubic and {Borcherds} products},
        date={2013},
   booktitle={Arithmetic and geometry of {$K3$} surfaces and {C}alabi-{Y}au threefolds. proceedings of the workshop, {T}oronto, {C}anada, {A}ugust 16--25, 2011},
   publisher={New York, NY: Springer},
       pages={549\ndash 565},
}

\bib{kondo2016igusa}{incollection}{
      author={Kond{\=o}, Shigeyuki},
       title={The {Igusa} quartic and {Borcherds} products},
        date={2016},
   booktitle={K3 surfaces and their moduli},
   publisher={Basel: Birkh{\"a}user/Springer},
       pages={147\ndash 170},
}

\bib{kottwitz1988tamagawa}{article}{
      author={Kottwitz, Robert~E},
       title={Tamagawa numbers},
        date={1988},
     journal={Ann. Math.},
      volume={127},
      number={3},
       pages={629\ndash 646},
}

\bib{KalethaPrasad}{book}{
      author={Kaletha, Tasho},
      author={Prasad, Gopal},
       title={Bruhat-{T}its theory---a new approach},
      series={New Mathematical Monographs},
   publisher={Cambridge University Press, Cambridge},
        date={2023},
      volume={44},
        ISBN={978-1-108-83196-3},
      review={\MR{4520154}},
}

\bib{looijenga2007period}{article}{
      author={Looijenga, Eduard},
      author={Swierstra, Rogier},
       title={The period map for cubic threefolds},
        date={2007},
     journal={Compos. Math.},
      volume={143},
      number={4},
       pages={1037\ndash 1049},
}

\bib{ma2013finiteness}{misc}{
      author={Ma, Shouhei},
       title={Finiteness of stable orthogonal modular varieties of non-general type},
         how={Preprint, {arXiv}:1309.7121 [math.{AG}] (2013)},
        date={2013},
         url={https://arxiv.org/abs/1309.7121},
}

\bib{ma2017finiteness}{article}{
      author={Ma, Shouhei},
       title={Finiteness of $2$-reflective lattices of signature {{\((2,n)\)}}},
        date={2017},
        ISSN={0002-9327},
     journal={Am. J. Math.},
      volume={139},
      number={2},
       pages={513\ndash 524},
}

\bib{ma2018kodaira}{article}{
      author={Ma, Shouhei},
       title={On the {K}odaira dimension of orthogonal modular varieties},
        date={2018},
     journal={Invent. Math.},
      volume={212},
      number={3},
       pages={859\ndash 911},
}

\bib{Mae24}{article}{
      author={Maeda, Yota},
       title={Reflective obstructions of unitary modular varieties},
        date={2024},
        ISSN={0021-8693},
     journal={J. Algebra},
      volume={647},
       pages={341\ndash 399},
}

\bib{Meschiari1972reflections}{article}{
      author={Meschiari, Mauro},
       title={On the reflections in bounded symmetric domains},
        date={1972},
        ISSN={0036-9918},
     journal={Ann. Sc. Norm. Super. Pisa, Sci. Fis. Mat., III. Ser.},
      volume={26},
       pages={403\ndash 435},
         url={https://eudml.org/doc/83600},
}

\bib{maeda2023fano}{book}{
      author={Maeda, Yota},
      author={Odaka, Yuji},
       title={Fano {Shimura} varieties with mostly branched cusps},
      series={Springer Proc. Math. Stat.},
   publisher={Cham: Springer},
        date={2023},
      volume={409},
        ISBN={978-3-031-17858-0; 978-3-031-17861-0; 978-3-031-17859-7},
}

\bib{mostow1986generalized}{article}{
      author={Mostow, G.~D.},
       title={Generalized {Picard} lattices arising from half-integral conditions},
        date={1986},
        ISSN={0073-8301},
     journal={Publ. Math., Inst. Hautes {\'E}tud. Sci.},
      volume={63},
       pages={91\ndash 106},
}

\bib{mumford1977hirzebruch}{article}{
      author={Mumford, David},
       title={Hirzebruch's proportionality theorem in the non-compact case},
        date={1977},
     journal={Invent. Math.},
}

\bib{mumford2006kodaira}{inproceedings}{
      author={Mumford, David},
       title={On the {K}odaira dimension of the {S}iegel modular variety},
organization={Springer},
        date={1982},
   booktitle={Algebraic geometry—open problems: Proceedings of the conference held in ravello},
       pages={348\ndash 375},
}

\bib{Omeara2000introduction}{book}{
      author={O'Meara, O.~T.},
       title={Introduction to quadratic forms},
      series={Grundlehren Math. Wiss.},
   publisher={Springer, Cham},
        date={1963},
      volume={117},
}

\bib{Pra89}{article}{
      author={Prasad, Gopal},
       title={Volumes of {{\(S\)}}-arithmetic quotients of semi-simple groups. {With} an appendix by {Moshe} {Jarden} and {Gopal} {Prasad}},
        date={1989},
        ISSN={0073-8301},
     journal={Publ. Math., Inst. Hautes {\'E}tud. Sci.},
      volume={69},
       pages={91\ndash 117},
}

\bib{prasad2005fake}{article}{
      author={Prasad, Gopal},
      author={Yeung, Sai-Kee},
       title={Fake projective planes},
        date={2007},
        ISSN={0020-9910},
     journal={Invent. Math.},
      volume={168},
      number={2},
       pages={321\ndash 370},
}

\bib{prasad2009arithmetic}{article}{
      author={Prasad, Gopal},
      author={Yeung, Sai-Kee},
       title={Arithmetic fake projective spaces and arithmetic fake {Grassmannians}},
        date={2009},
        ISSN={0002-9327},
     journal={Am. J. Math.},
      volume={131},
      number={2},
       pages={379\ndash 407},
}

\bib{prasad2012nonexistence}{article}{
      author={Prasad, Gopal},
      author={Yeung, Sai-Kee},
       title={Nonexistence of arithmetic fake compact {Hermitian} symmetric spaces of type other than {$A_n$} $(n\leq 4)$},
        date={2012},
        ISSN={0025-5645},
     journal={J. Math. Soc. Japan},
      volume={64},
      number={3},
       pages={683\ndash 731},
}

\bib{prasad2023arithmeticfakecompacthermitian}{article}{
      author={Prasad, Gopal},
      author={Yeung, Sai-Kee},
       title={Arithmetic fake compact {H}ermitian symmetric spaces of type {$A_3$}},
        date={2023},
     journal={arXiv preprint arXiv:2306.09605},
}

\bib{resnikoff1978structure}{article}{
      author={Resnikoff, Howard~Leonard},
      author={Tai, Yung-Sheng},
       title={On the structure of a graded ring of automorphic forms on the 2-dimensional complex ball},
        date={1978},
     journal={Math. Ann.},
      volume={238},
       pages={97\ndash 117},
}

\bib{savin1989limit}{article}{
      author={Savin, Gordan},
       title={Limit multiplicities of cusp forms.},
        date={1989},
}

\bib{scheithauer2006classification}{article}{
      author={Scheithauer, Nils~R.},
       title={On the classification of automorphic products and generalized {Kac}-{Moody} algebras},
        date={2006},
        ISSN={0020-9910},
     journal={Invent. Math.},
      volume={164},
      number={3},
       pages={641\ndash 678},
}

\bib{serre1971cohomology}{misc}{
      author={Serre, Jean-Pierre},
       title={Cohomology of discrete groups},
    language={French},
         how={S{\'e}m. {Bourbaki} 23e ann{\'e}e (1970/1971), {Exp}. {No}. 399, {Lect}. {Notes} {Math}. 244, 337-350 (1971).},
        date={1971},
         url={https://eudml.org/doc/109803},
}

\bib{shiga1988representation}{article}{
      author={Shiga, Hironori},
       title={On the representation of the {Picard} modular function by $\theta$ constants. {I}-{II}},
        date={1988},
        ISSN={0034-5318},
     journal={Publ. Res. Inst. Math. Sci.},
      volume={24},
      number={3},
       pages={311\ndash 360},
}

\bib{Siegel1935}{article}{
      author={Siegel, Carl},
       title={Über die classenzahl quadratischer zahlkörper},
    language={ger},
        date={1935},
     journal={Acta Arith.},
      volume={1},
      number={1},
       pages={83\ndash 86},
         url={http://eudml.org/doc/205054},
}

\bib{siegel1945some}{article}{
      author={Siegel, Carl~Ludwig},
       title={Some remarks on discontinuous groups},
        date={1945},
        ISSN={0003486X, 19398980},
     journal={Ann. Math.},
      volume={46},
      number={4},
       pages={708\ndash 718},
         url={http://www.jstor.org/stable/1969206},
}

\bib{shephard1954finite}{article}{
      author={Shephard, Geoffrey~C},
      author={Todd, John~A},
       title={Finite unitary reflection groups},
        date={1954},
     journal={Can. J. Math},
      volume={6},
       pages={274\ndash 304},
}

\bib{stuken2017nonfreeness}{article}{
      author={Stuken, Ekaterina},
       title={Free algebras of {H}ilbert automorphic forms},
        date={2019},
     journal={Funct. Anal. Appl.},
      volume={53},
       pages={37\ndash 50},
}

\bib{stuken2022nonfreeness}{article}{
      author={Stuken, Ekaterina},
       title={Nonfreeness of some algebras of hermitian modular forms},
        date={2022},
     journal={arXiv preprint arXiv:2209.09303},
}

\bib{tai1982kodaira}{article}{
      author={Tai, Yung-Sheng},
       title={On the {K}odaira dimension of the moduli space of abelian varieties},
        date={1982},
     journal={Invent. Math.},
      volume={68},
      number={3},
       pages={425\ndash 439},
}

\bib{Thilmany2019lattices}{article}{
      author={Thilmany, Fran{\c{c}}ois},
       title={Lattices of minimal covolume in {{$\mathrm{SL}_n(\mathbb{R})$}}},
        date={2019},
        ISSN={0024-6115},
     journal={Proc. Lond. Math. Soc. (3)},
      volume={118},
      number={1},
       pages={78\ndash 102},
}

\bib{tai1982structure}{article}{
      author={Tai, Yung-Sheng},
      author={Resnikoff, Howard~Leonard},
       title={On the structure of a graded ring of automorphic forms on the 2-dimensional complex ball. {II}},
        date={1982},
     journal={Math. Ann.},
      volume={258},
       pages={367\ndash 382},
}

\bib{vinberg2010some}{article}{
      author={Vinberg, Ernest~Borisovich},
       title={Some free algebras of automorphic forms on symmetric domains of type {IV}},
        date={2010},
     journal={Transform. Groups},
      volume={15},
       pages={701\ndash 741},
}

\bib{vinberg2013algebra}{article}{
      author={Vinberg, Ernest~Borisovich},
       title={On the algebra of {S}iegel modular forms of genus $2$},
        date={2013},
     journal={Trans. Mosc. Math. Soc.},
      volume={74},
       pages={1\ndash 13},
}

\bib{vinberg2018some}{article}{
      author={Vinberg, Ernest~Borisovich},
       title={On some free algebras of automorphic forms},
        date={2018},
     journal={Funct. Anal. Appl.},
      volume={52},
      number={4},
       pages={270\ndash 289},
}

\bib{vinberg1989invariant}{article}{
      author={Vinberg, Ernest~Borisovich},
      author={Popov, Vladimir~Leonidovich},
       title={Invariant theory},
        date={1989},
     journal={Sovrem. Probl. Mat. Fund. Naprav.},
      volume={55},
       pages={137\ndash 309},
}

\bib{vinberg2017criterion}{article}{
      author={Vinberg, Ernest~Borisovich},
      author={Schwarzman, Osip~Vladimirovich},
       title={A criterion of smoothness at infinity for an arithmetic quotient of the future tube},
        date={2017},
     journal={Funktsional. Anal. i Prilozhen},
      volume={51},
      number={1},
       pages={40\ndash 59},
}

\bib{wakatsuki2018dimensions}{article}{
      author={Wakatsuki, Satoshi},
       title={The dimensions of spaces of {S}iegel cusp forms of general degree},
        date={2018},
     journal={Adv. Math.},
      volume={340},
       pages={1012\ndash 1066},
}

\bib{wall1970classification}{article}{
      author={Wall, CTC},
       title={On the classification of hermitian forms. {I}. rings of algebraic integers},
        date={1970},
     journal={Compos. Math.},
      volume={22},
      number={4},
       pages={425\ndash 451},
}

\bib{Wang2021classification}{article}{
      author={Wang, Haowu},
       title={The classification of free algebras of orthogonal modular forms},
        date={2021},
        ISSN={0010-437X},
     journal={Compos. Math.},
      volume={157},
      number={9},
       pages={2026\ndash 2045},
}

\bib{wang2021some}{article}{
      author={Wang, Haowu},
       title={On some free algebras of orthogonal modular forms {II}},
        date={2021},
     journal={Res. Number Theory},
      volume={7},
       pages={1\ndash 21},
}

\bib{wang2023classificationreflective}{article}{
      author={Wang, Haowu},
       title={On the classification of reflective modular forms},
        date={2023},
     journal={arXiv preprint arXiv:2301.12606},
}

\bib{wang2024reflective}{article}{
      author={Wang, Haowu},
       title={The classification of 2-reflective modular forms},
        date={2024},
        ISSN={1435-9855},
     journal={J. Eur. Math. Soc. (JEMS)},
      volume={26},
      number={1},
       pages={111\ndash 151},
}

\bib{williams2021two}{article}{
      author={Williams, Brandon},
       title={Two graded rings of {Hermitian} modular forms},
        date={2021},
        ISSN={0025-5858},
     journal={Abh. Math. Semin. Univ. Hamb.},
      volume={91},
      number={2},
       pages={257\ndash 285},
}

\bib{wang2020some}{article}{
      author={Wang, Haowu},
      author={Williams, Brandon},
       title={On some free algebras of orthogonal modular forms},
        date={2020},
     journal={Adv. Math.},
      volume={373},
       pages={107332},
}

\bib{WW2021}{article}{
      author={Wang, Haowu},
      author={Williams, Brandon},
       title={Free algebras of modular forms on ball quotients},
        date={2021},
     journal={arXiv preprint arXiv:2105.14892},
}

\bib{wang2023graded}{article}{
      author={Wang, Haowu},
      author={Williams, Brandon},
       title={Graded rings of {H}ermitian modular forms with singularities},
        date={2023},
     journal={Manuscr. Math.},
      volume={170},
      number={1},
       pages={283\ndash 311},
}

\end{biblist}
\end{bibdiv}

\end{document}